\documentclass[10pt,a4paper]{amsart}		

\usepackage[applemac]{inputenc}					
\usepackage{xcolor,ulem,textcomp}

\usepackage{lmodern}						
\usepackage[english]{babel}					
\usepackage[weather]{ifsym}	
\usepackage{mathrsfs}				
\usepackage{cancel}
\usepackage{hyperref,enumitem,soul}

\hypersetup{
	colorlinks,
	linkcolor={red!50!black},
	citecolor={blue!50!black},
	urlcolor={blue!80!black}
}

\usepackage{tikz}

\usepackage{amsmath,amsfonts,amssymb,amsthm}

\usepackage[left=3cm,right=3cm,top=3cm,bottom=3cm]{geometry}

\usepackage{mathtools}						

\theoremstyle{plain}
\newtheorem{proposition}{Proposition}[section]
\newtheorem{theorem}[proposition]{Theorem}

\newtheorem{corollary}[proposition]{Corollary}
\newtheorem{lemma}[proposition]{Lemma}

\newtheorem{definition}[proposition]{Definition}
\newtheorem{remark}[proposition]{Remark}

\numberwithin{equation}{section}

\DeclareMathOperator{\Vol}{Vol}	

\newcommand{\R}{\mathbb{R}}					
\newcommand{\C}{\mathbb{C}}					
\newcommand{\N}{\mathbb{N}}					

\newcommand{\eps}{{\varepsilon}}

\newcommand{\x}{\theta}

\newcommand{\red}{\color{red}}

\def\norm#1#2{\|#1\|_{#2}}

\def\refer#1{~\ref{#1}}
\def\refeq#1{~(\ref{#1})}
\def\ccite#1{~\cite{#1}}

\def\suite#1#2#3{(#1_{#2})_{#2\in {#3}}}

\def\inte#1{
\displaystyle\mathop{#1\kern0pt}^\circ }

\let\al=\alpha

\let\g=\gamma

\let\d=\delta
\let\e=\varepsilon

\let\lam=\lambda

\let\s=\sigma

\let\f=\phi
\let\p=\psi

\let\D=\Delta

\let\Lam=\Lambda
\let\S=\Sigma

\let\wt=\widetilde
\let\wh=\widehat

\def\cC{{\mathcal C}}
\def\cD{{\mathcal D}}

\def\cF{{\mathcal F}}

\def\cL{{\mathcal L}}

\def\cS{{\mathcal S}}

\def\cU{{\mathcal U}}

\def\cW{{\mathcal W}}
\def\cX{{\mathcal X}}

\def\S{{\mathop{\mathbb  S\kern 0pt}\nolimits}}

\def\virgp{\raise 2pt\hbox{,}}
\def\cdotpv{\raise 2pt\hbox{;}}
\def\eqdefa{\buildrel\hbox{\footnotesize def}\over =}

\def\Sp{\mathop{\rm Sp}\nolimits}
\def\C{\mathop{\mathbb C\kern 0pt}\nolimits}
\def\EE{\mathop{{\mathbb E \kern 0pt}}\nolimits}
\def\K{\mathop{\mathbb K\kern 0pt}\nolimits}
\def\Q{\mathop{\mathbb Q\kern 0pt}\nolimits}
\def\R{{\mathop{\mathbb R\kern 0pt}\nolimits}}
\def\SS{\mathop{\mathbb S\kern 0pt}\nolimits}
\def\ZZ{\mathop{\mathbb Z\kern 0pt}\nolimits}
\def\TT{\mathop{\mathbb T\kern 0pt}\nolimits}
\def\P{\mathop{\mathbb P\kern 0pt}\nolimits}
\def \H{{\mathop {\mathbb H\kern 0pt}\nolimits}}
\newcommand{\ds}{\displaystyle}

\def\p{\partial}

\newcommand{\beq}{\begin{equation}}
\newcommand{\eeq}{\end{equation}}
\newcommand{\ben}{\begin{eqnarray}}
\newcommand{\een}{\end{eqnarray}}
\newcommand{\beno}{\begin{eqnarray*}}
\newcommand{\eeno}{\end{eqnarray*}}
\newcommand{\bqs}{\begin{equation*}}
\newcommand{\eqs}{\end{equation*}}
\newcommand{\andf}{\quad\hbox{and}\quad}
\newcommand{\with}{\quad\hbox{with}\quad}

\def\equivH#1 {\buildrel\hbox{\tiny {$#1$}}\over \equiv}
\def\simH#1 {\buildrel\hbox{\footnotesize {$#1$}}\over \sim}

\title[Engel group]
{Spectral  summability for  the quartic  oscillator \\ with applications to  the Engel group}

\date{\today}
\author[H. Bahouri]{Hajer Bahouri}
\address[H. Bahouri]
{CNRS  \&  Sorbonne Universit\'e  \\
 Laboratoire Jacques-Louis Lions (LJLL) UMR  7598 \\
4, Place Jussieu\\
75005 Paris, France.}
\email{hajer.bahouri@sorbonne-universite.fr}
\author[D. Barilari]{Davide Barilari}\address[D. Barilari]%
{Dipartimento di Matematica "Tullio Levi-Civita" \\
      Universit{\`a} di Padova \\
 Via Trieste 63 \\
 Padova, Italy}
\email{barilari@math.unipd.it}
\author[I. Gallagher]{Isabelle Gallagher}
\address[I. Gallagher]%
{DMA UMR 8553, \'Ecole normale sup\'erieure, CNRS, PSL Research University, 75005 Paris 
 \\
and UFR de math\'ematiques, Universit\'e  Paris Cit\'e,  75013 Paris, France.}
\email{gallagher@math.ens.fr}
\author[M. L\'eautaud]{Matthieu L\'eautaud}
\address[M. L\'eautaud]
{Laboratoire de Math\'ematiques d'Orsay, UMR 8628, Universit\'e Paris-Saclay, CNRS, Universit\'e Paris-Saclay, B\^atiment 307, 91405 Orsay Cedex France}
\email{matthieu.leautaud@math.u-psud.fr}
\begin{document}
\setstcolor{red}

\maketitle

\setcounter{tocdepth}{2}

\begin{abstract}
In this article, we investigate spectral properties of the sublaplacian $-\Delta_{G}$ on the Engel group, which is the main example of a Carnot group of step~3. We develop a new approach to the Fourier analysis on the Engel group in terms of a frequency set. 

This enables us to give fine estimates on the convolution kernel satisfying $F(-\Delta_{G})u=u\star k_{F}$, for suitable scalar functions $F$, and in turn to obtain proofs of classical functional embeddings, via Fourier techniques. 

This analysis requires a summability property on the spectrum of the quartic oscillator, which we obtain by means of semiclassical techniques and which is of independent interest. 
\end{abstract}
\tableofcontents



\section{Introduction and statement of the main results} 

\subsection{The Engel group} Analysis on Lie groups is nowadays a rich and independent research field, with   applications and intersections with many fields of mathematics, from PDEs to geometry~\cite{faraut,Hebisch}. A particular class of such groups receiving increasing attention is given by the so-called Carnot groups. These groups, playing the role of local models in sub-Riemannian geometry as the Euclidean $\R^{d}$ does for Riemannian geometry, are nilpotent Lie groups diffeomorphic to $\R^{d}$ and homogeneous with respect to a family of dilations, which are automorphisms of the Lie algebra. The most renowned examples of such groups are Heisenberg groups, which are Carnot groups of step $2$.

 The Lie algebra $\mathfrak{g}$ of a Carnot group admits a stratification $\mathfrak{g}=\oplus_{i=1}^{s} \mathfrak{g}_{i}$ where the grading   is compatible with the dilations, and the first layer $\mathfrak{g}_{1}$ is Lie bracket generating, i.e., the smallest Lie algebra containing $\mathfrak{g}_{1}$ is $\mathfrak{g}$ itself, satisfying $\mathfrak{g}_{i+1}=[\mathfrak{g}_{1},\mathfrak{g}_{i}]$ with the convention $\mathfrak{g}_{s+1}=0$. The (smallest) integer $s$ satisfying this property is then called the step of the Carnot group.

While the analysis on Carnot groups of step 2 is now quite well understood (see for instance the monographs\ccite{bcdbookh, bfg2, fisher, follandstein, Hall, stein2, thangavelu, taylor1} and the references therein), much less can be said for Carnot groups of higher steps. The main example of a Carnot group of step 3, which is the focus of the present paper, is the so-called Engel group.
 
The Engel group $G$ is a nilpotent 4-dimensional Lie group which is connected and simply connected, and whose Lie algebra $\mathfrak{g}$ satisfies the following decomposition  
$$\mathfrak{g}=\mathfrak{g}_{1}\oplus \mathfrak{g}_{2}\oplus \mathfrak{g}_{3},
$$
with 
$$\dim \mathfrak{g}_{1}=2\,,\qquad \mathfrak{g}_{2}\eqdefa[\mathfrak{g}_{1},\mathfrak{g}_{1}]\,,\qquad \mathfrak{g}_{3}\eqdefa[\mathfrak{g}_{1},\mathfrak{g}_{2}]\, .
$$This group is described in  detail in Section~\ref{basic}.  Let us   recall that 
it is homogeneous of degree~$Q = 7$, and   
one can define a  sub-Riemannian distance  on~$G$, and the sub-Riemannian gradient~$\nabla_G f$. One can then consider the sublaplacian operator $$
\Delta_G  f\eqdefa\mathrm{div}(\nabla_G f)\, ,
$$
where $\mathrm{div}$ denotes the divergence with respect to the Haar measure on~$G$.  

\subsection{Spectral analysis of the sublaplacian} One of our goals in this paper is to provide an effective analysis of the spectral  properties of the sublaplacian~$\Delta_G$, having in mind the following version   of the classical spectral theorem for selfadjoint operators (see~\cite[Theorem~VIII.4 p.~260]{Reed-Simon-1} or~\cite[Th\'eor\`eme~4.5 p.~117]{Lewin}).
 \begin{theorem}
 \label{t:genspect}
{\sl Let $(A,D(A))$ be a selfadjoint operator on a separable Hilbert space $H$. Then, there exists: 
\begin{itemize}
\item a Borel set $B \subset \R^d$, $d \geq 1$, endowed with a locally finite Borel measure $\mathfrak{m}$ on $B$, 
\item a locally bounded real valued function $a \in L^\infty_{loc}(B;\R,d\mathfrak{m})$,
\item an isometry $U :H \to L^2(B,d\mathfrak{m})$, 
\end{itemize}
such that $UAU^*=M_a$, the operator of multiplication by the function $a$, with $UD(A)=D(M_a)$.
 }\end{theorem}
 
 Any such Borel set $B$ can be seen as a ``frequency space'' for the operator $A$, and the unitary operator $U:H \to L^2(B,d\mathfrak{m})$ can be understood as a ``Fourier transform'' adapted to the operator~$A$. 
 
 \smallskip

 Let us discuss the spirit of this theorem on two main examples: the Euclidean space $\R^d$ (which is a commutative Lie group) and the Heisenberg group~$\H^d$, which is a non commutative, nilpotent Lie group, whose Lie algebra~$ \mathfrak{h}$ satisfies~$\mathfrak{h}=\mathfrak{h}_{1}\oplus \mathfrak{h}_{2}$ with~$ \mathfrak{h}_{2}=[\mathfrak{h}_{1},\mathfrak{h}_{1}]$ and $[\mathfrak{h}_{1},\mathfrak{h}_{2}]=0$.
\begin{enumerate}
\item  The Euclidean space~$\R^d$. In this case for the (opposite of the) classical Laplace operator~$A=-\Delta$ the standard Fourier identity
\begin{equation}\label{eq:frr1}
\cF(-\Delta u)(\xi)=|\xi|^{2}\cF(u)(\xi),\qquad \xi \in \R^{d} \, ,
\end{equation}
 can be reinterpreted in terms of Theorem~\ref{t:genspect}
by choosing~$H= L^2(\R^d)$ and $B = \R^d$ endowed with the Lebesgue measure, where $U=\cF$ is the Fourier transform and~$a(\xi) =|\xi|^2$.   
 
 \smallskip
 
 \item The Heisenberg group~$\H^d$. In this case the (opposite of the) sublaplacian~$-\Delta_{\H^d}$ becomes after non commutative Fourier transform a rescaled version 
 of the  harmonic oscillator acting on $L^{2}(\R^d)$
\begin{equation}
\label{e:quadratic}
 \mathsf{H}  \eqdefa  - \Delta_z   + |\lambda|^2|z|^2 \, , \quad z \in \R^d, \,  \, \lambda \in \R^*\, ,
\end{equation} 
whose spectrum is given by the set $\big\{  |\lambda|(2|m|+d),\lambda \in \R^*, m \in \N^d\}$. 
A formulation of Theorem \ref{t:genspect} for the operator~$A=-\Delta_{\H^d}$ can be given for $H=L^2(\H^d, dw)$ and $U$ a Heisenberg Fourier transform~$\cF_{\H^d}$. An explicit description has been provided in~\cite{bcdh} where~$B = \N^{d}\times \N^d\times \R^*$ (writing elements of~$B$ as triplets~$\hat w =(n, m,\lambda)$) is the space of frequencies endowed with the measure~$\delta (n)\delta (m) |\lambda|^d d\lambda$, where~$\delta (n)\delta (m)$ denotes the counting measure on $\N^{2d}$. The function $a$ is given by $a(n, m,\lambda) =  |\lambda|(2|m|+d)$.

Notice that this translates into the analogue to the Fourier identity \eqref{eq:frr1} for $A=-\Delta_{\H^d}$ as follows  
\begin{equation}\label{eq:frr2}
\cF_{\H^d}(-\Delta_{\H^d}u)(n,m,\lambda)= |\lambda|(2|m|+d)\cF_{\H^d}(u)(n,m,\lambda)\,.
\end{equation}
We highlight that the function~$a$ in the case of the Heisenberg group does not depend on~$n$: this is related to the fact that the operator~$-\Delta_{\H^d}$ diagonalizes the Hermite basis of eigenfunctions of~$\mathsf{H}$.

\end{enumerate}
\medskip
 In this paper our first aim is to identify a family of objects $(B,\mathfrak{m},a,U)$ as presented in Theorem~\ref{t:genspect}
for the sublaplacian~$ \Delta_G$ on the Engel group, acting on the Hilbert space $L^2(G, dx)$ (as we shall see in Section~\ref{basic}, the Haar measure on~$G$ can be identified with the Lebesgue measure $dx$ in suitable coordinates), that is useful in applications. 
In the case of the Engel group it is known that the  non commutative Fourier transform exchanges   (the opposite of) the sublaplacian~$-\Delta_G$ with an operator acting on~$L^2(\R)$, which turns out to be the (family of conveniently rescaled) quartic oscillator 
\begin{equation}
\label{e:quartic}
\mathsf{P}_\mu \eqdefa  - \frac{d^2}{d\theta^2}  + \Big( \frac{\theta^2}{2} - \mu \Big)^2
\, , \quad \theta \in \R \, , 
\end{equation} 
where~$ \mu \in \R$ is a real parameter (see\refeq{eq: rellap}-\eqref{eq: oscop} below).    To the best of our knowledge, this operator appeared for the first time in relation   with hypoelliptic operators in the paper by Pham The Lai and Robert\ccite{Pham-Robert} (but had already been studied before that in relation to quantum mechanics).
Since then it has received enduring attention and has been extensively studied under different perspectives:  more references  on the spectral theory for $\mathsf{P}_\mu$ are provided in Section~\ref{poisson}.

 \medskip In order to 
state our first  result, we need to recall that~$\mathsf{P}_\mu$  can be endowed  with the  domain 
\begin{align}
\label{e:def-domain-P}
D(\mathsf{P}_\mu) = \Big\{ u \in L^2(\R) \, , \quad -   \frac{d^2}{d\theta^2} + \Big( \frac{\theta^2}{2} - \mu \Big)^2 u  \in L^2(\R) \Big\} \, ,
\end{align}
and that its spectrum consists in countably many real eigenvalues $\{\mathsf{E}_{m}(\mu)\}_{m\in \N}$ of multiplicity~1 and satisfying
$$  0< \mathsf{E}_0(\mu) < \mathsf{E}_1(\mu)< \cdots < \mathsf{E}_m(\mu)< \mathsf{E}_{m+1}(\mu) \to + \infty \, . $$
We also define, for~$(\nu,\lambda) \in  \R \times \R^*$, the rescaled eigenvalues
 \begin{equation}\label{defEmnulambda}
 E_m(\nu,\lambda)\eqdefa |\lambda|^\frac23 \mathsf{E}_m \Big(\frac \nu{|\lambda|^\frac43}\Big) \,  \cdotp
 \end{equation}
   \begin{theorem}
\label{t:key}
{\sl Set $\widehat{G}\eqdefa \N\times \N\times \R \times \R^*$, write elements of~$\widehat{G}$ as $\wh{x} = (n,m,\nu,\lambda)  $, and define a measure on~$\widehat{G}$ by~$d\wh x \eqdefa\delta(n)\delta(m)d\nu d\lambda$, recalling that~$\delta(n)\delta(m)$ is the counting measure on~$\N^2$. 
 Then define on $\widehat{G}$ the function $$\wh{x}\longmapsto a(\widehat{x}) \eqdefa E_m(\nu, \lambda) \, .$$ 
 There exists a unitary operator $U : L^2(G,dx)\to L^2(\widehat{G}, d\wh{x})$ such that 
\beq\label{UDU*=a}
 U(-\Delta_G)U^*=M_a \, , \quad UD(-\Delta_G)=D(M_a) \, ,
\eeq
where $M_{a}$ denotes the operator given by multiplication by the function $a$.
}\end{theorem}
  The set~$\widehat{G}$ will be understood in the following as the frequency set of the Engel group. 
The operator~$U$ will be a Fourier transform~$\mathcal F$ (a function acting on elements~$\wh{x} = (n,m,\nu,\lambda) $ of~$ \widehat{G}$) which we construct explicitly, see~\eqref{eq: defFf}--\eqref{eq: W}.   We recall that the Fourier transform on non commutative Lie groups   is classically defined as a family of bounded operators on some Hilbert space. That notion of  Fourier transform  enjoys the  same properties (in terms of operators) as the Fourier transform on~$\R^d$, such as inversion and Fourier-Plancherel formulae.   As we show in Section\refer{Fourier}, our new approach  is equivalent to the classical Engel Fourier transform which as already mentioned above converts~$-\Delta_G$ into~$\mathsf{P}_\mu$, up to scaling. 
 The   Fourier transform given by Theorem\refer{t:key} consists in considering  the classical Engel Fourier transform (as a family of operators) by means of its coefficients in the basis of the eigenfunctions  of~$\mathsf{P}_\mu$, and as we shall see, the difficulty of the classical Engel Fourier transform  is shifted to the frequency set~$\widehat{G}$ which turns out to be discrete with respect to a part of the variables  and continuous with respect to the other part, and thus it cannot be identified with~$G$ as in the Euclidean setting; in Section\refer{topofrequency}, we attempt to equip it  with a topology  which takes into account the basic principles of the Fourier transform,  namely that   regularity of functions on~$G$ is converted into decay of the Fourier transform on~$\widehat{G}$. 
 Contrary to the Heisenberg setting investigated in\ccite{bcdh}, the study of topological properties of~$\widehat{G}$ such as determining its completion,   computing the measure on its unit   sphere and providing the spectral decomposition of $-\D_G$  prove to be a challenging task   requiring    refined spectral analysis of~$\mathsf{P}_\mu$.

 As in the case of the Heisenberg group described above, notice that the function~$a$ involved in Theorem~\ref{t:key} does not depend on~$n$. Again this is related to the fact that the Engel sublaplacian is diagonal  on the basis of eigenfunctions of~$\mathsf{P}_\mu$.

\medskip

The explicit representation of the Fourier transform in terms of a basis allows us to make effective computations. 
 Once the Fourier transform is well understood, it is natural to try to recover  via this tool well-known functional inequalities on~$G$, such as Sobolev embeddings, and to analyze evolution equations involving the sublaplacian. This requires estimating quantities involving the operator~$F(-\Delta_G)$, for  suitable functions $F$ defined on $
  \R_+$.  
For such~$F$ there holds  (for all~$u$ in the Schwartz space $\cS(G)$ which is nothing else than the Schwartz space $\cS(\R^4)$) 
\begin{equation}\label{defFDelta}
\cF (F(-\Delta_G) u)(\hat x)=F(E_m(\nu, \lambda)) \cF ( u)(\hat x)\, ,
\end{equation}
hence  we are led to computing integrals of the form  $$
 \sum_{m\in \N}\int_{\R\times \R^{*}} F(E_m(\nu, \lambda))\, d\lambda d\nu\, ,
 $$
which can be rewritten as $\displaystyle \int_{\widehat{G}} F(a(\wh{x})) \delta_{n,m}d\widehat{x} $, 
 for $\wh{x} = (n,m,\nu,\lambda) \in \widehat{G}$ and $a(\widehat{x}) = E_m(\nu, \lambda)$.
 Contrary to the  Euclidean case, or to the harmonic oscillator~(\ref{e:quadratic}) appearing in the Heisenberg group, the eigenvalues of~$\mathsf{P}_\mu$ are not explicitly known. 
However
the spectral analysis we conduct in this paper leads to the following theorem, which   enables  us to generalize    (with some technicalities) to the Engel group   many results in real analysis,  such as classical   functional inequalities  and  Bernstein inequalities. 
 Our second main result is indeed the following.  
\begin{theorem}
\label{t:key2}
{\sl With the notation of Theorem~\ref{t:key}, the following result holds.
  For all measurable functions $F:  \R_+ \to \R$, the function~$F\circ a $ belongs to~$L^1(\widehat{G}, \delta_{n,m} d\widehat{x})$ if and only if~$F \in L^1(\R_+, r^{5/2} dr)$, and there holds
\beq\label{magic formula introduction} 
\int_{\widehat{G}} F(a(\wh{x})) \delta_{n,m}  d\widehat{x} = \left(   \sum_{m\in \N} \int_{\R}  \frac{3}{\mathsf{E}_m(\mu)^{\frac72}} d\mu \right) \int_0^\infty r^{5/2} F(r) dr\, . 
\eeq
 Moreover
 \begin{equation}
\label{e:summability-eigenvaluesbis}
\sum_{m\in \N} \int_{\R}  \frac{1}{\mathsf{E}_m(\mu)^{\gamma}} d\mu <\infty \quad \Longleftrightarrow \quad\gamma>2\, .
\end{equation}
}\end{theorem}
To better understand the content of the previous theorem, let us reconsider our two basic examples:
\begin{enumerate}
\item 
In the Euclidean space~$\R^d$ we have $a(\widehat x)=a(\xi)=|\xi|^{2}$ with $d\widehat x=d\xi$ so that  the left-hand side of \eqref{magic formula introduction}  can be computed using spherical coordinates
\beq \label{magic formula introduction Rd}
\displaystyle \int_{B} F(a(\wh{x})) d\widehat{x}  =\displaystyle \int_{\R^d} F \big(|\xi|^2\big) d\xi =|{\mathbb S}^{d-1}| \int_{\R_+ } F(r)r^{\frac{d-2}2} dr  \, .
\eeq

\item On the Heisenberg group $\H^d$, we have $a(\widehat x)= a(n, m,\lambda) =  |\lambda|(2|m|+d)$ and $\delta_{n,m}d\widehat x =\delta(m)|\lambda|^d d\lambda $ so that 
\beq \begin{aligned} \label{magic formula introduction Hd}
\displaystyle \int_{B} F(a(\wh{x})) \delta_{n,m} d\widehat{x} & = \sum_{m \in \N^d} \int_{\R^*}  F\big(|\lambda|(2|m|+d)\big) |\lambda|^d d\lambda 
\\ & =\Big(\sum_{m \in \N^d} \frac2{(2|m|+d)^{d+1}}\Big) \int_{\R_+}r^{d }F(r) \, dr\, , \end{aligned}
\eeq
where the last equality follows from a change of variables.  Note that the power of~$r$ is~$d = (Q-2)/2$ where~$Q = 2d+2$ is the homogeneous dimension of~$\H^d$, so 
the summability conditions have the same homogeneity on~$\R^d$ and on~$\H^d$, and are exactly the same as that given by~(\ref{magic formula introduction}) since~$Q=7$ for the Engel group.
\end{enumerate}

\begin{remark}[On the explicit constants] 
{\sl It is interesting to notice that the prefactor in the right-hand side of~{\rm(\ref{magic formula introduction Hd})}
$$\mathsf{C}_{\H^d}\eqdefa\sum_{m \in \N^d} \frac2{(2|m|+d)^{d+1}}$$
corresponds  to the measure of the dual unit sphere of the Heisenberg group (see~\cite{Muller}), when one endows the dual of the Heisenberg group by its natural metric structure and volume form. The same property appears also in the Euclidean case, by Formula~{\rm(\ref{magic formula introduction Rd})} and recalling that the dual of the  Euclidean space coincides in fact with the space itself.}
\end{remark}
The next proposition  shows that this is not a coincidence and is valid also in the Engel group, suggesting perhaps a more general pattern.
\begin{proposition}\label{prop:CG} {\sl The constant $\mathsf{C}_{G}$ defined by
\begin{equation}\label{eq:CG}
\mathsf{C}_{G}\eqdefa \sum_{m\in \N} \int_{\R}  \frac{3}{\mathsf{E}_m(\mu)^{\frac72}} d\mu 
 \end{equation}
 coincides with the volume  of the {\it dual unit sphere} of the  Engel group, when endowed with its natural metric structure and volume form.}
  \end{proposition}
  We refer the reader to Section\refer{spectral} for precise definitions and the proof of this result; we stress here  that the fact that the integral \eqref{eq:CG} is finite is part of the statement of Theorem~\ref{t:key2}.
\subsection{Functions of the sublaplacian and their convolution kernel}
Let us go further in the analysis of operators of the type~$F(-\Delta_G)$ by considering their convolution kernel. To this end we  define  the space of functions of polynomial growth~${\mathcal O}^\infty(\R_+)\eqdefa\displaystyle\bigcup_{m\in \N}{\mathcal O}^\infty_{m}(\R_+)$, with
$$
F\in {\mathcal O}^\infty_{m}(\R_+)
\Longleftrightarrow   \langle \cdot\rangle^{-m} F \in  L^\infty(\R_+) \, , $$
and recall the following rather classical result (which holds for any left-invariant sublaplacian on a Carnot group; the proof is recalled in Section~\ref{Proofkey} for the sake of completeness).
\begin{proposition}
\label{p:function-calculus-abstract}
{\sl For any $F \in {\mathcal O}^\infty(\R_+)$, the operator $F(-\Delta_G) : \mathcal{S}(G) \to L^2(G)$ is well-defined (via spectral theory) and there is $k_F \in \mathcal{S}'(G)$ such that 
\begin{equation}
\label{e:F-convolution}
F(-\Delta_G) u = u \star k_{F} \, ,  \quad \text{for all } u \in {\mathcal S}(G ) \, ,
\end{equation}
where $\star$ is the natural convolution product on $G$ (see~{\rm(\ref{conv})} and~\eqref{e:def-convolS} below). }
\end{proposition}

The Fourier transform defined in the present article allows to generalize the set of functions $F$ for which the functional calculus is well-defined and to characterize the regularity of the kernel in terms of properties of $F$.
We define the space~${\mathcal O}^{1,s}(\R_+)\eqdefa\displaystyle\bigcup_{m\in \N}{\mathcal O}^{1,s}_{m}(\R_+)$, where
$$
F\in {\mathcal O}^{1,s}_{m}(\R_+)
\Longleftrightarrow   \langle \cdot\rangle^{-m} F \in  L^1(\R_+, r^{s} dr) \, , $$
where ${\mathcal O}^{1,s}_{m}(\R_+)$ is endowed with the norm
$$
\|F\|_{{\mathcal O}^{1,s}_{m}(\R_+)}\eqdefa\big\| \langle \cdot\rangle^{-m} F\big\|_{L^1(\R_+,r^s dr)} = \int_{\R_+} r^{s}\langle r\rangle^{-m}| F(r) |dr \, .
$$
The space~${\mathcal O}^{1,s}(\R_+)$ is endowed with   the associated Fr\'echet topology.

\begin{theorem}
\label{t:function-calculus-L1}
 {\sl Assume $F\in {\mathcal O}^{1,5/2}(\R_+)$. For any function $u \in \mathcal{S}(G)$,    one can define in~$L^\infty(G)$ the inverse Fourier transform of the function $(n,m,\nu, \lambda)\mapsto F(E_m(\nu, \lambda)) \cF ( u)(n,m,\nu, \lambda)$  and the operator $F(-\Delta_G) : \mathcal{S}(G) \to L^\infty(G)$ is thus well-defined by 
$$
F(-\Delta_G)u  \eqdefa \cF^{-1} \big( F(E_m(\nu, \lambda)) \cF ( u)(\hat x) \big) \, .
$$
 Moreover, there is a distribution~$k_{F}$  in~${\mathcal S}'(G )$ such that~\eqref{e:F-convolution} is satisfied and the map 
$$
\begin{aligned}
{\mathcal O}^{1,5/2}_m(\R_+)& \longrightarrow \mathcal{S}'(G) \\
  F& \longmapsto k_F   
\end{aligned}$$
is continuous.}
\end{theorem}  
\begin{remark}  {\sl  For a function $\theta:\widehat{G}\to \C$, sufficient conditions to have a well-defined inverse Fourier transform are given in Proposition \ref{p:ifH}.}
\end{remark}
We next give a sufficient condition for continuity/boundedness of the kernel $k_F$ in terms of properties of $F$.
\begin{theorem}
\label{t:kernel}
{\sl If $F \in L^1(\R_+,r^{5/2} dr)$, the kernel $k_F$ given by Theorem~{\rm\ref{t:function-calculus-L1}} belongs to~$C^0(G) \cap L^\infty(G)$ (where the distribution $k_{F}$ is identified with a function using the Haar measure of $G$) and there holds
 $$\begin{aligned}
\|k_F\|_{L^\infty(G)} & \leq  (2\pi)^{-3} \left(   \sum_{m\in \N} \int_{\R}  \frac{3}{\mathsf{E}_m(\mu)^{\frac72}} d\mu \right) \int_0^\infty r^{5/2} |F(r)| dr \quad \mbox {and} \\
 k_F(0) & =  (2\pi)^{-3} \left(   \sum_{m\in \N} \int_{\R}  \frac{3}{\mathsf{E}_m(\mu)^{\frac72}} d\mu \right) \int_0^\infty r^{5/2} F(r) dr\, .
 \end{aligned}$$}
\end{theorem}

\begin{remark}
{\sl The proof shows that  if~$F \in L^1(\R_+,r^{5/2+\ell} dr)$ for some $\ell \in \N$, then $\Delta_G^{\ell} k_{F}$ belongs to~$C^0(G) \cap L^\infty(G)$ and there holds
 $$\begin{aligned}
\| \Delta_G^{\ell}  k_F\|_{L^\infty(G)} & \leq  (2\pi)^{-3} \left(   \sum_{m\in \N} \int_{\R}  \frac{3}{\mathsf{E}_m(\mu)^{\frac72}} d\mu \right) \int_0^\infty r^{5/2+\ell} |F(r)| dr  .
 \end{aligned}$$}
\end{remark}

\medbreak

Remark that it was shown in\ccite{Folland, hul} that the kernel $k_{F}$ belongs to $\cS(G)$ in the case when~$F$  belongs to~$ \cS(\R_+)$. Here, the assumption on $F$ is much weaker, and the regularity we deduce is accordingly weaker. However, the regularity of $k_F$ described in Theorem~\ref{t:kernel} is the appropriate one for many applications in analysis.
 As   will be discussed in Section\refer{pointfunct},  the Fourier transform of the kernel~$k_F$ satisfies $U (k_{F} )(\hat x)=F(E_m(\nu, \lambda)) \delta_{m,n}$ (see\refeq{kernel} below). 
 Taking for instance~$F(r) = F_t(r)=\exp(-tr)$,~$t>0$, one recovers the fact that the {\it  Engel heat kernel}  at the origin satisfies    (for further details see\refeq{defhkernel1})
$$k_{F}(0) =  \frac {\mathsf{C}_{G} \Gamma (\frac Q 2)} {(2\pi)^{3} t^{\frac Q 2}} \, \virgp$$
where $\Gamma$ denotes the Gamma function and $\mathsf{C}_{G}$  the volume  of the {\it dual unit sphere} defined by\refeq{eq:CG}.

\smallskip
 Finally, let us also recall that the investigation of necessary and sufficient conditions for operators of the form $F(-\Delta)$ to be
bounded on $L^{p}$ for some $p\neq 2$ in terms of properties of the spectral multiplier $F$
is a traditional and very active area of research of harmonic analysis. For related results when working with sublaplacians $\Delta_{G}$ we refer the reader to \cite{martini1,martini2,martini3} and references therein.

\subsection{Layout}  In Section~\ref{poisson} we establish  the summability property\refeq{e:summability-eigenvaluesbis}, thanks to a semiclassical analysis of the operator~$\mathsf{P}_\mu$. This property is at the core of our work, but is independent from the rest of this text, and its proof  can be skipped altogether by a reader interested only in   applications to the Engel group. 

In Section\refer{basic} we recall some basic facts about the Engel group. 
 Section~\ref{Fourier} is dedicated to the  study of the Fourier transform on the Engel group.    In Paragraph~\ref{defsta}, we give  a brief description of    the standard Engel Fourier transform, using irreducible representations.  Then in Paragraph~\ref{pointfunct}, we start the proof of Theorem\refer{t:key} by revisiting this Fourier transform  in the spirit of~\cite{bcdh} providing a new, equivalent, functional point of view which consists in looking at the Fourier transform as a complex valued  function that is defined on the  {\it frequency set} $\widehat{G}$. This is based on the spectral analysis of the quartic oscillator~$\mathsf{P}_\mu$. Granted with this new approach, we   furnish in Paragraph\refer{spectral} a convenient expression for the spectral decomposition of~$-\Delta_G$. In Paragraph\refer{Proofkey} we achieve  the proof of  Theorem\refer{t:key}. 

Section~\ref{s:app} is dedicated to some applications of our Fourier decomposition. In Paragraph~\ref{Sobembed}, taking advantage of~\eqref{magic formula introduction}, we recover  many functional inequalities due to Folland\ccite{Folland}   using the approach based on   the Engel Fourier transform, while in Paragraph~\ref{Bernstein}, we define the notion of spectral localization and establish Bernstein inequalities as well as their inverse version.  In Paragraph\refer{heat} we highlight once again the efficiency of~\eqref{magic formula introduction} by analyzing the heat   kernel on the Engel group.  

Finally in Section\refer{topofrequency}, we endow the  frequency set~$\widehat{G}$ with a distance linked to its Lie structure.  
  We deal in two appendices with several complements for the sake of completeness, as we strive for a  self-contained paper. 
In Appendix~\ref{irrep},  we recall  the construction of the irreducible representations.
In Appendix\refer{ap2}, we relate the spectral theory of a family  of operators $P_{\nu,\lambda}$ with that of our reference quartic oscillator~$\mathsf{P}_\mu$ and recall basic facts of spectral theory.

\medskip  
To avoid heaviness, all along this article~$C$ will denote  a  positive  constant   which may vary from line to line.    We also use $f\lesssim g$  to
denote an estimate of the form $f\leq C g$.

\medskip

{\bf Acknowledgements. } $ $The authors wish to warmly thank Jacques Faraut   and Georges Skandalis  for enlightening discussions and remarks around the  Kirillov theory. They also thank  Bernard Helffer for providing references concerning the operator~$\mathsf{P}_\mu$, and Albert Cohen for questions related to Lemma~6.2. 

\section{Summability of eigenvalues of the operator $\mathsf{P}_\mu$}
\label{poisson}
In this section, we study some spectral properties of the operator $(\mathsf{P}_\mu, D(\mathsf{P}_\mu))$ introduced in~\eqref{e:quartic}--\eqref{e:def-domain-P}.\\
This operator appears in different contexts:
\begin{itemize}
\item in quantum mechanics, see Simon \cite{Sim} (see also~\cite{Reed-Simon-4});
\item  in the study of irreducible representations of certain nilpotent Lie groups (see for example \cite{Pham-Robert,He:80} with focus on analytic hypoellipticity of hypoelliptic operators, see also~\cite{Chatzakou}), which is the application we have in mind here (see also~\cite{DVL} for the analysis of a related sublaplacian);
\item in the study of Schr\"odinger operators with magnetic fields on compact manifolds and in superconductivity (see e.g. Montgomery~\cite{Mon95} or \cite{HM,Pan-Kwek,HK:10}).
 \end{itemize}
Properties of the first eigenvalue of $\mathsf{P}_\mu$ have also been investigated in~\cite{Helffer2}.

\smallskip

Here, motivated by the study of functions of the   Engel sublaplacian $\Delta_{G}$, the ultimate goal of the section is to prove~(\ref{e:summability-eigenvaluesbis}).
Before this, we recall basic spectral properties of this operator.
The following proposition serves as a definition for the eigenvalue~$\mathsf{E}_m(\mu)$ and the associated eigenfunction~$\varphi_m^\mu$ for $m\in \N$, and a proof is given in Appendix~\ref{anspctm} for the convenience of the reader. 
\begin{proposition}
\label{p:def-E-fimu}
{\sl For any $\mu \in\R$, the following statements hold true.
The operator $(\mathsf{P}_\mu, D(\mathsf{P}_\mu))$ is selfadjoint on $L^2(\R)$, with compact resolvent. Its spectrum consists in countably many real eigenvalues with finite multiplicities, accumulating only at $+\infty$. Moreover, 
\begin{enumerate}
\item \label{eigenvalue-ordering} all eigenvalues are simple and positive, and we may thus write $\Sp(\mathsf{P}_\mu) = \{\mathsf{E}_m(\mu), m \in \N\}$ with 
\begin{align*}
& 0< \mathsf{E}_0(\mu) < \mathsf{E}_1(\mu)< \cdots < \mathsf{E}_m(\mu)< \mathsf{E}_{m+1}(\mu) \to + \infty \,  , \\
& \dim \ker (\mathsf{P}_\mu-\mathsf{E}_m(\mu)) = 1\,  ,
\end{align*}
\item \label{regularity} all eigenfunctions are real-analytic and decay exponentially fast at infinity (as well as all their derivatives),
\item \label{parity} for all $m\in \N$, functions in $\ker (\mathsf{P}_\mu-\mathsf{E}_m(\mu))$ have the parity of $m$,
\item \label{def-psim} for all $m\in \N$, there is a unique function $\varphi_m^\mu$ in $\ker (\mathsf{P}_\mu-\mathsf{E}_m(\mu))$ such that 
$$
\varphi_m^\mu \text{ is real-valued}, \quad \|\varphi_m^\mu\|_{L^2(\R)} = 1 , \quad\varphi_m^\mu(0)>0 \text{ if } m\text{ is even} ,\quad \frac{d}{d\theta} \varphi_m^\mu(0)>0 \text{ if } m\text{ is odd} , 
$$
\item the family $\big(\varphi_m^\mu\big)_{m\in \N}$ forms a Hilbert basis of $L^2(\R)$.
 \end{enumerate}}
\end{proposition}
The aim of this section is now to prove~(\ref{e:summability-eigenvaluesbis}), that is to say, discuss (in terms of the parameter~$\gamma$) convergence of  
\begin{equation*}
\mathcal{I}_\gamma \eqdefa \sum_{k\in \N} \int_{\R}  \frac{1}{\mathsf{E}_k(\mu)^{\gamma}} d\mu \,   , \quad \text{ for }\gamma >0 \,  .
\end{equation*}
We rewrite the integral in consideration as
$$
\mathcal{I}_\gamma  =\int_{\R \times \N}  \frac{1}{\mathsf{E}_k(\mu)^{\gamma}} d\mu d\delta(k)\,   ,
$$
where $d\delta(k)$ is the counting measure on $\N$.
As will appear in the proof of~(\ref{e:summability-eigenvaluesbis}) in Theorem~\ref{t:key2}, there are three main regimes to be considered in the analysis of the eigenvalues $\mathsf{E}_k(\mu)$ in terms of~$(\mu,k) \in \R \times \N$. In each of these regimes, we will use a semiclassical reformulation of the problem with a single (small) parameter $h$ related either to a power of $k^{-1}$ or a power of $\mu^{-1}$. The three main regimes in the study of convergence of $\mathcal{I}_\gamma$ are as follows:
\begin{enumerate}
\item \label{i:regime-classical} $|\mu| \lesssim 1$ or $|\mu| \ll \sqrt{\mathsf{E}_k(\mu)}$ (classical and perturbative classical regime) that is, $\mu$ bounded or going to $\pm\infty$ not too fast, 
\item  \label{i:regime-harmonic} $\mu \to -\infty$ and $\mathsf{E}_k(\mu) \lesssim \mu^2$ (Semiclassical Harmonic oscillator/single well regime),
\item \label{i:regime-double} $\mu \to +\infty$ and $\mathsf{E}_k(\mu) \lesssim \mu^2$  (Semiclassical double well regime).
\end{enumerate}
We shall then split $\mathcal{I}_\gamma$ accordingly, for some $\eps>0$ (small) and $\mu_0>0$ (large) as  
\begin{align}
\mathcal{I}_\gamma=\mathcal{I}_\gamma^-(\eps,\mu_0)+\mathcal{I}_\gamma^0(\eps, \mu_0)+\mathcal{I}_\gamma^+(\eps, \mu_0),\quad \text{ with } \\
\label{e:decomposition-Ek}
\mathcal{I}_\gamma^\bullet(\eps, \mu_0) \eqdefa  \int_{\mathcal{E}^\bullet(\eps, \mu_0)}  \frac{d\mu d\delta(k)}{\mathsf{E}_k(\mu)^\gamma} , \quad \text{ with } \bullet = -,0,+ , \quad \text{ and } \\
\mathcal{E}^0(\eps,\mu_0)\eqdefa \{(\mu,k)\in\R\times\N,|\mu|\leq\mu_0\text{ or } |\mu|^2 \leq \eps^2 \mathsf{E}_k(\mu)\} , \\
\mathcal{E}^-(\eps,\mu_0)\eqdefa \{(\mu,k)\in\R\times\N, \mu \leq - \mu_0\text{ and } |\mu|^2 \geq \eps^2 \mathsf{E}_k(\mu)\} , \\
\mathcal{E}^+(\eps,\mu_0)\eqdefa \{(\mu,k)\in\R\times\N, \mu \geq \mu_0\text{ and } |\mu|^2 \geq \eps^2 \mathsf{E}_k(\mu)\} .
\end{align}

In each region, we shall make use of scaling operators in $\R$.
We define for $\alpha >0$ the following unitary (dilation) operator  
 \begin{equation}
 \label{e:Talpha}
\begin{array}{rcl}
T_\alpha : L^2(\R)& \to & L^2(\R)\,  , \\
u(x)& \mapsto &\alpha^{\frac{1}{2}}u(\alpha x) \,  ,
\end{array}
\end{equation}
having adjoint/inverse $T_\alpha^*=T_\alpha^{-1}=T_{\alpha^{-1}}$.

Note that the (necessary and sufficient) condition $\gamma >2$ for having $\mathcal{I}_\gamma<\infty$, as stated in Theorem~\ref{t:key2}, comes from the third (double well) region, see Corollary~\ref{c:troisieme-zone-cv} below.
\subsection{Classical and perturbative classical regime~\ref{i:regime-classical}}
In the regime~\ref{i:regime-classical} we consider $\mathsf{P}_\mu$ as a ``small'' perturbation of the quartic oscillator $\mathsf{P}_0=- \frac{d^2}{d\theta^2}+\frac{\theta^4}{4}$ and look at the asymptotics $k \to + \infty$.
\begin{lemma}
\label{l:asympt-k-grand}
{\sl There exist two continuous nondecreasing functions $\Gamma_\pm : \R_+\to \R_+$ such that~$\Gamma_\pm(\eps_0)>0$ for $\eps_0 >0$ and $\Gamma_\pm(0)=0$ satisfying the following statements.

 For all $\eps >0$ and $\mu \in \R$ such that $|\mu \Lambda^{-1/2}| \leq \eps$, we have 
\begin{align}
\label{e:counting-k}
 \Lambda^{3/4} \left( \Vol_1 - \Gamma_- (\eps)+ o(1) \right)  \leq 2\pi\sharp \{k \in \N , \mathsf{E}_k(\mu) \leq \Lambda \} \leq \Lambda^{3/4} \left( \Vol_1 + \Gamma_+(\eps)+ o(1) \right)\,  ,
\end{align}
 as $\Lambda \to +\infty$, where 
$$
\Vol_1  \eqdefa \int_{\{ \xi^2 + \frac{\x^4}{4} \leq 1 \}} d\x \, d\xi  >0 \, .
$$ 
 For all $\eps, \mu_0 >0$ and for all $(\mu,k) \in \R\times \N$ such that $|\mu \mathsf{E}_{k}(\mu)^{-1/2}| \leq \eps$ or $|\mu| \leq \mu_0$, we have
\begin{align*}
\mathsf{E}_{k}(\mu) \geq  \left(\frac{2\pi}{\Vol_1 + \Gamma_+(\eps)} k\right)^{4/3} \left(1+ o(1) \right) ,  \quad \text{as }k \to +\infty \, .
\end{align*}
 For all $\eps >0$ such that $ \Vol_1 - \Gamma_- (\eps)>0$ (that is, $\eps$ small enough), for all $(\mu,k) \in \R\times \N$ such that~$|\mu \mathsf{E}_{k}(\mu)^{-1/2}| \leq \eps$, 
 \begin{align*}
\mathsf{E}_{k}(\mu)\leq  \left(\frac{2\pi}{\Vol_1 - \Gamma_-(\eps)} k\right)^{4/3} \left(1+ o(1) \right) ,  \quad \text{as }k \to +\infty \, .
\end{align*}}
\end{lemma}

 In the end, the first term in the decomposition~\eqref{e:decomposition-Ek} can be estimated as follows.
 \begin{corollary} 
 \label{c:cor-classial-regime}
{\sl There is $\eps_0>0$ such that for all $\eps \in (0,\eps_0)$ and for all $\mu_0>0$, $\mathcal{I}_\gamma^0(\eps, \mu_0)<+\infty$ if and only if~$\gamma>\frac{5}{4}$. } 
 \end{corollary}
 \begin{proof}
Fix $\eps_0>0$ such that $\Vol_1 - \Gamma_\pm(\eps_0)>0$ (and take any $\mu_0>0$). For all $\eps \in (0,\eps_0)$, there is~$k_0\in \N$ such that if $\frac{|\mu|}{\sqrt{\mathsf{E}_{k}(\mu)}} \leq \eps$ or $ |\mu| \leq\mu_0$, then $|\mu| \leq \max (C_\eps k^{4/3}, \mu_0)$ and $C_\eps^{-1}k^{4/3}\leq\mathsf{E}_{k}(\mu)\leq C_\eps k^{4/3}$ for all $k \geq k_0$. As a consequence, using that $\mathsf{E}_k(\mu) >0$ on $\R$ together with Lemma~\ref{l:asympt-k-grand} for all fixed $k \in \N$, we have 
\begin{align*}
\mathcal{I}_\gamma^0(\eps, \mu_0) & = \int_{ |\mu| \leq \eps \sqrt{\mathsf{E}_k(\mu)} \text{ or } |\mu| \leq\mu_0}  \frac{d\mu d\delta(k)}{\mathsf{E}_k(\mu)^{\gamma}} 
 \leq \int_{|\mu| \leq\mu_0, k \le k_0}  \frac{d\mu d\delta(k)}{\mathsf{E}_k(\mu)^{\gamma}} 
+\int_{|\mu| \leq C_\eps\sqrt{ k^{4/3}}, k \geq k_0} \frac{d\mu d\delta(k)}{(C_\eps^{-1} k^{4/3})^{\gamma}} \\
& \leq C(\mu_0, k_0) +  \tilde{C}_\eps \sum_{k \in \N^*} \frac{k^{2/3}}{(k^{4/3})^{\gamma}} = C(\mu_0, k_0)+ \tilde{C}_\eps \sum_{k \in \N^*} \frac{1}{(k^{2/3})^{2\gamma-1}} <\infty ,
\end{align*}
if and only if $\gamma >  \frac{1}{2} (1+\frac{3}{2}) = \frac{5}{4}$. Finally, Lemma~\ref{l:asympt-k-grand} also yields the associated lower bound $$\int_{ |\mu| \leq \eps \sqrt{\mathsf{E}_k(\mu)}}\frac{d\mu d\delta(k)}{\mathsf{E}_k(\mu)^{\gamma}}  \geq  \tilde{c}_\eps \sum_{k \in \N}\frac{1}{(k^{2/3})^{2\gamma-1}}\, \cdotp$$
Corollary~\ref{c:cor-classial-regime}
 is proved.
\end{proof}
 \begin{proof}[Proof of Lemma~{\rm\ref{l:asympt-k-grand}}]
We use   the dilation operator $T_\alpha$ defined in~\eqref{e:Talpha} to recast the problem as~$k \to+\infty$ in a semiclassical setup.
We have  
\begin{align*}
\mathsf{P}_\mu \psi = \mathsf{E}_k(\mu) \psi \quad \Longleftrightarrow \quad T_\alpha \mathsf{P}_\mu T_{\alpha^{-1}} T_\alpha \psi= \mathsf{E}_k(\mu) T_\alpha \psi  \, ,
\end{align*}
where
$$
T_\alpha \mathsf{P}_\mu T_{\alpha^{-1}} = - \alpha^{-2} \frac{d^2}{d\theta^2} + \left( \alpha^2 \frac{\theta^2}{2} - \mu \right)^2 \,  .
$$
We deduce that 
\begin{align*}
\mathsf{P}_\mu \psi = \mathsf{E}_k(\mu) \psi \quad 
& \Longleftrightarrow\quad  \left[ - \alpha^{-6} \frac{d^2}{d\theta^2} + \left(  \frac{\theta^2}{2} - \mu\alpha^{-2} \right)^2 \right]T_\alpha \psi= \alpha^{-4}\mathsf{E}_k(\mu) T_\alpha \psi  \, .
\end{align*} 
We now choose $h \eqdefa  \alpha^{-3}$, i.e. $\alpha= h^{-1/3}$, so that 
\begin{align*}
\mathsf{P}_\mu \psi = \mathsf{E}_k(\mu) \psi 
\quad \Longleftrightarrow \quad P(h) (T_{h^{-1/3}} \psi)= h^{4/3}  \mathsf{E}_k(\mu) (T_{h^{-1/3}}  \psi ) \,  ,
\end{align*}
with 
$$
P(h) =  - h^2 \frac{d^2}{d\theta^2} + \left(  \frac{\theta^2}{2} - \mu h^{2/3}\right)^2 \,  . 
$$
As a consequence of the simplicity of the spectrum, we obtain that $\Sp(P(h)) = \{h^{4/3}  \mathsf{E}_k(\mu) , k \in \N\}$, and that these eigenvalues are sorted increasingly.
We may now apply Proposition~\ref{p:semiclass-perturb} for $L=1$, yielding existence of the functions $\Gamma_\pm$ satisfying the following statement. For all $\eps >0$ and $\mu \in \R$ such that $|\mu h^{2/3}| \leq \eps$, we have 
$$
 \Vol_1- \Gamma_-(\eps)+ o(1) 
 \leq (2\pi h) \sharp  \{k \in \N , h^{4/3}  \mathsf{E}_k(\mu) \leq 1 \}
\leq  \Vol_1 + \Gamma_+(\eps)+ o(1) 
$$
as $h \to 0^+$.
Setting $\Lambda = h^{-4/3} \to +\infty$, i.e. $h =\Lambda^{-3/4}$, we have obtained that for all $\eps >0$ and~$\mu \in \R$ such that~$|\mu \Lambda^{-1/2}| \leq \eps$, \eqref{e:counting-k} is satisfied.
 
Finally, we deduce an asymptotics of the $\mathsf{E}_k(\mu)$ from an asymptotics of the counting function.
We recall from Proposition~\ref{p:def-E-psi} that the eigenvalues are ordered increasingly, $\mathsf{E}_{k}(\mu) < \mathsf{E}_{k+1}(\mu)$ and we set
$$
k(\Lambda)\eqdefa \sup\{k\in\N,\mathsf{E}_k(\mu)\leq\Lambda\}  \, .
$$
By definition (forgetting temporarily the dependence in $\mu$) and simplicity of eigenvalues, we thus have 
$$
\mathsf{E}_{k(\Lambda)} \leq \Lambda < \mathsf{E}_{k(\Lambda)+1}, \quad \text{and} \quad \sharp \{k \in \N , \mathsf{E}_k(\mu) \leq \Lambda \} = k(\Lambda)+1 \,  .
$$ 
As a consequence,~\eqref{e:counting-k} rewrites, 
\begin{align*}
 \Lambda^{3/4} \left( \Vol_1 - \Gamma_- (\eps)+ o(1) \right)   \leq 2\pi(k(\Lambda)+1) \leq \Lambda^{3/4} \left( \Vol_1 + \Gamma_+(\eps)+ o(1) \right) ,  \quad \text{as }\Lambda \to +\infty  \, ,
\end{align*}
whence,  assuming $|\mu \mathsf{E}_{k(\Lambda)}^{-1/2}| \leq \eps$  (which then implies $|\mu \Lambda^{-1/2}| \leq \eps$),
\begin{align*}
\mathsf{E}_{k(\Lambda)}^{3/4} \left( \Vol_1 - \Gamma_- (\eps)+ o(1) \right)  \leq 2\pi(k(\Lambda)+1) \leq \mathsf{E}_{k(\Lambda)+1} ^{3/4} \left( \Vol_1 + \Gamma_+(\eps)+ o(1) \right) ,  \quad \text{as }\Lambda \to +\infty \,  .
\end{align*}
Since $\Lambda \mapsto k(\Lambda)$ is nondecreasing, tending to infinity and onto from $\R_+\to \N$, we deduce that, assuming $|\mu \mathsf{E}_{k}(\mu)^{-1/2}| \leq \eps$, 
\begin{align*}
\mathsf{E}_{k}^{3/4} \left( \Vol_1 - \Gamma_- (\eps)+ o(1) \right)  \leq 2\pi k  \leq \mathsf{E}_{k} ^{3/4} \left( \Vol_1 + \Gamma_+(\eps)+ o(1) \right) ,  \quad \text{as }k \to +\infty \,  ,
\end{align*}
which implies the last two statements.
\end{proof}

\subsection{Semiclassical Harmonic oscillator/single well regime~\ref{i:regime-harmonic}}
In this region, we only need rather loose properties.
First notice that, for all $m\in \N$,
\begin{equation} \label{eq:neqbeh} \text{for all }  \mu <0\, ,  \quad \mathsf{E}_m(\mu) \geq |\mu|^2 \,.\end{equation}  Indeed starting from the eigenvalue equation
$$
\left(  - \frac{d^2}{d\theta^2}  + \left( \frac{\theta^2}{2} - \mu \right)^2 \right)\varphi_m^\mu  =  
 \mathsf{E}_m(\mu)\varphi_m^\mu  \, ,
$$
and  taking the inner product with $\varphi_m^\mu$ yields
$$
\|\partial_\theta \varphi_m^\mu\|_{L^2(\R)}^2 - \mu\| \theta \varphi_m^\mu \|_{L^2}^2 + \frac 1 4 \| \theta^2 \varphi_m^\mu \|_{L^2}^2+   \mu^2   =  
 \mathsf{E}_m(\mu)  \, ,
$$
which implies the bound\refeq{eq:neqbeh}. We further need a Weyl-type asymptotics.

\begin{lemma}
\label{l:estim-mu--infty}
{\sl   For all $L>0$, one has
\begin{align}
\label{e:asympt-harmonic}
\sharp \{k \in \N , \mathsf{E}_k(\mu) \leq L |\mu|^2 \} = (2\pi)^{-1} |\mu|^{3/2} \big( \Vol_L + o(1) \big) ,  \quad \text{as }\mu \to - \infty \,  ,
\end{align}
 where $\Vol_L$ is defined by   
\begin{align}
\label{e:volL-bis}
\Vol_L   &\eqdefa \int_{\{p(x,\xi)\leq L\}} dx d\xi  \,  , \quad \text{ with }p(x,\xi) = \xi^2 + V(x) \,  , \quad V(x) =  \left(  \frac{x^2}{2}  + 1 \right)^2  \,  ,\\
& =  \int_{x_-(L)}^{x_+(L)}  \sqrt{L  -\left(  \frac{x^2}{2}  + 1 \right)^2 } dx ,\quad \text{ with } \left(  \frac{x_\pm(L)^2}{2}  + 1 \right)^2 = L , \quad \text{ for } L>1 \,   .
\nonumber
\end{align}}
\end{lemma}

In the end, this is helpful to estimate the second term in the decomposition~\eqref{e:decomposition-Ek}. 
\begin{corollary}\label{reg2}
{\sl For any $\eps>0$, there is $\tilde{\mu}_0>0$ such that for all $\mu_0 \geq \tilde{\mu}_0$, $\mathcal{I}_\gamma^-(\eps, \mu_0)<+\infty$ if $\gamma >\frac54$.}
\end{corollary}
\begin{proof}[Proof of Corollary~{\rm\ref{reg2}} from Lemma~{\rm\ref{l:estim-mu--infty}}]
The set of integration is $\mu \leq - \eps \sqrt{\mathsf{E}_k(\mu)}$ and $\mu \leq -\mu_0<0$. Then, the integral can be estimated as: given $\eps>0$, there is $\mu_0 = \mu_0(\eps)>0$ such that for all~$\mu \leq -\mu_0$, the number of eigenvalues in $\mu \leq - \eps \sqrt{\mathsf{E}_k(\mu)}$ is according to~\eqref{e:asympt-harmonic}
$$
\sharp \{k \in \N , \mathsf{E}_k(\mu) \leq \frac{1}{\eps^2} |\mu|^2 \} \leq  (2\pi)^{-1} |\mu|^{3/2} \big( \Vol_{\eps^{-2}} + 1 \big) ,  \quad \text{ for }\mu \leq -\mu_0(\eps)  \, .
$$
Since $\mathsf{E}_k(\mu) \geq |\mu|^2$ for $\mu<-\mu_0(\eps)$, there is $C_\eps$ such that
\begin{align*}
\mathcal{I}_\gamma^-(\eps, \mu_0)&  = \int_{\mu \leq - \eps \sqrt{\mathsf{E}_k(\mu)}, |\mu| \geq \mu_0}  \frac{d\mu d\delta(k)}{\mathsf{E}_k(\mu)^\gamma}
  \leq \int_{\mu \leq - \eps \sqrt{\mathsf{E}_k(\mu)}, |\mu| \geq \mu_0}  \frac{d\mu d\delta(k)}{|\mu|^{2\gamma}} \\
& \leq \int_{\mu <- \mu_0}  (2\pi)^{-1} |\mu|^{3/2} \big( \Vol_{\eps^{-2}} + 1 \big)  \frac{d\mu}{|\mu|^{2\gamma}}
 \leq C_\eps<\infty \,  ,
\end{align*}
as soon as $2\gamma-\frac32>1$ that is $\gamma >\frac54\cdotp$
\end{proof}
\begin{proof}[Proof of Lemma~{\rm\ref{l:estim-mu--infty}}]
We set $\eta \eqdefa -\mu$ and study for $\eta \to +\infty$, 
$$
\mathsf{P}_{-\eta} =  - \frac{d^2}{d\theta^2}  + \left( \frac{\theta^2}{2} + \eta \right)^2 \,  .
$$
We choose $\alpha = \sqrt{\eta}$ in the rescaling  
$$
T_\alpha \mathsf{P}_{-\eta} T_{\alpha^{-1}} = - \alpha^{-2} \frac{d^2}{d\theta^2} + \left( \alpha^2 \frac{\theta^2}{2} + \eta \right)^2 = - \alpha^{-2} \frac{d^2}{d\theta^2} +\alpha^4 \left( \frac{\theta^2}{2} + 1 \right)^2 \,  .
$$
As a consequence
\begin{align*}
\mathsf{P}_{-\eta} \psi = \mathsf{E}_k(-\eta) \psi & \Longleftrightarrow T_\alpha \mathsf{P}_{-\eta} T_{\alpha^{-1}} T_\alpha \psi= \mathsf{E}_k(-\eta) T_\alpha \psi  \\
& \Longleftrightarrow \left[ - \alpha^{-6} \frac{d^2}{d\theta^2} + \left(  \frac{\theta^2}{2}  + 1 \right)^2 \right]T_\alpha \psi= \alpha^{-4}\mathsf{E}_k(-\eta) T_\alpha \psi \,   .
\end{align*}
We set $h= \alpha^{-3} = \eta^{-3/2}$ and obtain 
\begin{align}
\label{e:P-switch}
\mathsf{P}_{-\eta} \psi = \mathsf{E}_k(-\eta) \psi & \Longleftrightarrow  \left[ -h^2 \frac{d^2}{d\theta^2} + \left(  \frac{\theta^2}{2}  + 1 \right)^2 \right]T_\alpha \psi= h^{4/3}\mathsf{E}_k(-\eta) T_\alpha \psi  \,  .
\end{align}
The Weyl Law~\eqref{e:weyl} applied to the operator $-h^2 \frac{d^2}{d\theta^2} + \left(  \frac{\theta^2}{2}  + 1 \right)^2$ then reads: for all $L>0$ fixed,
\begin{align*}
\sharp \{k \in \N , h^{4/3}\mathsf{E}_k(\mu)  \leq L \} = (2\pi h)^{-1} \big( \Vol_L + o(1) \big) ,  \quad \text{as }h \to 0 \,  ,
\end{align*}
with $\Vol_L$ defined by~\eqref{e:volL-bis}. Recalling that $h= \eta^{-3/2}$ then yields 
\begin{align*}
\sharp \{k \in \N , \eta^{-2} \mathsf{E}_k(-\eta)  \leq L \} = (2\pi)^{-1} \eta^{3/2}\big( \Vol_L + o(1) \big) ,  \quad \text{as } \eta\to +\infty  \, ,
\end{align*}
and then we write back $\mu=-\eta$ to obtain~\eqref{e:asympt-harmonic}.
\end{proof}

\subsection{Semiclassical double well regime~\ref{i:regime-double}}
We want to estimate the last term in the decomposition~\eqref{e:decomposition-Ek}, namely $\mathcal{I}^+_\gamma(\eps,\mu_0)$.
To this aim, we study for $\mu \to +\infty$, the operator $\mathsf{P}_{\mu}$ and as above rescale it with $h= \alpha^{-3} = \mu^{-3/2}$, as
\begin{align}
\label{e:Pmu-Ph}
\mathsf{P}_\mu \psi = \mathsf{E}_k(\mu) \psi & \Longleftrightarrow  \left[ -h^2 \frac{d^2}{d\theta^2} + \left(  \frac{\theta^2}{2} - 1 \right)^2 \right]T_\alpha \psi= h^{4/3}\mathsf{E}_k(\mu) T_\alpha \psi =\mu^{-2}\mathsf{E}_k(\mu) T_\alpha \psi \,   .
\end{align}
We thus need to study the spectrum of the operator 
\begin{equation}
\label{e:def-semiclassic-2-well}
P_h \eqdefa  -h^2 \frac{d^2}{d\theta^2} + V(\theta), \quad \text{ with } \quad V(\theta) =  \left(  \frac{\theta^2}{2} - 1 \right)^2
\end{equation}
for energies $0\leq E\leq M$ for $M = \eps^{-2}$ (fixed by Corollary~\ref{c:cor-classial-regime}). Remark that this is a symmetric double well problem, which has been much studied~\cite{HS:84,Robert:87,HR:84,Helffer:booksemiclassic,DS:book}.

\bigskip
In this section, we only work in a semiclassical regime; we thus reformulate completely the problem with $h = \mu^{-3/2}$, and $E_k(h) = h^{4/3}\mathsf{E}_k(\mu)$ the $k$-th eigenvalue of $P_h$. In the integral~$\mathcal{I}^+_\gamma(\eps,\mu_0)$ in~\eqref{e:decomposition-Ek}, we set $\mu= h^{-2/3}$, $d\mu = \frac23 h^{-5/3} dh$, and obtain with $h_0 = \mu_0^{-3/2}$
\begin{equation}
\label{e:def-I}
\mathcal{I}^+_\gamma(\eps,\mu_0)= \frac23 \int_{E_k(h) \leq \eps^{-2}, 0<h \leq h_0}  \frac{h^{-5/3} dh d\delta(k)}{(h^{-4/3}E_k(h))^\gamma} =  \frac23 \int_{E_k(h) \leq \eps^{-2}, 0<h \leq h_0}  \frac{h^{(4\gamma-5)/3}}{E_k(h)^\gamma} dh d\delta(k) \,  .
\end{equation}
To prove convergence of this integral, we split the energy region $[0,M]$ where $M=\eps^{-2}$ is large into three different regions as 
$$
[0,M] = [0,\beta h] \cup [\beta h , \alpha] \cup [\alpha,M] \,  ,
$$
where $\beta> 0, \alpha \in (0,1), M >\alpha$ are fixed (independent of $h$). Concerning the energy window~$[\alpha,M]$, a counting estimate will be enough for our needs: the following is a rewriting of~\eqref{e:weyl} in the present context.
\begin{lemma}
\label{l:weyl-mou}
{\sl For $V(\theta) =  \left(  \frac{\theta^2}{2} - 1 \right)^2$  and $p(\theta, \xi) = \xi^2+V(\theta)$, for any $\alpha \leq M$, we have 
\begin{align*}
\sharp \{ j \in \N , E_j(h) \in [\alpha,M]\} & = (2\pi h)^{-1} \Big( \Vol p^{-1}([\alpha, M])+ o_{\alpha,M}(1) \Big) , \quad \text{ as } h \to 0^+ \,  .
\end{align*}}
\end{lemma}
Concerning the energy window $[\beta h ,\alpha]$, we shall use the following much more precise result from~\cite[p294-295]{HR:84}.
\begin{lemma}
\label{l:helffer-robert}
{\sl For $E \in [0,1)$, we set  
\begin{align}
\label{e:def-phi}
\Phi(E) \eqdefa \frac{1}{4\pi} \int_{p(\x,\xi)\leq E} d\x d\xi \,  , \quad \text{ with }\quad p(\x,\xi) = \xi^2 + \left(  \frac{\x^2}{2} - 1 \right)^2 .
\end{align}
There are $\beta >0$ and $N_\beta \in\N$ such that for all $\alpha <1$, there are $K,h_0>0$ such that for all~$h \in (0,h_0)$, there exist $N^\pm(h) \in \N$ with $|N^+(h)-N^-(h)| \leq 1$ and $N^\pm(h) \leq Kh^{-1}$, and two finite sequences~$E_j^\pm (h) \in  [\beta h , \alpha]$ for $j \in \{N_\beta, \cdots , N^\pm(h)\}$
with 
$$
\Phi \big( E_j^\pm (h) \big) = (j+1/2) h + O_{\alpha,\beta}(h^2), \quad  \text{ as }h \to 0^+ , 
$$
such that we have 
$$
 \Sp(P_h)\cap [\beta h , \alpha] = \bigcup_{j \in \{N_\beta, \cdots , N^+(h)\}}E_j^+ (h) \cup \bigcup_{j \in \{N_\beta , \cdots , N^-(h)\}}E_j^- (h)  \, . 
$$}
\end{lemma}
  Note that here, $E_j^\pm(h)$ is not the $j$-th eigenvalue of $P_h$. However, the $E_j^\pm(h)$'s exhaust the spectrum of~$P_h$ in the energy window $[\beta h , \alpha]$ as $h \to 0$.

Concerning the bottom of the spectrum, that is the energy window $[0,\beta h]$, we shall need a precise description of the eigenvalues~\cite{HS:84}. We recall that $V'(\x) = 2\x(\frac{\x^2}{2}-1)$ and $V''(\x) = 3 \x^2-2$. In particular at the two minima $V''(\pm \sqrt{2}) = 4$ and $\omega : = \sqrt{\frac{V''(\pm \sqrt{2})}{2}} = \sqrt{2}$. 
The following result is a consequence of~\cite{HS:84}, see also~\cite{Robert:87} and~\cite[pp~55--60]{Helffer:booksemiclassic}, and states that the low-lying eigenvalues are close to those of the Harmonic oscillator $-h^2\frac{d^2}{d\theta^2} +\omega^2\theta^2$.

\begin{lemma}[Bottom of the spectrum for the double well problem]
\label{l:helffer-sjostrand}
{\sl For all $\beta>0$, there are~$N_\beta \in \N , h_0>0$ such that 
$$
  \Sp(P_h)  \cap (-\infty , \beta h)    
  = \big\{E_n(h), n \in \{0,\dots , N_\beta\} \big\} , \quad \text{ uniformly for }  h \in (0,h_0) ,
$$
with $0< E_n(h)<E_{n+1}(h) < \beta h$ for all $n \in \{0,\dots , N_\beta-1\}$ and $h \in (0,h_0)$. Moreover, as $h \to 0^+$, we have 
\begin{enumerate}
\item $E_{2k}(h) =(2k+1)\omega h + O_\beta(h^2)$ is simple and associated to an even eigenfunction $\psi_{2k}(h)$,
\item  $E_{2k+1}(h) =(2k+1)\omega h + O_\beta(h^2)$ is simple and associated to an odd eigenfunction $\psi_{2k+1}(h)$.
 \end{enumerate}
}\end{lemma}

The regimes of Lemmata~\ref{l:helffer-robert} and~\ref{l:helffer-sjostrand} overlap (depending on the choice of the constant $\beta$ in these two statements) and we now check that the two asymptotics as $h \to 0^+$ coincide. 
\begin{lemma}
\label{e:info-function-phi}
{\sl The function $\Phi : [0,1)\to \R_+$ defined in~\eqref{e:def-phi} is continuous, of class $C^1$ on $(0,1)$, and we have, for $E \in (0,1)$,
$$\Phi'(E)  =  \frac1{2\pi} \int_{x_-(E)}^{x_+(E)}  \frac{1}{\sqrt{E  -V(x) }} dx >0  \, , \quad \text{ with } \quad  x_\pm(E) = \sqrt{2\pm2\sqrt{E}} \,  .$$
Moreover, the function $\Phi$ is differentiable at $E=0^+$ with $\Phi'(0^+)=(2\sqrt{2})^{-1}>0$.
}
\end{lemma}

\begin{remark}
A consequence of Lemma~{\rm\ref{e:info-function-phi}} is that the asymptotics given by Lemmata~{\rm\ref{l:helffer-robert}} and~{\rm\ref{l:helffer-sjostrand}} coincide in the regime in which they overlap. Indeed, for all eigenvalues belonging to both regimes, we have, using Lemmata~{\rm\ref{l:helffer-robert}} and~{\rm\ref{e:info-function-phi}}
$$
(j+1/2) h \sim \Phi \big( E_j^\pm (h) \big) \sim \Phi'(0^+) E_j^\pm (h)  = (2\sqrt{2})^{-1}E_j^\pm (h) \,   ,
$$
that is to say $E_j^\pm (h) \sim \sqrt{2} (2j+1) h$ as $h \to 0^+$, which is consistent with Lemma~\ref{l:helffer-sjostrand}.
\end{remark}
The proof of Lemma~\ref{e:info-function-phi} is postponed to the end of the section.
As a corollary of these four lemmata, we prove that $\mathcal{I}^+_\gamma(\eps,\mu_0)$ is finite.
\begin{corollary}
\label{c:troisieme-zone-cv} {\sl For all $M=\eps^{-2}>0$, there exists $\mu_0>0$ such that 
$\mathcal{I}^+_\gamma(\eps,\mu_0)< +\infty$ if~$\gamma>2$.
If~$\gamma \leq 2$, $\mathcal{I}^+_\gamma(\eps,\mu_0)= +\infty$ for all $\eps >0$ and $\mu_0>0$.}
\end{corollary}
\begin{proof}[Proof of Corollary~\ref{c:troisieme-zone-cv}]
We let $\beta$ be fixed by Lemma~\ref{l:helffer-robert}, fix $\alpha=\frac12$ in this lemma and split the integral in~\eqref{e:def-I} according to 
$$
\frac32 \mathcal{I}^+_\gamma(\eps,\mu_0)
= \mathcal{I}_1 + \mathcal{I}_2 + \mathcal{I}_3 \,  ,
$$
with, writing $h_0 = \mu_0^{-3/2}$ sufficiently small, 
$$
\mathcal{I}_1 = \int_{E_k(h) \in [0, \beta h], 0<h \leq h_0} , \quad \mathcal{I}_2 =  \int_{E_k(h) \in [\beta h, 1/2], 0<h \leq h_0}   , 
\quad \mathcal{I}_3 = \int_{E_k(h) \in [\frac12, M], 0<h \leq h_0} .
$$
Concerning $\mathcal{I}_3 $, we use Lemma~\ref{l:weyl-mou} (which applies for $h_0$ sufficiently small) to estimate 
\begin{align*}
 \mathcal{I}_3 & = \int_{E_k(h) \in [\frac12, M], 0<h \leq h_0}\frac{h^{(4\gamma-5)/3}}{E_k(h)^\gamma} dh d\delta(k) 
  \leq 2^\gamma \int_{E_k(h) \in [\frac12, M], 0<h \leq h_0} h^{(4\gamma-5)/3} dh d\delta(k) \\
 & \leq 2^\gamma \int_0^{h_0}h^{(4\gamma-5)/3} \  \sharp \{k \in \N, E_k(h)\in [1/2,M]\}  \ dh \leq C \int_0^{h_0}h^{(4\gamma-5)/3} \ h^{-1} \ dh <\infty  \, ,
\end{align*}
as soon as $(4\gamma-5)/3-1>-1$, that is $\gamma >\frac54$.

Concerning $\mathcal{I}_2$, we need an additional information on the function $\Phi$ in~\eqref{e:def-phi}. Lemma~\ref{e:info-function-phi} implies that $\Phi'$ is continuous (and positive) on $[0,1/2]$, and we may thus set   $M_\Phi \eqdefa \max_{[0,1/2]}\Phi'$. We therefore obtain 
\begin{align}
\label{e:m-phi}
 M_\Phi E \geq \Phi(E)   , \quad \text{ for all } E \in [0,1/2]  \, .
\end{align}
According to Lemma~\ref{l:helffer-robert}, we have 
$$
 \Sp(P_h)\cap [\beta h , 1/2] = \bigcup_{j \in \{N_\beta, \cdots , N^+(h)\}}E_j^+ (h) \cup \bigcup_{j \in \{N_\beta , \cdots , N^-(h)\}}E_j^- (h)  
$$
with, for all $j \in \{N_\beta , \cdots , N^\pm (h)\}$
$$
M_\Phi E_j^\pm (h)  \geq \Phi \big( E_j^\pm (h) \big) = (j+1/2) h + O_{\alpha,\beta}(h^2) \geq (j+\frac14) h ,\quad  \text{ for } h\leq h_0 \,  ,
$$
where the inequality comes from~\eqref{e:m-phi}.
As a consequence, we have 
\begin{align*}
 \mathcal{I}_2&  =  \int_{E_k(h) \in [\beta h, 1/2], 0<h \leq h_0} \frac{h^{(4\gamma-5)/3}}{E_k(h)^\gamma} dh d\delta(k)
  =  \int_0^{h_0} h^{(4\gamma-5)/3}  \sum_{\pm} \sum_{j =N_\beta}^{N^\pm (h)} \frac{1}{E_j^\pm (h)^\gamma} \ dh \\
 & \leq C \int_0^{h_0}  h^{(4\gamma-5)/3}   \sum_{j =N_\beta}^{\max\{N^-(h),N^+(h)\}} \frac{1}{\big( (4j+1)h \big)^{\gamma}} \ dh 
 \leq C \int_0^{h_0}  h^{(4\gamma-5)/3}   \ h^{-\gamma}  \sum_{j \leq K h^{-1}} \frac{1}{(4j+1)^{\gamma}} \ dh\\
 & \leq  
 C \int_0^{h_0} h^{(4\gamma-5)/3}   \ h^{-\gamma}  \ dh  <\infty ,
\end{align*}
as soon as $\gamma >1$ and $(4\gamma-5)/3 -\gamma >-1$, that is to say $\gamma >2$. 
The term $\mathcal{I}_3$ is estimated similarly but using Lemma~\ref{l:helffer-sjostrand}:
\begin{align*}
 \mathcal{I}_3&  =  \int_{E_k(h) \in [0,\beta h], 0<h \leq h_0} \frac{h^{(4\gamma-5)/3}}{E_k(h)^\gamma} dh d\delta(k) 
 \leq C \int_0^{h_0} h^{(4\gamma-5)/3}
\sum_{j =0}^{N_\beta} \frac{1}{((2j+1)h)^\gamma} \ dh  \\
& \leq C \int_0^{h_0} h^{(4\gamma-5)/3} h^{-\gamma}\ dh \,  ,
\end{align*}
which is finite as soon as $\gamma >2$ for the same reason.

To conclude the proof, we simply notice that Lemma~\ref{l:helffer-sjostrand} also implies
\begin{align*}
 \mathcal{I}_3&  =  \int_{E_k(h) \in [0,\beta h], 0<h \leq h_0} \frac{h^{(4\gamma-5)/3}}{E_k(h)^\gamma} dh d\delta(k) 
\geq c \int_0^{h_0} h^{(4\gamma-5)/3} h^{-\gamma}\ dh  = +\infty
\end{align*}
if $\gamma \leq 2$.
\end{proof}
For the proof to be complete, we now prove Lemma~\ref{e:info-function-phi}.
\begin{proof}[Proof of Lemma~{\rm\ref{e:info-function-phi}}]
We have $\Phi(0)=0$ and, for $E\in (0,1)$,
\begin{align*}
\Phi(E) & = \frac{1}{4\pi} \int_{\{p(\x,\xi)\leq E\}} d\x d\xi  = \frac{1}{2 \pi} \int_{\{\xi^2+V(\x)\leq E , \x>0 \}} d\x d\xi  = \frac{1}{\pi} \int_{\left\{0< \xi \leq\sqrt{E-V(\x)} , \x>0 \right\}} d\x d\xi 
\\
& = \frac{1}{\pi} \int_{\x_-(E)}^{\x_+(E)}  \sqrt{E  -V(\x) } d\x  \, , 
\end{align*}
with $V(\x) = \left(  \frac{\x^2}{2}  - 1 \right)^2$ and 
$$
 \x_\pm(E) \quad \text{are such that } V(\x_\pm(E))= E \text{ and } 0< \x_-(E) <\sqrt2< \x_+ (E)  \, ,
$$
that is, $\x_\pm(E) = \sqrt{2\pm2\sqrt{E}}$.
As a consequence, $\Phi$ is a continuous and strictly increasing function on $[0, 1)$ with $\Phi(0)=0$. The functions $E \mapsto \x_\pm(E)$ are smooth on $(0,1)$ and 
\begin{align*}
 \pi \Phi'(E) & = \x_+'(E) \sqrt{E  -V(\x_+(E)) } - \x_-'(E) \sqrt{E  -V(\x_+(E)) } + \int_{\x_-(E)}^{\x_+(E)}  \frac{1}{2\sqrt{E  -V(\x) }} d\x \\
& =  \frac12 \int_{\x_-(E)}^{\x_+(E)}  \frac{1}{\sqrt{E  -V(\x) }} d\x \,  .
\end{align*}
Moreover, recalling $\omega = \sqrt{\frac{V''(\sqrt{2})}{2}} =\sqrt{2}$, we have 
$$
V(\x)= \omega^2 (\x-\sqrt{2})^2  +O((\x-\sqrt{2})^3) , \quad \text{ as } \x \to \sqrt2  \, .
$$
Hence, setting $y = \frac{\omega}{\sqrt{E}}(\x-\sqrt{2})$, we have
\begin{align*}
\pi \Phi(E)   = \int_{\x_-(E)}^{\x_+(E)}  \sqrt{E  -V(\x) } d\x  & = \frac{\sqrt{E}}{\omega}\int_{\frac{\omega}{\sqrt{E}}(\x_-(E)-\sqrt{2})}^{\frac{\omega}{\sqrt{E}}(\x_+(E)-\sqrt{2})}  \sqrt{E  -V\left(\sqrt{2} +\frac{\sqrt{E}}{\omega}y \right) } dy ,
\end{align*}
with the following asymptotic properties as $E\to 0^+$
\begin{align*}
&1= \frac{V(\x_\pm(E))}{E} \leftarrow \frac{\omega^2}{E} (\x_\pm(E)-\sqrt{2})^2 ,\quad \text{whence } \frac{\omega}{\sqrt{E}}(\x_\pm(E)-\sqrt{2}) \to \pm1 ,\\
&V\left(\sqrt{2} +\frac{\sqrt{E}}{\omega}y \right)  = E y^2 + O(E^{3/2}) ,\quad\text{ uniformly for } y \text{ bounded}.
\end{align*}
As a consequence, as $E\to 0^+$ we have,
\begin{align*}
\pi \Phi(E)  =  \frac{\sqrt{E}}{\omega}\int_{-1}^{1}  \sqrt{E  -Ey^2 + O(E^{3/2})} dy  + o(E) = \frac{E}{\omega} \int_{-1}^{1}  \sqrt{1-y^2 } dy+ o(E) ,
\end{align*}
with $ \int_{-1}^{1}  \sqrt{1-y^2 } dy = \frac{\pi}{2}$. As a consequence, recalling that $\Phi(0)=0$, we deduce that $\Phi$ is differentiable at $E=0^+$ with 
$\Phi'(0^+) = \frac{1}{2\omega} = \frac{1}{2\sqrt{2}}$. Lemma~{\rm\ref{e:info-function-phi}} is proved.
\end{proof}

\bigskip

\section{Basic facts on the Engel group} \label{basic}

As recalled in the introduction, the Engel group $G$ is a nilpotent 4-dimensional Lie group which is connected and simply connected, and whose Lie algebra $\mathfrak{g}$ satisfies the following decomposition  
$$
\mathfrak{g}=\mathfrak{g}_{1}\oplus \mathfrak{g}_{2}\oplus \mathfrak{g}_{3}
$$
with $\mathfrak{g}_{i+1}=[\mathfrak{g}_{1},\mathfrak{g}_{i}]$ for $i=1,2,3$ with the properties $\dim \mathfrak{g}_{1}=2$ and $[\mathfrak{g}_{1},\mathfrak{g}_{3}]=0$.

Notice that the subspace  $\mathfrak{g}_{1}$ is bracket-generating in the Lie algebra~$\mathfrak{g}$ and if $\mathfrak{g}_{1}$ is endowed with an inner product, we can define on $G$ a left-invariant sub-Riemannian structure. In this way $G$ belongs to the class of the so-called Carnot groups~\cite{ABB19,BLU}. There exists a unique Carnot group satisfying the above properties, up to isomorphisms \cite{ABB12,BeMe}, called the Engel group (cf.\ also the discussion in~\cite[Section~6.11]{Mon09}).

\smallskip

It is well known that the exponential map $\exp: \mathfrak{g}\to G$ is a global diffeomorphism and defining for $x,y\in \mathfrak{g}$
\begin{equation}\label{eq:gprod}
x\cdot y\eqdefa \exp^{-1}(\exp(x)\cdot \exp(y))
\end{equation}
 the Lie group $G$ can be identified with $\mathfrak{g}\simeq\R^{4}$ endowed with a polynomial group law \cite{BLU}. Indeed using the Baker-Campbell-Hausdorff formula and the fact that the Engel group $G$ is nilpotent we can write for $x,y\in \mathfrak{g}$ the identity
\begin{equation} \label{eq:bch}
\exp(x)\cdot \exp(y)=\exp\Big(x+y+\frac12 [x,y]+ \frac{1}{12} ([x,[x,y]-[y,[x,y]])\Big) \, .
\end{equation}
Fixing a basis $X_{1},X_{2},X_{3},X_{4}$ of $\mathfrak{g}$ (which we can identify with left-invariant vector fields on $G$) such that 
\begin{gather} \label{eq:gat1}
\mathfrak{g}_{1}=\mathrm{span}\{X_{1},X_{2}\}\,,\qquad \mathfrak{g}_{2}=\mathrm{span}\{X_{3}\}\,,\qquad \mathfrak{g}_{3}=\mathrm{span}\{X_{4}\}\,,\\
X_{3}\eqdefa[X_{1},X_{2}]\,,\qquad X_{4}\eqdefa[X_{1},X_{3}]\label{eq:gat2}
\end{gather}
one can define a set coordinates $x=(x_{1},x_{2},x_{3},x_{4})$ on $G$  by the identity
\begin{equation}\label{eq:g2}
g=\exp\left(\sum_{i=2}^{4}x_{i}X_{i}\right)\exp(x_{1}X_{1})\,.
\end{equation}
 After some computations exploiting \eqref{eq:bch}, one gets 
\begin{align}\label{eq:group2}
\begin{pmatrix}
x_{1}\\x_{2}\\x_{3}\\x_{4}
\end{pmatrix}
\cdot 
\begin{pmatrix}
y_{1}\\y_{2}\\y_{3}\\y_{4}
\end{pmatrix}
=
\begin{pmatrix}
x_{1}+y_{1}\\
x_{2}+y_{2}\\
x_{3}+y_{3}+x_{1} y_{2}\\
x_{4}+y_{4}+x_{1} y_{3}+\frac{x_{1}^{2}}{2}y_{2}
 \end{pmatrix}\, .
\end{align}
With this choice of coordinates, a basis of left-invariant vector fields is given by
\begin{align}
X_{1}&\eqdefa \partial_{x_{1}},\\
 X_{2}&\eqdefa \partial_{x_{2}}+x_{1}\partial_{x_{3}}+\frac{x_{1}^{2}}{2}\partial_{x_{4}}  
 \, ,
 \end{align}
 and thus 
 $$X_{3}=  \partial_{x_{3}}+x_{1}\partial_{x_{4}} \andf
 X_{4} = \partial_{x_{4}}\, . $$
 Notice  that   the inverse of an element $x=(x_{1},x_{2},x_{3},x_{4})$ in the coordinates \eqref{eq:g2} is given by  
\begin{equation}\label{eq:inv}
(x_{1},x_{2},x_{3},x_{4})^{-1}=\big(-x_{1},-x_{2},-x_{3}+x_{1}x_{2},-x_{4}+x_{1}x_{3}-\frac12 x_{1}^{2}x_{2}\big)\, .
\end{equation}

One can define a sub-Riemannian structure on the Engel group $G$ by introducing the bracket-generating distribution $D$ spanned by the vector fields in $\mathfrak{g}_{1}$ and defining an inner product $\langle\cdot,\cdot\rangle$ on~$D$ such that $X_{1}$ and $X_{2}$ define an orthonormal frame. Thanks to the bracket generating  condition, we have the following well-known connectivity property through the so-called horizontal curves for the distribution, which is a consequence of the classical Rashevski-Chow theorem: for every pair of points $x,y\in G$ there exists an absolutely continuous curve $\gamma:[0,T]\to G$ such that~$\dot \gamma(t)\in D_{\gamma(t)}$ and $\gamma(0)=x$, $\gamma(T)=y$. We denote by $\Omega_{x,y}$ the set of absolutely continuous horizontal curves joining $x$ and $y$. If~$\gamma:[0,T]\to G$ belongs to $\Omega_{x,y}$ we set
$$\ell_G (\gamma)\eqdefa\int_{0}^{T}\langle\dot \gamma(t),\dot \gamma(t)\rangle^{1/2}dt\, .$$
This enables one to introduce the sub-Riemannian (also called Carnot-Carath\'eodory) distance $d_G $ on $G$ which is defined as follows 
\begin{equation}\label{eq:dsr}
d_G (x,y)\eqdefa\inf\big\{\ell_G (\g)\mid \gamma\in \Omega_{x,y}\big \}\,\cdot
\end{equation}
This is a well-defined distance inducing the Euclidean topology, moreover the metric space~$(G,d_G )$ is complete. In particular all closed balls $\overline B_G (x,r)$ are compact \cite{ABB19}.

By construction, the sub-Riemannian distance on the Engel group is invariant with respect to   left-invariant multiplications $\tau_{z}:G\to G$ defined by $\tau_{z}(x)\eqdefa z\cdot x$, namely
$$d_G (\tau_{z}x,\tau_{z}y)= d_G (x,y)\, .$$
Moreover, being a Lie group, $G$ can be endowed with a Haar measure which turns out to be a scalar multiple of the Lebesgue measure in $\R^{4}$ in the coordinate set  we have chosen; we shall therefore denote in what follows for simplicity by~$dx$ the Haar measure on~$G$. The corresponding Lebesgue spaces $L^p(G)$ are thus the set of measurable functions~$u: G \to \C$ such that 
$$ \norm {u}{L^{p} (G)}\eqdefa \Big( \int_G |u(x)|^p dx \Big)^{\frac 1 p}< \infty,  \,\, \mbox {if} \,\,  1 \leq p < \infty \, ,$$ 
with the standard modification if $p =\infty$.

\smallskip
The convolution product of any two integrable functions $u$ and $v$ is defined by  
 \begin{equation}\label{conv}
u \star v ( x ) \eqdefa \int_{G} u (x \cdot y^{-1} ) v( y)\, dy
= \int_{G} u ( y ) v( y^{-1} \cdot x)\, dy\, ,
\end{equation}
and even though it is not  commutative, the following Young inequalities hold true:
 \begin{equation}\label{young}\norm{ u \star v }{L^r(G)} \leq \norm u {L^p(G)} \norm v {L^q(G)},
 \quad \hbox{whenever}\ 1\leq p,q,r\leq\infty\ \hbox{ and }\
 \frac{1}{r} =  \frac{1}{p} + \frac{1}{q} - 1\, .
\end{equation}
Moreover if $\cX$ is a left-invariant vector field on $G$, then we have for all $C^1$ functions $u$ and~$v$ with sufficient decay at infinity:
 \begin{equation}\label{convderv}\cX ( u \star v) =  u \star (\cX v)\, .\end{equation}
We also define the left translation by
\begin{equation}
\label{e:left-trans}
(L_xu)(y) \eqdefa u(\tau_x y) = u(x\cdot y)\, .
\end{equation}
 According to~\eqref{conv}, we may also define the convolution between $T \in \mathcal{S}'(G)$ and $u \in \mathcal{S}(G)$  (where 
we recall that the Schwartz space $\cS(G)$  is nothing else than the Schwartz space $\cS(\R^4)$):
 as 
   \begin{equation}\label{e:def-convolS}
 \begin{aligned}
(T\star u)(x) &\eqdefa \langle T, \check u_x \rangle_{\mathcal{S}'(G),\mathcal{S}(G)} , \quad \text{with}\quad\check u_x(y) \eqdefa u(y^{-1}\cdot x) =(L_{y^{-1}}u)(x) \, , \\
( u\star T)(x) & \eqdefa \langle T, \check u^x \rangle_{\mathcal{S}'(G),\mathcal{S}(G)} , \quad \text{with } \quad \check u^x(y) \eqdefa u( x\cdot y^{-1})  = (L_xu)(y^{-1})\, ,
 \end{aligned}
 \end{equation}
which both satisfy $T\star u \in C^\infty(G)$ and $u\star T \in C^\infty(G)$. Note that this actually stands for the definition of the convolution product in~\eqref{e:F-convolution}.

\smallskip Recall also   the following homogeneity property: the Haar measure~$|B_G(x,r)|$ of the ball centered at~$x \in G$ and of radius~$r$ satisfies
\begin{equation}\label{eq:homball}
|B_G(x,r)|=cr^{Q}
\end{equation}
where $c\eqdefa|B_G(0,1)|$, and $Q$ is the homogeneous dimension of the Engel group which is given by
\begin{equation}\label{eq:dimhom}
Q\eqdefa\sum_{j=1}^{3}j\dim \mathfrak{g}_{j}=7\, .
\end{equation}
Identity \eqref{eq:homball} is related to the following crucial fact: defining the dilations
\begin{equation}\label{eq:dilat}
\forall \lambda>0 \, , \quad \delta_{\lambda}:G\to G\,,\qquad \delta_{\lambda}(x_{1},x_{2},x_{3},x_{4})\eqdefa(\lambda x_{1},\lambda x_{2},\lambda^{2}x_{3},\lambda^{3}x_{4})
\end{equation}
we have the following homogeneity
\begin{equation}\label{eq:dsrdil}
d_G (\delta_{\lambda}x,\delta_{\lambda}y)=\lambda d_G (x,y)\, .
\end{equation}
Given $u:G\to \R$ one can introduce its sub-Riemannian gradient $\nabla_G u$ defined as the unique horizontal vector field satisfying
\begin{equation}\label{eq:srgrad1}
\langle \nabla_G u,X\rangle\eqdefa du(X)
\end{equation}
for every horizontal vector field $X\in D$. This translates in terms of the vector fields in the identity
\begin{equation}\label{eq:srgrad2}
\nabla_G u=(X_{1}u)X_{1}+(X_{2}u)X_{2}\, .
\end{equation}
One can then introduce a sublaplacian operator $\Delta_G $   as follows:
\begin{equation}\label{eq:srlapl1}
\Delta_G  u\eqdefa\mathrm{div}(\nabla_G u)
\end{equation}
where $\mathrm{div}$ denotes the divergence with respect to the Haar measure of $G$. In terms of the vector fields we have
\begin{equation}\label{eq:srlapl2}
\Delta_G u=(X^{2}_{1}+X^{2}_{2})u
\end{equation}
but the definition given above guarantees that $\Delta_G $ is an operator which is canonically associated with the sub-Riemannian structure on $G$, i.e., independent of the choice of orthonormal frame~$X_{1},X_{2}$. 

\begin{remark} \label{deftildeX}{\sl  In a similar way one can build a right-invariant sub-Riemannian structure on the Engel group, and build the corresponding right-invariant sub-Riemannian Laplacian. With respect to the product law given by \eqref{eq:group2}, a basis of right-invariant vector fields is given as follows
\begin{align}
\widetilde{X}_{1}&\eqdefa\partial_{x_{1}}+x_{2}\partial_{x_{3}}+x_{3}\partial_{x_{4}}\, ,\\
 \widetilde{X}_{2}&\eqdefa\partial_{x_{2}}\, .
 \end{align}
This defines a right-invariant metric which in turns defines a right-invariant sublaplacian~$\widetilde{\Delta}_G $ as follows:
\begin{equation}\label{eq:srlapl3}
\widetilde\Delta_G  u\eqdefa\mathrm{div}(\widetilde\nabla_G u)\, ,
\end{equation}
where $\mathrm{div}$ denotes the divergence with respect to the   Haar measure  on $G$ (which is indeed bi-invariant since the group $G$ is nilpotent) while the gradient is different since the metric has changed. In terms of the vector fields we have
\begin{equation}\label{eq:srlapl4}
\widetilde{\Delta}_G u=(\widetilde{X}^{2}_{1}+\widetilde{X}^{2}_{2})u.
\end{equation}
}  
\end{remark}

\begin{remark} {\sl The Engel group can also be   described as the set $J^2(\R,\R)$ of $2$-jets of a real function of a single real variable as follows: an element $(x,y,p,q)\in \R^{4}$ represents a 2-jet of a real function if it is of the form $(x,u(x),u'(x),u''(x))$ which is equivalent to the relations $p=\frac{dy}{dx}$, $q=\frac{dp}{dx}$. These relations define a vector distribution (playing the role of $\mathfrak{g}_{1}$) defined by the kernel of the differential forms in $\R^{4}$
$$\omega_{1}=dy-p dx,\qquad \omega_{2}=dp-qdx.$$
For more details on sub-Riemannian structures on jet spaces one can see, for instance, \cite{bdm, viz}.}
\end{remark}


\section{The Fourier transform on the Engel group: Proof of Theorems~\ref{t:key} and~\ref{t:key2} }\label{Fourier}
\subsection{The standard Fourier theory on the Engel group}\label{defsta}
\subsubsection{Definition}
As recalled in the introduction, the  standard way to define an Engel  Fourier transform  consists in using irreducible unitary representations.  The one that we shall use here relies on the representations~$(\mathcal{R}_x^{\nu,\lambda})_{(\nu, \lambda) \in \R\times \R^*}$ introduced in Appendix\refer{irrep}, and that are given for all~$x
$ in~$G$ and~$\f$ in~$L^2(\R)$, by   
\begin{equation}
\label {eq: rep}\mathcal{R}_x^{\nu,\lambda}\f(\theta) \eqdefa\exp\left[i\left(-\frac{\nu}{\lambda} x_{2}+\lambda\big(x_{4}+\theta x_{3}+\frac{\theta^{2}}{2}x_{2}\big)\right)\right] \f(\theta+x_1)\, . \end{equation}
For any~$(\nu, \lambda) \in \R\times \R^*$, the map
$$
\mathcal{R}^{\nu,\lambda}:\left\{
\begin{array} {ccl}
G& \longrightarrow & \cU(L^2(\R))\\
x & \longmapsto & \mathcal{R}_x^{\nu,\lambda}
\end{array}
\right.
$$
is a group homomorphism between the Engel group and the   unitary group~$\cU(L^2(\R))$ of~$L^2(\R)$. 
 Actually~$(\mathcal{R}^{\nu,\lambda})_{(\nu, \lambda) \in \R\times \R^*}$
 plays the same role as the map~$x\mapsto e^{i\langle \xi,x\rangle}$  in the Euclidean case,
as regards the definition of the Fourier transform.
\begin{definition}
\label {definFourier}
{\sl The Fourier transform 
of  an integrable function~$u$ on~$G$ is defined by  
 \begin{equation}
\label {eq: defF}
\forall (\nu, \lambda) \in \R\times \R^* \,,\quad \mathscr{F}(u) (\nu, \lambda)\eqdefa  \int_G u(x) \mathcal{R}_x^{\nu,\lambda} dx\, . 
\end{equation}
}
\end{definition}
\subsubsection{Main properties}
According to   Definition~\ref{definFourier}, the Fourier transform of  an integrable function on $G$ is a family, parametrized by $(\nu, \lambda) \in \R\times \R^*$, of bounded operators on~$L^2(\R)$: for all $u$ in~$L^1(G)$, there holds  
 \begin{equation}
\label {eq: bound}
\forall (\nu, \lambda) \in \R\times \R^* \,,\quad\|\mathscr{F}(u) (\nu, \lambda)\|_{\mathscr{L}(L^2, L^2)} \leq \|u\|_{L^1(G)} \, .  \end{equation}
Despite first appearances,  this Fourier transform  has many common features with the Fourier  transform on  $\R^d$. First, since~$\mathcal{R}^{\nu,\lambda}$ is a group homomorphism,~$\mathscr{F}(u) (\nu, \lambda)$  transforms convolution into composition, that is to say, for all integrable functions $u$ and~$v$,
 \begin{equation}
\label {eq: conv}
\forall (\nu, \lambda) \in \R\times \R^* \,,\quad\mathscr{F}(u\star v )  (\nu, \lambda)= \mathscr{F}(u)  (\nu, \lambda) \circ \mathscr{F}(v)  (\nu, \lambda)\, .  \end{equation}
Moreover as in the Euclidean case, the Fourier-Plancherel  and   inversion formulae hold  true in that setting, with $d\nu d\lambda$ as Plancherel measure, resorting respectively to Hilbert-Schmidt norms and  trace-class operators (see for instance  Corwin-Greenleaf\ccite{corwingreenleaf}).

In order to state  the  Fourier-Plancherel formula, let us recall   the definition of the Hilbert-Schmidt norm. Denoting by~$(e_m)_{n \in \N}$    an orthonormal basis of~$L^2(\R)$, we define the  Hilbert-Schmidt 
norm~$\norm {\mathscr{F}(u) (\nu, \lambda)} 
{HS}$ on~$L^2(\R)$ (which is independent of the choice of the basis) by 
$$
\norm {\mathscr{F}(u) (\nu, \lambda)} 
{HS}\eqdefa\left(\sum_{m \in \N} \norm {\mathscr{F}(u) (\nu, \lambda)e_m} {L^2(\R)}^2 \right) ^\frac12\, .
$$
{Then}   if $u$ belongs to~$L^1(G) \cap L^2(G)$, {then}  $\mathscr{F}(u) (\nu, \lambda)$ is  a Hilbert-Schmidt operator for  almost every~$(\nu, \lambda) $ in~$ \R\times \R^*$, and there holds   \begin{equation}
\label{Plancherelformula} \norm u { L^2(G)}^2  =
(2\pi)^{-3}   \int_{\R\times \R^*} \norm {\mathscr{F}(u) (\nu, \lambda)} 
{HS}^2  \, d\nu d\lambda\, .
\end{equation}
The inversion formula requires introducing the trace of the operator~$\mathcal{R}_{x^{-1}}^{\nu,\lambda}  \mathscr{F}(u) (\nu, \lambda)$. By definition, this operator 
 is trace-class if
$$
 \sum_{m \in \N} 
 \Big| \big(\mathcal{R}_{x^{-1}}^{\nu,\lambda}  \mathscr{F}(u) (\nu, \lambda) e_m|e_m \big)\Big | < \infty \, ,
$$
and, if so, its trace is defined as follows (and as the Hilbert-Schmidt  norm it is independent of the choice of the basis)
$$
{\rm tr}\big(\mathcal{R}_{x^{-1}}^{\nu,\lambda}  \mathscr{F}(u) (\nu, \lambda)\big)\eqdefa \sum_{m \in \N} \big(\mathcal{R}_{x^{-1}}^{\nu,\lambda}  \mathscr{F}(u) (\nu, \lambda) e_m|e_m\big)\, .
$$
In particular if \begin{equation}
\label{condinversion} \sum_{m \in \N} \int_{\R\times \R^*}  \norm {\mathscr{F}(u) (\nu, \lambda)e_m} {L^2(\R)} 
d\nu d\lambda < \infty \, , \end{equation}
 then the  operator $\mathcal{R}_{x^{-1}}^{\nu,\lambda}  \mathscr{F}(u) (\nu, \lambda)$ is of trace-class, and  one has  
\begin{equation}
\label{inverselformula} 
 u(x)=
(2\pi)^{-3}  \int_{\R\times \R^*}  {\rm tr}
\big( \mathcal{R}_{x^{-1}}^{\nu,\lambda}  \mathscr{F}(u) (\nu, \lambda)\big)
d\nu d\lambda\, .
\end{equation}
Let us emphasize  that the hypothesis\refeq{condinversion}  is satisfied
   in  the Schwartz space $\cS(G)$  (see Proposition\refer {inversionSchwartz} below).

\begin{remark} \label{invariance by translation}{\sl  Observe that
  if~$u \in L^1(G)$,  then for all~$(\nu, \lambda) \in \R\times \R^*$ and all~$x \in G$,  
\begin{equation}\label{eq:invariance by translation}
 \mathscr{F} (L_x u)(\nu, \lambda)= \mathcal{R}_{x^{-1}} ^{\nu,\lambda}\,  \mathscr{F} (u)(\nu, \lambda) 
\end{equation}
   where~$L_x$ is the left-translation operator defined in~{\rm(\ref{e:left-trans})}.
Indeed by definition of~$L_x$, we have
$$
\mathscr{F} (L_x u)(\nu, \lambda) = \int_G u(x \cdot y) \mathcal{R}_y^{\nu,\lambda} dy\,  .
$$
Using the left  invariance of the Lebesgue measure, changing variable~$z=x\cdot y$  and
taking advantage of  the fact that $\mathcal{R}^{\nu,\lambda}$ is a group homomorphism, we get
$$\begin{aligned}
\mathscr{F} (L_x u)(\nu, \lambda) & =  \int_G u(z) \mathcal{R}_{x^{-1} \cdot z} ^{\nu,\lambda} dz\\
& =    \mathcal{R}_{x^{-1}} ^{\nu,\lambda}\int_G u(z) \mathcal{R}_{z} ^{\nu,\lambda} dz\\
 & =  \mathcal{R}_{x^{-1}} ^{\nu,\lambda} \mathscr{F} (u)(\nu, \lambda)
\end{aligned}
$$
which proves~{\rm(\ref{eq:invariance by translation}}).}
\end{remark}

\subsubsection{Action on the sublaplacian}
  A key   point in the analysis of   the Engel group consists in studying the action of the Fourier transform  on the sublaplacian $\D_G$ defined by\refeq{eq:srlapl2}.    Actually, we check that,  
for any~$C^2$ function~$\f$ on~$\R,$ for any~$(\nu, \lambda) \in \R\times \R^*$ and any~$x$ in~$G$, there holds
\begin{equation}
\label {eq: rellap}
- \D_G \mathcal{R}_x^{\nu,\lambda} (\f) =  \mathcal{R}_x^{\nu,\lambda} P_{\nu,\lambda} \f \andf - \wt \D_G \mathcal{R}_x^{\nu,\lambda} (\f) = P_{\nu,\lambda} \mathcal{R}_x^{\nu,\lambda}  \f \, ,
\end{equation}
with\footnote{As   will be seen later, the operators $P_{\nu,\lambda}$ and  $\mathsf{P}_\mu$ are, up to the factor $|\lambda|^{2/3}$,  unitarily equivalent.} 
\begin{equation}
\label {eq: oscop} P_{\nu,\lambda}\eqdefa - \frac{d^{2}}{d\theta^{2}} + \left( \frac{\lambda}{2}\theta^{2}-\frac{\nu}{\lambda}\right)^{2} \,. \end{equation} 
This  shows, as explained in the introduction of this paper,  that  the Fourier transform on the Engel group is  strongly tied to the spectral analysis of the quartic oscillator.  To obtain~(\ref{eq: rellap}) we
 take advantage of\refeq{eq: rep} to gather that
 \begin{equation}
\label {eq: 12}X_{1}\mathcal{R}_x^{\nu,\lambda} (\f)  = \mathcal{R}_x^{\nu,\lambda} \frac{d \f}{d\theta}  \andf X_{2}\mathcal{R}_x^{\nu,\lambda} (\f) =  i \Big(\frac{\lambda}{2}\big(\theta+x_1\big)^{2}-\frac{\nu}{\lambda}\Big) \mathcal{R}_x^{\nu,\lambda} (\f)\, ,\end{equation}
which implies that $-\D_G \mathcal{R}_x^{\nu,\lambda} (\f) =  \mathcal{R}_x^{\nu,\lambda} P_{\nu,\lambda} \f$.  Along the same lines, one gets  
\begin{equation}
\label {eq: wt12} \wt X_{1}\mathcal{R}_x^{\nu,\lambda} (\f) =  \frac{d}{d\theta}  \Big(\mathcal{R}_x^{\nu,\lambda} (\f)\Big) \andf  \wt X_{2}\mathcal{R}_x^{\nu,\lambda} (\f) =  i \Big( \frac{\lambda}{2}\theta^{2}-\frac{\nu}{\lambda}\Big) \mathcal{R}_x^{\nu,\lambda} (\f) \, ,\end{equation}
which completes the proof of\refeq{eq: rellap}.
 Note also that 
\begin{equation}
\label {eq: 34}X_{3}\mathcal{R}_x^{\nu,\lambda} (\f) = i \lambda \mathcal{R}_x^{\nu,\lambda}  (\theta \f) \, ,  \,\,X_{4}\mathcal{R}_x^{\nu,\lambda} (\f) = i \lambda \mathcal{R}_x^{\nu,\lambda}  (\f) \end{equation}
 and 
\begin{equation}
\label {eq: wt34}\wt X_{3}\mathcal{R}_x^{\nu,\lambda} (\f) = i \lambda \theta \mathcal{R}_x^{\nu,\lambda}  (\f) \, ,  \,\,  \wt X_{4}\mathcal{R}_x^{\nu,\lambda} (\f) = i \lambda \mathcal{R}_x^{\nu,\lambda}  (\f)  \, .\end{equation} 
 \begin{remark}
\label{rmorder}{\sl Let us  give some insight on the parameters~$(\nu, \lambda) \in \R\times \R^*$ involved in the definition of the Fourier transform\refeq{eq: defF}. By definition,~$\lambda$ belongs to the dual of the center of~$G$, which in accordance with the structure of the   Lie algebra of $G$  is associated to  an operator of homogeneous order~$3$. On the other hand    the parameter~$\nu$   is associated to   the operator~$X_{4}X_{2}-\frac12 X_{3}^{2}$ which is  an  operator of homogeneous order~$4$.  This   can be illustrated through the relations~{\rm(\ref{eq: 12})} and~{\rm(\ref{eq: 34})} which give 
\begin{equation}
\label {eq: nu} \big(X_{4}X_{2}-\frac12 X_{3}^{2}\big) \mathcal{R}_x^{\nu,\lambda} (\f)=  \nu  \mathcal{R}_x^{\nu,\lambda} (\f) \, . \end{equation}} \end{remark}
 \subsection{The Fourier transform seen as a function: proof of Theorem~\ref{t:key}}\label{pointfunct}
 This section is dedicated to introducing 
     an alternative definition of the Fourier transform on~$G$ introduced in Section\refer{defsta}. This
     will provide the construction of the set~$\widehat{G}$, the operator~$U$ and the function~$a$  satisfying~(\ref{UDU*=a}) of Theorem~\ref{t:key}.
     
     \subsubsection{The frequency set}
           This new approach, initiated by H.~Bahouri, J.-Y.~Chemin  and R.~Danchin in the setting of the Heisenberg group\ccite{bcdh},   is based on the spectral analysis of $P_{\nu,\lambda}$ conducted in Appendix\refer{anspctm}, where it is in particular established that  the operator $P_{\nu,\lambda}$ is self-adjoint on its domain, in $L^2(\R)$, with compact resolvent  (for any choice of the parameters). Thus it can be associated with  an orthonormal basis of eigenfunctions~$\psi_m^{\lambda,\nu}$ associated to the eigenvalues~$E_m(\lambda,\nu) \in \R^*_+$  (see Proposition\refer{p:def-E-psi} for further details)
\begin{equation}
\label {eq: sp} P_{\nu,\lambda}\psi_m^{\nu, \lambda} = E_m(\nu, \lambda) \psi_m^{\nu, \lambda}\,.  \end{equation}
Then by projecting~$\mathscr{F}(u) (\nu, \lambda)$ on the basis $(\psi_m^{\nu, \lambda})_{m \in \N}$,  one can see the Fourier transform of~$u$  as the mean value of~$u$ modulated by some oscillatory function in the following way:
 for all~$\wh x\eqdefa(n,m, \nu, \lambda)$ in $\widehat{G}\eqdefa \N^{2}\times \R\times \R^*,$ 
 $$\cF(u) (n,m, \nu, \lambda)\eqdefa  \big(\mathscr{F}(u) (\nu, \lam) \psi_m^{\nu,\lambda}|\psi_n^{\nu,\lambda}\big)_{L^2(\R)}
\, .$$ 
Now computing the right-hand side of the above formula, we discover that 
\begin{equation}
\label {eq: defFf} \cF(u) (n,m, \nu, \lambda)=  \int_{G} 
 \cW \big ((n, m, \nu, \lambda),x\big) u(x) dx\, ,  
  \end{equation}
 with 
 \begin{equation}  \begin{aligned}
\label {eq: W}   \cW\big((n,m, \nu, \lambda), x\big) & \eqdefa \big(\mathcal{R}_x^{\nu,\lambda} \psi_m^{\nu,\lambda}|\psi_n^{\nu,\lambda}\big)_{L^2(\R)}\\  &= e^{i (\lambda x_{4}- \frac{\nu}{\lambda} x_{2})} \int_\R e^{i \lambda(\theta x_{3}+\frac{\theta^{2}}{2}x_{2})} \psi_m^{\nu,\lambda}(\theta+x_1)\psi_n^{\nu,\lambda}(\theta) d \theta 
\, .  \end{aligned} \end{equation}
It readily stems from\refeq{eq: defFf} (and the Cauchy-Schwarz inequality in~\eqref{eq: W} and the fact that~$(\psi_m^{\nu,\lambda})_{m\in \N}$ is an orthonormal basis) that the following continuous mapping holds:
\beq
\label {embtG}
\cF : L^1(G)  \to L^\infty(\wh G)\, .\eeq
In the following $\wh G$ will be called the frequency set of~$G$.

\subsubsection{Proof of Theorem~{\rm\ref{t:key}}}
With this point of view, the Fourier-Plancherel and  inversion    formulae\refeq{Plancherelformula}-\eqref{inverselformula} may be expressed in a similar  way as in the Euclidean case, 
namely  \begin{eqnarray}
\label {FPH}
\|u\|_{L^2(G)}^2  &=& (2\pi)^{-3} \|\cF(u) \|_{L^2(\widehat{G})}^2  \\ 
\label{e:parseval}
\left( u,v \right)_{L^2(G)}  &=& (2\pi)^{-3} \left( \cF(u) ,\cF(v) \right)_{L^2(\widehat{G})} \\
\label {inverseFourierH}u(x) &= &(2\pi)^{-3}  \int_{\widehat{G}} 
\cW\big((n,m, \nu, \lambda), x^{-1}\big) \cF(u) (\wh x) \, d\wh x\, ,
\end{eqnarray}
where the measure $d\wh x$ is defined by 
\begin{equation}
\label {eq: defint} \int_{\widehat{G}} \theta (\wh x)\,d\wh x\eqdefa  \int_{\R\times \R^*} \sum_{(n,m)\in \N^{2} }\theta(n,m, \nu, \lambda)  d\nu d\lam\,,\end{equation}
  and where~$x^{-1}$ is given by\refeq{eq:inv}.
Finally  for any function~$u$ in the Schwartz class~$\cS(G)$ and any~$\wh x\in\widehat{G}$, combining\refeq{eq: rellap} together with\refeq{eq: defFf}-\eqref{eq: W}  along with an integration by parts,  we get according to~(\ref{eq: sp})  
\begin{equation}
\label {actionsubf}\begin{aligned}
\cF(- \D_G  u) (n,m, \nu, \lambda) & =  E_m(\nu, \lambda)\cF(u) (n,m, \nu, \lambda)\andf \\
\cF(-\wt \D_G u)(n,m, \nu, \lambda) & =  E_n(\nu, \lambda)\cF(u) (n,m, \nu, \lambda)\,.
\end{aligned}\end{equation}
This construction   proves Theorem~\ref{t:key}. \qed

\medskip

\subsubsection{Additional properties}
First observe that the  relations\refeq{actionsubf}  lead in particular  to the definition of the homogeneous Sobolev semi-norms   as in the Euclidean case by means of the Fourier transform: \beq
\label {SobH} \|u\|_{ \dot H^s(G)} \eqdefa \big\|(-\Delta_G)^{\frac s 2} u \big\|_{L^2(G)}= (2\pi)^{-3/2} \Big(\int_{\widehat{G}} E_m^s(\nu,\lam) 
\big|\cF (u) (\wh x)\big|^2 \, d\wh x\Big)^{\frac 1 2} \, ,  \eeq
and along the same lines in  the non homogeneous framework
\beq
\label {SobinH}\|u\|_{H^s(G)}\eqdefa \|({\rm Id}-\Delta_G)^{\frac s 2} u\|_{L^2(G)}=  (2\pi)^{-3/2} \Big(\int_{\widehat{G}} (1+E_m(\nu,\lam))^s
|\cF (u) (\wh x)|^2 \, d\wh x\Big)^{\frac 1 2} \,.\eeq  
  Second   note     that in this new setting,  the convolution identity\refeq{eq: bound}  rewrites as follows, for all integrable functions~$u$ and~$v$ and all $\hat x= (n, m, \nu, \lam) \in \widehat{G}$, 
\beq
\label {newFourierconvoleq1}
\cF(u\star v) (\hat x)  = (\cF(u) \cdot \cF(v))(\hat x)\eqdefa \sum_{p\in \N} \cF(u)(n,p,\nu,  \lam)\cF(v)(p,m,\nu, \lam)\, .
\eeq
 Before going further, let us list some   useful properties of the function $\cW$.
\begin {proposition}
\label {relationcWmorphism}
{\sl 
For any $\wh x=(n, m, \nu, \lam)$ in $\widehat{G}$ and~$x$ in $G$, we have 
\begin{eqnarray}
\label {relationcWmorphismeq0}
  \cW\big((n, m, \nu, \lam), 0\big)  & =& \delta_{n,m} \andf  |\cW\big((n, m, \nu, \lam), x\big)| \leq 1\, ,\\
   \label {relationeven} \overline {\cW\big((n, m, \nu, \lam), x\big)} &=& \cW\big((n, m, \nu, - \lam), x\big)\, ,\\
   \label {relationcWinverse}
 \cW\big((n, m, \nu, \lam), x^{-1}\big)& = & \overline {\cW\big((m, n, \nu, \lam), x\big)} \, , \\ 
\label {eq: usew} 
  \ds\sum_{n \in \N} |\cW\big((n, m, \nu, \lam), x\big)|^2&  =&1\, ,\\
\label {relationcWmorphismeq1b}
\cW\big((n, m, \nu, \lam),\d_{r}(x)\big)& =&\cW\big((n, m, r^4 \nu, r ^3 \lam), x\big), \quad \forall r>0\, .
\end{eqnarray}
}
\end{proposition}
\begin{proof} 
The  first   property follows from the fact that $(\psi_m^{\nu,\lambda})_{m\in \N}$ is an orthonormal basis and the Cauchy-Schwarz inequality in~\eqref{eq: W}. 
The second one is an immediate consequence of the fact that, for all $m \in \N$,  thanks to the symmetry invariance\footnote{ One has $P_{\nu,\lambda}=P_{\nu,-\lambda}$.}  of $P_{\nu,\lambda}$ with respect to $\lambda$,
$$\psi_m^{\nu,\lambda}=\psi_m^{\nu, -\lambda}\, .$$
Identity\refeq{relationcWinverse}  follows from\refeq{eq:inv},  while Formula\refeq{eq: usew}   stems from the fact  that~$\mathcal{R}_x^{\nu,\lambda}$ are unitary operators and thus $\|\mathcal{R}_x^{\nu,\lambda}\psi_m^{\nu,\lambda}\|_{L^2(\R)} =1$,  which implies that for all~$m \in \N$
$$  \sum_{n \in \N}\big |(\mathcal{R}_x^{\nu,\lambda}\psi_m^{\nu,\lambda}|\psi_n^{\nu,\lambda})_{L^2(\R)}\big|^2=\sum_{n \in \N} \big|\cW\big((n, m, \nu, \lam), x\big)\big|^2=1\, .$$
In order to prove\refeq {relationcWmorphismeq1b}, we first observe that in view of\refeq{eq: W}, there holds
$$ \begin{aligned} \cW\big((n, m, \nu, \lam),\d_{r}(x)\big) & = e^{i (\lambda r^3x_{4}- \frac{\nu}{\lambda} r x_{2})}  \int_\R e^{i\lambda(\theta r^2 x_3  +\frac{\theta^{2}}{2}r x_{2})} \psi_m^{\nu,\lambda}(\theta+
r x_{1})\psi_n^{\nu,\lambda}(\theta) d \theta. \end{aligned}$$
Then performing the change of variable $\ds \theta=r z$, we deduce that  
$$ \begin{aligned}\cW\big((n, m, \nu, \lam),\d_{r}(x)\big)& = e^{i (r^3 \lambda x_{4}- \frac{r^4\nu}{r^3\lambda} x_{2})}   \int_\R e^{i r^3 \lambda (z y_3 +\frac{z^{2}}{2} y_{2})} T_{r}\psi_m^{\nu,\lambda}(z+ y_{1}) T_{r}\psi_n^{\nu,\lambda}(z) d z \, ,  \end{aligned}$$
where~$T_r$ is the unitary operator in $L^{2}(\R)$ defined by\refeq{e:Talpha}.

 Recalling that by\refeq{e:relbis}, we have~$T_{r} \psi_m^{\nu,\lambda}(\theta) =\psi_m^{r^4 \nu,r^3 \lambda}(\theta)$,  this completes the proof of~(\ref{relationcWmorphismeq1b}), hence of   the proposition.
\end{proof}

 \medbreak
 
  \begin{remark}
\label{rmkerh}{\sl     Note   that  introducing,  for all $r>0$,  \begin{equation}
\label {eq: dilationhat}\wh \d_r(n, m, \nu, \lam) \eqdefa (n, m, r^4 \nu, r^3 \lambda)\, , \end{equation}   it readily follows from{\rm\refeq {relationcWmorphismeq1b}} that \begin{equation}
\label {eq: scaledil}
 \cF(u\circ\d_r)  = r^{-Q} \cF(u) \circ \wh \d_{r^{-1}} \, .\end{equation}
We deduce  that  the frequency set~$\widehat{G}$ has the same homogeneous dimension~$Q$ as  $G$. 
Indeed according to{\rm\refeq{eq: defint}},   we get for any integrable function~$\theta$ on~$\widehat{G}$,   $$\begin{aligned}
\int_{\widehat{G}} (\theta\circ \wh \d_r)(\wh x)  d\wh x &= \int_{\R\times \R^*} \sum_{(n,m)\in \N^{2}} \theta (n,m,r^{4}\nu, r^{3}\lambda)  d\nu d\lambda\\
 &= r^{-Q}\int_{\R\times \R^*} \sum_{(n,m)\in \N^{2} } \theta (n, m, \nu, \lam)   d\nu d\lam\,.
\end{aligned}$$  
}\end{remark}
 In contrast with the Euclidean situation,  when adopting the function point of view for the Engel Fourier transform one has to take into account that somehow, we deal with infinite matrices associated to bounded operators in $L^2(\R)$. To better exploit  the relation between the definitions{\rm\refeq{eq: defF}} and{\rm\refeq{eq: defFf}},  let us introduce the following definition.
 \begin{definition} {\sl For $p\in[1,\infty]$, we define 
 $\cL^{p, 2}_\cF(\widehat{G})$ as the set of   functions~$\theta$ on~$\widehat{G}$ 
equipped with the norm
$$
\|\theta\|_{\cL^{p, 2}_\cF(\widehat{G})} \eqdefa  \|\theta\|_{L_{\nu, \lam, m}^{p}(\R\times \R^*\times \N; \ell_n^2(\N))} \, .
$$}
\end{definition}

The following two statements  are the analogues of the Hausdorff-Young inequality in $\R^{d}$, which we here have both for $\mathcal{F}$ and $\mathcal{F}^{-1}$.
Recall that $\cF^{-1}(\theta)$ is defined by (cf.\ also \eqref{inverseFourierH})
$$\cF^{-1}(\theta)(x) = (2\pi)^{-3}  \int_{\widehat{G}} 
\cW\big((n,m, \nu, \lambda), x^{-1}\big) \theta (\wh x) \, d\wh x\, .
$$

\begin{proposition} \label{p:ifH}
{\sl For all~$1 \leq p \leq 2$, the following inequality holds: for all~$u $ in~$ {\mathcal S}(G)$,
\begin{equation} \label{eq:ifH1}
\|\cF(u) \|_{\cL^{p', 2}_\cF(\widehat{G})} \leq   \|u\|_{L^{p}(G)}\, ,\end{equation}
where~$p'$ is the dual exponent of~$p$, and thus $\mathcal{F}$ extends as a continuous linear map $\mathcal{F}:L^{p}(G) \to \cL^{p', 2}_\cF(\widehat{G})$.

For all $\theta$ in~$\cL^{1, 2}_\cF(\widehat{G})$, its inverse Fourier transform~$\cF^{-1}(\theta)$ belongs to $L^{\infty}(G) \cap C^0(G)$ and the map 
$\cF^{-1} : \cL^{1, 2}_\cF(\widehat{G}) \to L^{\infty}(G) \cap C^0(G)$ is continuous. 

For  all~$1 \leq p \leq 2$ there is a positive constant~$C$ such that for any~$\theta$ in~$\cL^{p, 2}_\cF(\widehat{G})$, its inverse Fourier transform~$\cF^{-1}(\theta)$ belongs to $L^{p'}(G)$ and  
\begin{equation} \label{eq:ifH2}
\|\cF^{-1}(\theta)\|_{L^{p'}(G)}  \leq C \|\theta \|_{\cL^{p, 2}_\cF(\widehat{G})} \, .
\end{equation}
}
\end{proposition}
\begin{proof} Let us start by proving that if $u\in L^1(G)$ then $\cF(u) \in \cL^{\infty, 2}_\cF(\widehat{G})$ and 
\begin{equation}
\label {eq: useful} \|\cF(u) \|_{\cL^{\infty, 2}_\cF(\widehat{G})}  \leq \|u\|_{L^1(G)}  \,  .\end{equation} 
Note that~(\ref{eq: useful}) is more accurate  than{\rm\refeq{embtG}}. 
By definition~(\ref{eq: defFf})  followed by the Cauchy-Schwarz inequality, there holds
$$
\begin{aligned}
\big|\cF(u)(\widehat x) \big |^2&= \Big|  \int_{G} 
 \cW \big ((n, m, \nu, \lambda),x\big) u(x) dx \Big |^2 \\
 &\leq   \int_{G} 
 | \cW \big ((n, m, \nu, \lambda),x\big)|^2 |u(x)|  dx  \int_{G} 
  |u(x)|  dx
 \end{aligned}
$$so according to Identity \eqref{eq: usew},
$$
\sum_{n \in \N}\big|\cF(u)(n, m, \nu, \lambda) \big |^2\leq \|u\|_{L^1(G)}^2 \, ,$$
which proves~(\ref{eq: useful}).
Combining the Fourier-Plancherel formula\refeq{FPH} together with complex  interpolation,  we deduce \eqref{eq:ifH1}. 

\medskip
In light of\refeq{condinversion}-\eqref{inverselformula},  all functions~$\theta \in \cL^{1, 2}_\cF(\widehat{G})$ admit an inverse Fourier transform given by $$\cF^{-1}(\theta)(x) = (2\pi)^{-3}  \int_{\widehat{G}}
\cW\big((n,m, \nu, \lambda), x^{-1}\big) \theta (\wh x) \, d\wh x\, .
$$
Invoking the continuity of the function $\cW$ with respect to $x$ together with  the smoothness of the group operations of multiplication and inversion on $G$, we infer that for all $(n,m, \nu, \lambda) \in \widehat{G}$, the function  $x\mapsto \cW\big((n,m, \nu, \lambda), x^{-1}\big) \theta (n,m, \nu, \lambda)$ is continuous. Since in view of\refeq{eq: usew}, there holds $$\Big|\sum_{n \in \N} \cW\big((n,m, \nu, \lambda), x^{-1}\big) \theta (n,m, \nu, \lambda)\Big| \leq \|\theta (\cdot,m, \nu, \lambda)\|_{\ell_n^2(\N)} \,,$$ applying the Lebesgue dominated convergence theorem, we deduce that   $\cF^{-1}(\theta)$ belongs to    $L^{\infty}(G) \cap C^0(G)$ and satisfies $$\|\cF^{-1}(\theta)\|_{L^{\infty}(G)}  \leq (2\pi)^{-3}  \|\theta \|_{\cL^{1, 2}_\cF(\widehat{G})} \, ,$$
which  implies Formula \eqref{eq:ifH2}   by combining Fourier-Plancherel formula\refeq{FPH} together with complex  interpolation.
\end{proof}

 \begin{remark}
\label{translationnew2}{\sl Let us emphasize that with the Fourier function point of view, the action on left translations\refeq{eq:invariance by translation} translates into  the following property: if~$u \in L^1(G)$,  then for all~$x \in G$ and~$\wh x=(n, m, \nu, \lam) \in \widehat{G}$,  there holds 
\begin{equation} \label {eq: newtran}\cF(L_x u) (n,m, \nu, \lambda)=  \sum_{p\in \N }\cW\big((p,n, \nu, \lambda), {x}\big) \cF(u) (p,m, \nu, \lambda) \, , \end{equation}
the latter sum being finite according to\refeq{eq: usew}  and\refeq{eq:ifH1}. Indeed since for all~$(\nu, \lambda) \in \R\times \R^*$ and all $y \in G$, $\mathcal{R}_y^{\nu,\lambda}$ is a unitary operator  of~$L^2(\R)$, it follows from\refeq{eq: W} that for any integer $n$ \begin{equation}
\label {eq: genform}\mathcal{R}_y^{\nu,\lambda}\psi_n^{\nu,\lambda}= \sum_{p\in \N }\cW\big((p,n, \nu, \lambda), y\big)\psi_p^{\nu,\lambda} \with   \sum_{p \in \N} |\cW\big((p, n, \nu, \lam), y\big)|^2  =1 \, .\end{equation}
Then invoking\refeq{eq:invariance by translation} together with\refeq{eq: defFf}-\eqref {eq: W},   we infer that  
$$ \begin{aligned} \cF(L_x u) (n,m, \nu, \lambda)&=\big(\mathscr{F}(L_x u) (\nu, \lam) \psi_m^{\nu,\lambda}|\psi_n^{\nu,\lambda}\big)_{L^2(\R)}= \big(\mathcal{R}_{x^{-1}} ^{\nu,\lambda}\,  \mathscr{F} (u)(\nu, \lambda)|\psi_n^{\nu,\lambda}\big)_{L^2(\R)} \\ & = \big(\mathscr{F} (u)(\nu, \lambda)|\mathcal{R}_{x} ^{\nu,\lambda}\,  \psi_n^{\nu,\lambda}\big)_{L^2(\R)}\, ,\end{aligned}$$ which thanks to\refeq {eq: genform} leads to\refeq{eq: newtran}. }\end{remark}

  \begin{remark}
\label{translationnew}{\sl  It will be useful later to note that for all $F\in {\mathcal O}^{1,5/2}(\R_+)$, the operator~$F(-\Delta_G)$ acting on $\cS(G)$ is invariant by left translation. This can be proved by means of functional calculus (see for instance~{\rm\cite[Section~5.3]{Schmudgen}}), but can also be obtained easily  from the above remark which ensures  according to\refeq{defFDelta}  that, for all~$u \in \cS(G)$,~$F\in {\mathcal O}^{1,5/2}(\R_+)$,~$x \in G$ and~$\wh x=(n, m, \nu, \lam) \in \widehat{G}$, there holds 
$$\cF(F(-\Delta_G) (L_x u)) (n,m, \nu, \lambda)=  F(E_m(\nu, \lambda)) \sum_{p\in \N }\cW\big((p,n, \nu, \lambda), {x}\big) \cF(u) (p,m, \nu, \lambda) \, ,$$
and 
$$\begin{aligned} \cF(L_x (F(-\Delta_G)  u) (n,m, \nu, \lambda) & = \sum_{p\in \N }\cW\big((p,n, \nu, \lambda), {x}\big) \cF(F(-\Delta_G)  u) (p,m, \nu, \lambda)\\ & =  F(E_m(\nu, \lambda)) \sum_{p\in \N }\cW\big((p,n, \nu, \lambda), {x}\big) \cF(u) (p,m, \nu, \lambda) \, .\end{aligned}$$ 
}
\end{remark} 

 \medbreak 
Let us end this section by establishing  that 
if~$u$ belongs to  the Schwartz space $\cS(G)$, then~$\cF(u)$
belongs to~$\cL^{1, 2}_\cF(\widehat{G})$, which according to Proposition~\ref{p:ifH} is a natural class 
to define the inverse Fourier transform, hence to write
$$
u(x) = (2\pi)^{-3}  \int_{\widehat{G}} 
\cW\big((n,m, \nu, \lambda), x^{-1}\big) {\mathcal F}(u) (\wh x) \, d\wh x\, .
$$
  \begin {proposition}
\label {inversionSchwartz}
{\sl 
For any~$\rho > \frac 7 2$,   there exists a positive constant $C$ such that the following result holds.  For all $u $ in~$ \cS(G)$, its Fourier transform~$\cF(u)$ belongs to~$\cL^{1, 2}_\cF(\widehat{G})$ and
$$\|\cF(u)\|_{\cL^{1, 2}_\cF(\widehat{G})} \leq C \big(\|u\|_{L^1(G)} + \|(-\Delta_G)^\rho u\|_{L^1(G)}  \big) \, .$$
}
\end{proposition}
\begin{proof} The proof is inspired from the proof of the corresponding result on the Heisenberg group which can be found in~\cite{bgheatkernel}.
 In order to establish the result, let us  consider~$u$ in  $\cS(G)$ and split~$\|\cF(u)\|_{\cL^{1, 2}_\cF(\widehat{G})}$  
into two parts  $ I_1+I_2$ where $$ I_1\eqdefa  \sum_{m\in \N}   \int_{E_m(\nu,\lam)\leq 1} \|  \mathcal{F}(u) (\cdot,m,\nu, \lambda)   \|_{\ell^2}\, 
d\nu d\lambda   \, .$$
Since by\refeq{eq:ifH1} (with~$p=1$), one has 
$$\|\cF(u)\|_{\cL^{\infty, 2}_\cF(\widehat{G})}\leq \|u \|_{L^1(G)} \, ,$$	
we deduce that 
$$ I_1  \leq \|u \|_{L^1(G)} \sum_{m\in \N}   \int_{E_m(\nu,\lam)\leq 1}  
d\nu d\lambda \, .$$
Then performing the  change of  variables $\ds \mu =  \frac{\nu}{|\lambda|^{4/3}}$ (for fixed $\lambda$), and recalling~(\ref{defEmnulambda}), we infer that 
\begin{align*} 
I_1  & \leq \|u \|_{L^1(G)} \sum_{m\in \N} \int_{\R}  \int_{|\lambda| \leq \frac 1 {\mathsf{E}_m(\mu)^{3/2}} }  |\lambda|^{4/3} d\lambda d \mu\\  & \lesssim \|u \|_{L^1(G)} \sum_{m\in \N} \int_{\R} \frac {d \mu}   {\mathsf{E}_m(\mu)^{7/2}} \, \virgp \end{align*}
which according to 	\eqref{e:summability-eigenvaluesbis} ensures that $$ I_1 \lesssim \|u \|_{L^1(G)} \, .$$
On the other hand 
thanks to~(\ref{actionsubf}), and again\refeq{eq:ifH1}  with~$p=1$, 
there holds
$$\|  \mathcal{F}(u) (\cdot,m,\nu, \lambda)   \|_{\ell^2} \leq E^ {-k}_m(\nu,\lam)\| (-\Delta_G)^ k u \|_{L^1(G)} $$
 for any integer~$k \in \ZZ$, and thanks to complex interpolation, we  find that for all $ \rho \in \R$, 
$$\|  \mathcal{F}(u) (\cdot,m,\nu, \lambda)   \|_{\ell^2}\leq E^ {-\rho}_m(\nu,\lam)\| (-\Delta_G)^ \rho u \|_{L^1(G)}\, .$$
We deduce that 
$$I_2  \leq \| (-\Delta_G)^ \rho u \|_{L^1(G)} \sum_{m\in \N}   \int_{E_m(\nu,\lam)\geq 1}  
E^ {-\rho}_m(\nu,\lam) d\nu d\lambda  \, .$$
Considering again the  change of  variables $\ds \mu =  \frac{\nu}{|\lambda|^{4/3}}$ (for fixed $\lambda$), this leads to the following estimate  
\begin{align*} 
I_2 &  \leq \| (-\Delta_G)^ \rho u \|_{L^1(G)} \sum_{m\in \N} \int_{\R} \mathsf{E}_m(\mu)^{-\rho}  \int_{|\lambda| \geq \frac 1 {\mathsf{E}_m(\mu)^{3/2}} }  |\lambda|^{4/3- 2/3 \rho} d\lambda d \mu\\  & \lesssim  \| (-\Delta_G)^ \rho u \|_{L^1(G)} \sum_{m\in \N} \int_{\R} \frac {d \mu}{\mathsf{E}_m(\mu)^{7/2}}  \end{align*} 
as soon as~$\rho>7/2$, which achieves the proof of the proposition. 
\end{proof}   
  \subsection{The spectral  measure of $- \Delta_G$}\label{spectral} 
 Let us start by  recalling that the spectral measure  of a self-adjoint operator~$A$  on~$L^2(\R^d)$ is characterized  for any     continuous bounded function~$F$   by
$$\langle \mu_{u,v} ,F\rangle \eqdefa \big(F(A)u,v\big)_{L^2(\R^d)} \, , \quad u,v\in L^2(\R^d) \, ,$$
and thus in particular when $A=-\D$, we get thanks to the Fourier-Plancherel formula
\begin{equation}
\label{eq:linkspectral} \langle \mu_{u,v} ,F\rangle =(2\pi)^{-d} \int_{\hat \R^d} F(|\xi|^2) \wh u(\xi) \overline {\wh v} (\xi) d\xi\,.\end{equation}
Using spherical  coordinates in $\hat \R^d$, we readily gather that
\begin{equation}
\begin{aligned}
\label{eq:SR} (-\D u,v)_{L^2(\R^d)} & = (2\pi)^{-d}  \int^\infty_0 R^2 \int_{\hat \S_{d-1}} \wh u(R\omega) \overline {\wh v} (R\omega) R^{d-1}d \omega  d R \\
&  = (2\pi)^{-d}  \int^\infty_0 R^2 \int_{\hat \S_{d-1}(R)} \wh u(\xi) \overline {\wh v} (\xi) d\hat\S_{d-1}(R)  d R\,.\end{aligned} \end{equation}
Then setting $\gamma=R^2$, we infer that 
$$ (-\D u,v)_{L^2(\R^d)}= \int^\infty_0   \gamma  (A(\gamma)u|v)d \gamma\, ,$$
with 
\begin{equation}
\begin{aligned}
\label{eq:repeuc} (A(\gamma)u|v) & \eqdefa (2\pi)^{-d} \int_{\hat \S_{d-1}} \wh u(\sqrt\gamma\omega) \overline {\wh v} (\sqrt\gamma\omega) (\sqrt\gamma)^{d-1} \frac 1 {2 \sqrt\gamma}  d \omega \\ &= (2\pi)^{-d} \frac 1 {2 \sqrt\gamma} \int_{\hat \S_{d-1}(\sqrt\gamma)} \wh u(\xi) \overline {\wh v} (\xi)d \hat \S_{d-1}(\sqrt\gamma) \, .\end{aligned} \end{equation}
The above formula can be interpreted  as the spectral decomposition of $-\D$
$$-\D u =\int^\infty_0   \gamma d{\mathcal P}_\gamma u$$
and one has
$$ (-\D u,v)_{L^2(\R^d)}= \int^\infty_0   \gamma (d{\mathcal P}_\gamma u|v)  \with (d{\mathcal P}_\gamma u|v)\eqdefa (A(\gamma)u|v)d \gamma\,.$$
We deduce that for any     continuous bounded function~$F$
\begin{equation}
\label{eq:eucformula1}
(F(-\D) v,v)_{L^2(\R^d)} =\int_{\R_+}F(\gamma)(d{\mathcal P}_\gamma v|v)=(2\pi)^{-d} \int_{\hat \R^d} F(|\xi|^2)  |{\wh v} (\xi)|^2 d\xi \, .
 \end{equation}

 Arguing similarly for the Engel group, we infer that  the spectral measure  of the self-adjoint operator~$-\D_G$ is given, thanks to{\refeq{e:parseval}}, by
\begin{equation}\begin{aligned}
\label{eq:linkspectralG}
\langle \mu_{u,v},F\rangle &\eqdefa \big(F(-\D_G)u,v\big)_{L^2(G)}\\ & = (2\pi)^{-3} \int_{\R\times \R^*} \sum_{n,m\in   \N}  F\bigl(E_m(\nu, \lambda)\bigr) \cF(u)(n,m, \nu,\lam) \overline {\cF(v)}(n,m, \nu,\lam) d\nu d\lam\,,
\end{aligned} \end{equation}
for any  continuous bounded function~$F$  
 and all~$u, v$   in~$L^2(G)$.
 Then performing the change of variable $\ds \mu =\frac{\nu}{|\lambda|^{4/3}}$ (for fixed $\lambda$), we deduce that 
$$
 \begin{aligned}&(-\D_Gu, v)_{L^2(G)} 
&= (2\pi)^{-3}\sum_{m\in \N}\int_{\R\times \R^*}  \sum_{n\in \N}  \cF(u)(n,m, \mu|\lambda|^{4/3},\lam) \overline {\cF(v)}(n,m, \mu |\lambda|^{4/3},\lam) \\
& & \qquad \times  |\lambda|^{2/3}\mathsf{E}_{m}(\mu) d\mu |\lambda|^{4/3} d\lam .\end{aligned}$$
 Recalling that $E_m(\nu, \lambda)=|\lambda|^{2/3}\mathsf{E}_{m}(\mu)$ plays the same role as $|\xi|^2$ in the Euclidean framework,  we now consider the change of variables  $R^2=|\lambda|^{2/3}\mathsf{E}_{m}(\mu)$, which gives rise to
$$ \begin{aligned}  (-\D_Gu,v)_{L^2(G)} =  (2\pi)^{-3}  \int_0^{\infty}R^2 \sum_{m\in \N}\sum_{ \pm} \int_{\R} \frac{3 d\mu}{\mathsf{E}_m(\mu)^{\frac Q 2}}  \sum_{n\in \N}  \cF(u)\Big(n,m, \frac{\mu  R^4}{\mathsf{E}_m(\mu)^{2}}, \frac{ \pm R^3}{\mathsf{E}_m(\mu)^{\frac 3 2}}\Big) \\
{} \times \overline {\cF(v)}\Big(n,m, \frac{\mu  R^4}{\mathsf{E}_m(\mu)^{2}}, \frac{ \pm R^3}{\mathsf{E}_m(\mu)^{\frac 3 2}}\Big) R^6 dR\,. \end{aligned}$$
 Analogously to\refeq{eq:SR}, the above formula  can be reinterpreted in terms of the dual sphere~$\S_{\widehat{G}}(R)$, namely \begin{equation}
\begin{aligned}
\label{eq:SGG} (-\D u,v)_{L^2(G)} &  = (2\pi)^{-3}  \int^\infty_0 R^2 \int_{\S_{\widehat{G}}(R)}  \cF(u)(\hat x) \overline {\cF(v)}(\hat x) d\S_{\widehat{G}}(R) d R\,,\end{aligned} \end{equation}
where $\ds \S_{\wh G} (R)\eqdefa \bigl\{(n,m,\nu, \lambda) \in \wh G \,/\,
E_m(\nu, \lambda)= R^2\bigr\}$ and 
\begin{equation}
\begin{aligned}
\label{eq:intSGG}   \int_{\S_{\widehat{G}}(R)}  |\theta(\hat x)| ^2 d\S_{\widehat{G}}(R) \eqdefa 3 \sum_{m\in \N} \sum_{ \pm} \int_\R     \frac{d\mu}{\mathsf{E}_m(\mu)^{\frac Q 2}} \sum_{n\in \N} \Big|\theta\Big(n, m, \frac{\mu  R^4}{\mathsf{E}_m(\mu)^{2}}, \frac{ \pm R^3}{\mathsf{E}_m(\mu)^{\frac 3 2}}\Big)\Big| ^2 R^6\, \cdotp
 \end{aligned} \end{equation}This definition is justified     by the following proposition which is the analog of the classical
integration formula  in spherical  coordinates.  
\begin {proposition}
\label {integpolarcoordiantes}
{\sl
For   any    function~$\theta \in L^2( \widehat{G})$,   we have
$$
\int_{\widehat{G} } |\theta(\hat x)| ^2 d\wh x = \int_0^\infty \biggl(\int_{\S_{\wh G}(R)} |\theta(\hat x)| ^2 d\s_{\S_{\wh G}(R)}\biggr)dR\,,
$$
where $\ds \int_{\S_{\wh G}(R)} |\theta(\hat x)| ^2 d\s_{\S_{\wh G}(R)}$ is given by{\rm\refeq{eq:intSGG}}. }
\end{proposition}
\begin{proof} By definition of the measure on~$\widehat{G}$, we have
$$\int_{\widehat{G} } |\theta(\hat x)| ^2 d\wh x =  \int_{\R\times \R^*} \sum_{n,m\in \N }|\theta(n,m, \nu, \lambda)| ^2  d\nu d\lam \,.$$
Then performing successively the change of variables~$\ds \mu =  \frac{\nu}{|\lambda|^{4/3}}$ (for fixed $\lambda$) and~$R= |\lambda|^{1/3}\sqrt{\mathsf{E}_m(\mu)}$ (for fixed $\mu$),  we infer  that 
$$
\begin{aligned}
\int_{\wh G }  |\theta(\hat x)| ^2 d\wh x = \int_0^\infty 3\sum_{m\in \N} \sum_{ \pm}  \int_{\R}    \frac{d\mu}{\mathsf{E}_m(\mu)^{\frac Q 2}} \sum_{n\in \N} \Big|\theta\Big(n, m, \frac{\mu  R^4}{\mathsf{E}_m(\mu)^{2}}, \frac{ \pm R^3}{\mathsf{E}_m(\mu)^{\frac 3 2}}\Big)\Big| ^2 R^6dR\,,
\end{aligned}
$$
which proves the proposition. 
\end{proof}
\begin{remark}
\label{rmksphere}{\sl  As a byproduct of the above formula, we find that  the measure of the dual unit sphere $\S_{\wh G}$ is given by
\begin{equation}
\label{meassphere} \s(\S_{\wh G}) = 3  \sum_{m\in \N} \int_\R     \frac{d\mu}{\mathsf{E}_m(\mu)^{\frac Q 2}} \, \cdotp \end{equation} 
This proves Proposition~{\rm\ref{prop:CG}}.} \end{remark} \noindent Finally setting $\gamma=R^2$, we deduce that 
\begin{equation}
\label{spectralGmeassphere} (-\D_Gu,v)_{L^2(G)}= \int^\infty_0   \gamma \big(A(\gamma)u|v\big)d \gamma $$ with 
$$\big(A(\gamma)u|v\big)= \frac {(2\pi)^{-3}}  {2 \sqrt \gamma}   \int_{\S_{\wh G}(\sqrt \gamma)} \sum_{n\in \N}  \cF(u)(n, \cdot)\overline  \cF(v)(n, \cdot)d \S_{\wh G}(\sqrt \gamma) \cdotp\end{equation} 
This shows that  the spectral decomposition of~$-\D_G$ takes the following form
 $$
  -\D_G u =\int^\infty_0   \gamma d{\mathcal P}_\gamma  u \quad \mbox{where} \quad d({\mathcal P}_\gamma u|v)\eqdefa \big(A(\gamma)u|v\big)d \gamma \, ,$$
and readily ensures that,  for any     continuous bounded function~$F$ and all functions $v$ in $L^2(G)$, there holds\begin{equation}  
\label{eq:spformulam2} 
(F(-\Delta_G) v , v)_{L^2(G)} =\int_{\R_+}F(\gamma)(d{\mathcal P}_\gamma v | v) \, d\gamma=(2\pi)^{-3}\int_{\tilde{G}} F(a(\widehat{x})) |\cF v(\widehat{x})|^2 d\widehat{x} \, ,
\end{equation} where $a$ denotes the function on $\widehat{G}$  introduced in Theorem~{\rm\ref{t:key}}.

 \subsection{End of the proof of Theorems~{\ref{t:key2}}, {\ref{t:function-calculus-L1}}, {\ref{t:kernel} and proof of Proposition~\ref{p:function-calculus-abstract}}}
 \label{Proofkey}

\begin{proof}[Proof of Theorem~{\rm\ref{t:key2}}]
    Recall that~(\ref{e:summability-eigenvaluesbis}) has been proved in Section~\ref{poisson}.   Let us start by proving the       following identity, for~$F \in C^0_c(\R_+^*)$: 
  \begin{equation}  \label{first identity}
\sum_{m\in \N} \int_{\R\times\R^*} F\big(E_m(\nu,\lambda)\big)  d\nu  d\lambda =  3\left( \int_{\R_+}  r^{5/2} F(r)  dr \right)  \sum_{m\in \N} \int_{\R}  \frac{1}{\mathsf{E}_m(\mu)^{7/2}} d\mu  \, .  \end{equation}
Recalling that $P_{\nu,\lambda}=P_{\nu,-\lambda}=P_{\nu,|\lambda|}$, we   have $E_m(\nu,\lambda)=E_m(\nu,-\lambda)= E_m(\nu,|\lambda|)$ so it suffices to prove that
$$
\sum_{m\in \N} \int_{\R\times{\R_+^*} } F\big(E_m(\nu,\lambda)\big)  d\nu  d\lambda=  \frac32 \left( \int_{\R_+} r^{5/2} F(r)  dr \right)  \sum_{m\in \N} \int_{\R}  \frac{1}{\mathsf{E}_m(\mu)^{7/2}} d\mu\, .
$$
Invoking Corollary~\ref{c:scaling}, and using changes of variables, the integral on $\lambda >0$ rewrites
$$
\sum_{m\in \N} \int_{\R \times {\R_+^*}} F\big(E_m(\nu,\lambda)\big)  d\nu  d\lambda
=\sum_{m\in \N} \int_{\R \times {\R_+^*}} F\left(\lambda^{2/3}E_m\left(\frac{\nu}{\lambda^{4/3}},1\right)\right)  d\nu d\lambda  
$$
so setting~$\mu = \frac{\nu}{\lambda^{4/3}}$ (for fixed $\lambda$) and then $r =\lambda^{2/3}\mathsf{E}_m\left(\mu \right) $ (for fixed $\mu$) we find
\begin{align*}
\sum_{m\in \N} \int_{\R \times {\R_+^*}} F\big(E_m(\nu,\lambda)\big)  d\nu  d\lambda 
&= \sum_{m\in \N}\int_{\R \times {\R_+^*}}F\left(\lambda^{2/3}\mathsf{E}_m\left(\mu \right)\right) \lambda^{4/3}d\lambda \, d\mu  \\
&= \left(\int_{{\R_+^*}} \frac32 r^{5/2} F(r)  dr \right)  \sum_{m\in \N} \int_{\R}  \frac{1}{\mathsf{E}_m(\mu)^{7/2}} d\mu  \, .
\end{align*}
This concludes the proof of the identity~(\ref{first identity}). The latter remains true (with equality in $[0,+\infty]$) for all nonnegative measurable functions $F$.
Applied to $|F|$ instead of $F$, this implies that~$F\circ a $ belongs to~$L^1(\widehat{G}, \delta_{n,m} d\widehat{x})$ if and only if $F\in L^1(\R_+,r^{5/2} dr)$, and, if so, then~(\ref{magic formula introduction}) (which is nothing but~(\ref{first identity}))  holds. 
 \end{proof}

\smallskip

We now prove Theorem~\ref{t:function-calculus-L1}. Note that the Schwartz kernel theorem~\cite[Thm 5.2.1]{hormander} or~\cite[Equation~(51.7) p~531]{treves} states that any continuous map~${\mathcal S}(G)\to {\mathcal S}'(G)$ has a distribution kernel. Here, left translation invariance of the operator $F(-\Delta_G)$ further implies that the operator is a right convolution operator. Our proof of Theorem~\ref{t:function-calculus-L1} is inspired by~\cite[Proof of Theorem~1.2 p~98]{Hor:trans-inv} and does not rely on the kernel theorem.

\begin{proof}[Proof of Theorem~{\rm\ref{t:function-calculus-L1}}]
To prove that, for any function $u \in \mathcal{S}(G)$,  one can define the inverse Fourier transform of the function~$(n,m,\nu, \lambda)\mapsto F(E_m(\nu, \lambda))\mathcal F({u})(n,m,\nu, \lambda)$, it suffices according to Proposition~\ref{p:ifH} to check that this function belongs to $\mathcal{L}^{1,2}_{\mathcal{F}}$, that is to say
$$
J \eqdefa  \left\| F(E_m(\nu, \lambda))\mathcal F({u})(n,m,\nu, \lambda)\right\|_{\mathcal{L}^{1,2}_{\mathcal{F}}} = \sum_{m\in \N} \int_{\R\times \R^*}  \big| F(E_m(\nu,\lambda)) \big|  \|  \mathcal{F}(u)(\cdot,m,\nu,\lambda) \|_{\ell^2}
d\nu d\lambda < \infty \, .$$
For this we reproduce the proof of Proposition~\ref{inversionSchwartz}, writing~$ J = J_1+J_2$
with  $$ J_1\eqdefa  \sum_{m\in \N}   \int_{E_m(\nu,\lam)\leq 1} \big| F(E_m(\nu,\lambda)) \big|  \|  \mathcal{F}(u)(\cdot,m,\nu,\lambda) \|_{\ell^2}
d\nu d\lambda   \, .$$
As in the estimate of~$I_1$ in the proof of Proposition~\ref{inversionSchwartz}, and thanks to~(\ref{first identity}), we find   
\begin{align*} J_1  & \leq \|u \|_{L^1(G)} \sum_{m\in \N} \int_\R  \int_{|\lambda| \leq \frac 1 {E^{3/2} _m(\mu)} }  \left | F\left(\lambda^{2/3}\mathsf{E}_m\left(\mu \right)\right) \right| |\lambda|^{4/3} d\lambda d \mu\\  & \lesssim \|u \|_{L^1(G)}   \int_0^1   r^{5/2} |F(r)|  dr \\
& \lesssim \|u \|_{L^1(G)} \|F\|_{\mathcal O^{1,\frac52 }_0(\R_+)}   \,.\end{align*}
On the other hand 
and again as in the estimate of~$I_2$ in the proof of Proposition~\ref{inversionSchwartz} 
$$
\begin{aligned}
J_2 & \leq \| (\mbox{Id}-\Delta_G)^ \ell u \|_{L^1(G)}  \int_1^\infty  \langle r\rangle^{-\ell} r^{5/2} |F(r)|  dr  \\
& \lesssim \|(\mbox{Id}-\Delta_G)^ \ell u\|_{L^1(G)} \|F\|_{\mathcal O^{1,\frac52 }_\ell(\R_+)}  
  \, .
  \end{aligned}$$
This implies that for $F \in \mathcal{O}^{1,5/2}_\ell(\R_+)$ and $u \in \mathcal{S}(G)$, 
\begin{equation*}  
\begin{aligned}
J &=  \left\| F(E_m(\nu, \lambda))\mathcal F({u})(n,m,\nu, \lambda)\right\|_{\mathcal{L}^{1,2}_{\mathcal{F}}}\\
& \lesssim \|u \|_{L^1(G)} \|F\|_{\mathcal O^{1,\frac52 }_0(\R_+)}  +  \|(\mbox{Id}-\Delta_G)^ \ell u \|_{L^1(G)} \|F\|_{\mathcal O^{1,\frac52 }_\ell(\R_+)}  
 \, .
  \end{aligned}\end{equation*}
We have thus obtained that 
$(n,m,\nu, \lambda)\mapsto F(E_m(\nu, \lambda))\mathcal F({u})(n,m,\nu, \lambda)$ belongs to $\mathcal{L}^{1,2}_{\mathcal{F}}$. According to Proposition~\ref{p:ifH}, its inverse Fourier transform 
$$
F(-\Delta_G) u \eqdefa \mathcal{F}^{-1} \left( F(E_m(\nu, \lambda))\mathcal F({u})(n,m,\nu, \lambda) \right) 
$$
thus satisfies $F(-\Delta_G) u  \in C^0(G) \cap L^\infty(G)$ together with 
\begin{align}
\|F(-\Delta_G) u \|_{L^\infty(G)} \lesssim  \|u \|_{L^1(G)} \|F\|_{\mathcal O^{1,\frac52 }_0(\R_+)}  +  \|(\mbox{Id}-\Delta_G)^ \ell u \|_{L^1(G)} \|F\|_{\mathcal O^{1,\frac52 }_\ell(\R_+)}  .
\end{align}
In particular, the bilinear map 
\begin{equation}
\label{e:cont-stat}
\mathcal O^{1,\frac52 }_\ell(\R_+) \times  \mathcal{S}(G) \to \C , \quad (F,u) \mapsto  \left( F(-\Delta_G) u \right)(0)
\end{equation}
is linear, continuous for the topology of $\mathcal O^{1,\frac52 }_\ell(\R_+) \times \mathcal{S}(G)$. As a consequence, for a fixed $F \in \mathcal O^{1,\frac52 }_\ell(\R_+)$, the partial map $\mathcal{S}(G) \to \C$ given by $u \mapsto  \left( F(-\Delta_G) u \right)(0)$ belongs to $\mathcal{S}'(G)$. That is to say, there is $T \in \mathcal{S}'(G)$ such that
$$
 \left( F(-\Delta_G) u \right)(0) = \langle T , u \rangle_{\mathcal{S}'(G),\mathcal{S}(G)} , \quad \text{ for all }u \in \mathcal{S}(G). 
$$
Now, there is $\check T \in\mathcal{S}'(G)$ such that, with $\check u(y) = u(y^{-1})$, we have 
$$
\langle T , u \rangle_{\mathcal{S}'(G),\mathcal{S}(G)} = \langle \check T , \check u \rangle_{\mathcal{S}'(G),\mathcal{S}(G)}  \, , \quad \text{ for all }u \in \mathcal{S}(G)  \, . 
$$
Recalling the definition of the convolution in~\eqref{e:def-convolS}, and noticing that~$\check u=\check u_0=\check u^0$, the above two lines rewrite 
$$
 \left( F(-\Delta_G) u \right)(0) = \langle \check T , \check u \rangle_{\mathcal{S}'(G),\mathcal{S}(G)} = (u \star \check T)(0)  \, , \quad \text{ for all }u \in \mathcal{S}(G) \, . 
$$
 Invoking Remark\refer{translationnew}, we infer that, for all $x \in G$ and $u \in \mathcal{S}(G)$, 
\begin{align}
\label{e:translation-}
 \left( F(-\Delta_G) u \right)(x) =  \left(L_x  F(-\Delta_G) u \right)(0) =  \left( F(-\Delta_G) L_x u \right)(0) =  ((L_xu) \star \check T)(0) = (u \star \check T)(x)
 \end{align}
 where, in the last equality, we have used again the definition of the convolution in~\eqref{e:def-convolS}.
 This concludes the proof of Theorem~\ref{t:function-calculus-L1} with $k_F\eqdefa \check T$. 
 The continuity statement of the map $F \mapsto \check T = k_F$ follows from~\eqref{e:cont-stat},~\eqref{e:translation-} and the continuity of $L_x$ as a map $\mathcal{S}(G)\to \mathcal{S}(G)$. 
\end{proof}

Notice that~(\ref{e:translation-}), joint with~(\ref{inverseFourierH}) and~(\ref{actionsubf}), imply the following useful identities:
\beq
\label {kernel}\cF (k_{F} )(\hat x)=F(E_m(\nu, \lambda)) \delta_{m,n}\,,\eeq
and
\beq
\label {kernel2}k_{F}(x)=(2\pi)^{-3}    \int_{\widehat{G}} 
\cW\big((n,m, \nu, \lambda), x^{-1}\big) F(E_m(\nu, \lambda)) \delta_{m,n}\, d\wh x\,.\eeq
For the sake of completeness, we also give here a proof of Proposition~\ref{p:function-calculus-abstract}, which follows that of Theorem~\ref{t:function-calculus-L1} with a different starting point (the general functional calculus for selfadjoint operators instead of the Fourier transform $\mathcal{F}$), and would hold in any Carnot group.

\begin{proof}[Proof of Proposition~{\rm\ref{p:function-calculus-abstract}}]
As it was emphasized in Section~\ref{spectral},  spectral theory associates with~$-\Delta_G$ a spectral measure that we   denoted by~$d{\mathcal P}_\gamma$, and it is well-known (see for instance~\cite[Theorem VIII.6]{Reed-Simon-1} or~\cite[Section~5.3]{Schmudgen}) that if~$F$ is a locally bounded Borel function on~$\R_+$, one can define on the Hilbert space $L^2(G)$ the operator~$F(-\Delta_G)$ by  
\begin{align*}
D(F(-\Delta_G))\eqdefa \Big\{u \in L^2(G) ,\,  \int _\R |F(\gamma) |^2 d({\mathcal P}_\gamma u|u)<\infty\Big\}\, , \\
 F(-\Delta_G)u \eqdefa \int _\R F(\gamma)   d{\mathcal P}_\gamma u  , \quad \text{ for } u \in D(F(-\Delta_G)) \, .
\end{align*}
Now if~$F$ is a function in~${\mathcal O}^\infty_m (\R_+)$ for some nonnegative real number~$m$, then we have, for $u \in \mathcal{S}(G)$,
$$
\| F(-\Delta_G)u\|_{L^2(G)}^2 =  \int _\R |F(\gamma) |^2 d({\mathcal P}_\gamma u|u) \leq C \int _\R \langle \gamma\rangle^{2m} d({\mathcal P}_\gamma u|u) = C \| (\mbox{Id}-\Delta_G)^m u\|_{L^2(G)}^2.
$$
Therefore,~$D\big((\mbox{Id} - \Delta_G)^m\big)$  is embedded continuously in~$D(F(-\Delta_G))$.
Next, since $\Delta_G$ is a differential operator with polynomial coefficients, we have $(\mbox{Id}-\Delta_G)^k u \in \mathcal{S}(G)$ for any $u \in \mathcal{S}(G)$ and~$k \in \N$, together with 
$$
\| (\mbox{Id}-\Delta_G)^k F(-\Delta_G)u\|_{L^2(G)}^2 \leq  C \| (\mbox{Id}-\Delta_G)^{m+k} u\|_{L^2(G)}^2, \quad \text{ for all } u \in \mathcal{S}(G) .
$$
As a consequence, we obtain
\begin{equation}
\label{e:e-Fuin}
u \in \mathcal{S}(G) \implies F(-\Delta_G)u \in \bigcap_{k \in \N} D((\mbox{Id}-\Delta_G)^k) =\bigcap_{k \in \N} D(\Delta_G^k).
\end{equation}
 Now, local hypoellipticity of the operator $\Delta_G$ proved in~\cite{Hor:67} (see also~\cite{RS:76} for the expression of~$\eps$) implies the existence of $\eps>0$ such that  
$D(-\Delta_G) \subset H^\eps_{\text{loc}}(G)$ with continuous embedding, where $H^\eps_{\text{loc}}(G)\eqdefa H^\eps_{\text{loc}}(\R^4)$ denotes the usual local Sobolev space defined by: $v \in H^\eps_{\text{loc}}(\R^4)$ if~$\langle \xi \rangle^{\eps}\mathsf{F}(\chi v)(\xi) \in L^2(\R^4)$ for all~$\chi \in C^\infty_c(\R^4)$, where $\mathsf{F}$ denotes the usual Euclidean Fourier transform. 
A classical induction argument (see e.g.~\cite[Corollary~B.2]{LL:22}) shows that $D(-\Delta_G^k) \subset H^{k\eps}_{\text{loc}}(G)$ for all $k \in \N$ with continuous embedding. 
We thus deduce from~\eqref{e:e-Fuin} (and the usual rough Sobolev embeddings in $\R^4$) that if $u \in \mathcal{S}(G)$, then $F(-\Delta_G)u \in \bigcap_{k \in \N} H^{k\eps}_{\text{loc}}(G) = C^\infty(G)$ (and this map is continuous).
In particular, $F(-\Delta_G)$ maps continuously $\mathcal{S}(G)$ in $C^0(G)$ and, from this point forward, we may follow the end of the proof of Theorem~{\rm\ref{t:function-calculus-L1}} line by line.

\end{proof}

Let us turn to the proof of Theorem~{\rm\ref{t:kernel}}.

\begin{proof}[Proof of Theorem~{\rm\ref{t:kernel}}]
First assume that $F \in L^1(\R_+,r^{5/2}dr)$. Thanks to formulae \eqref{relationcWmorphismeq0},~\eqref{first identity} and\refeq {kernel2}, we deduce that
\begin{align*}
|k_{F}(x) | & = (2\pi)^{-3}   \left|  \int_{\widehat{G}} 
\cW\big((n,m, \nu, \lambda), x^{-1}\big) F(E_m(\nu, \lambda)) \delta_{m,n}\, d\wh x \right| \\
 &  \leq 
(2\pi)^{-3}  \int_{\widehat{G}} 
 |F(E_m(\nu, \lambda))| \delta_{m,n}\, d\wh x 
  = (2\pi)^{-3} \left(\int_{\R_+^*} \frac32 r^{5/2} |F(r)|  dr \right)  \sum_{m\in \N} \int_{\R}  \frac{1}{\mathsf{E}_m(\mu)^{7/2}} d\mu\, .
\end{align*}
The continuity of $k_F$ under the hypothesis of Theorem~\ref{t:kernel} readily follows  from  the continuity of~$\cW$ with respect to $x$,\refeq{relationcWmorphismeq0} and Lebesgue dominated convergence theorem.
\end{proof}

 \section{Applications}\label{s:app}
\subsection{Functional embeddings} \label{Sobembed} Combining the Engel Fourier transform together with\refeq{magic formula introduction},  we  recover in this section many    functional inequalities, whose original proofs may be found in~\cite{Folland}.  
 
 Let us start with the following result concerning   Sobolev embeddings in Lebesgue spaces.
Such embeddings are known to hold in a variety of contexts (see for instance~\cite{bfggraded},~\cite{BPTV},~\cite{chamorro}~\cite{CSC},~\cite{FR},\cite{KRTT}).
  The proof conducted here is inspired from the paper of  Chemin-Xu\ccite{chemin16}, and adapted previously in other contexts (see for instance~\cite{bgx},~\cite{bahouri gallagher} for the Heisenberg group).
 \begin{theorem}
\label {generalSobolevembedL2}
{\sl
For any real number $s$ in $[0,Q/2[,$ there exists a constant $C_s$  such that 
the following inequality holds:
$$
\forall u \in \dot H^s(G)\,,\ \|u\|_{L^p(G)} \leq C_s\|u\|_{\dot H^s(G)} \with p = \frac {2Q} {Q-2s}\,\cdotp
$$
}
\end{theorem}
\begin{proof}
Let us assume that~$\|u\|_{\dot H^s(G)}=1$, and  compute the $L^p$ norm 
according to the Cavalieri principle: 
\[
\|u\|_{L^p(G)}^p  =  p  \int_0^\infty \beta^{p-1} \bigr|(|u|>\beta)\bigr|
d\beta\, .
\]
In order to go further, we shall use the    technique of decomposition into low and high frequencies, namely we  shall decompose,  for all $A>0,$  the function $u$ into two parts as follows
\begin{equation}\label{eq:dec}u = u_{\ell, A} + u_{h,A}\quad\hbox{with}\quad 
\cF (u_{\ell, A}) (n, m, \nu, \lam)\eqdefa \cF (u) (n, m, \nu, \lam)\, {\mathbf 1}_{E_m(\nu, \lam) \leq A^2}  \, .\end{equation}
{}From the inversion formula\refeq{inverseFourierH}, the definition 
of the Sobolev norm \eqref{SobH}  and the Cauchy-Schwarz inequality,   we get
$$ \|u_{\ell, A}\|_{L^\infty(G)}\leq C \|u_{\ell, A}\|_{\dot H^s(G)} \Big(\int_{E_m(\lambda, \nu) \leq A^2}  |\cW((n, m, \nu, \lam), x^{-1}) |^2 (E_m(\nu, \lam))^{-s} \, d\wh x\Big)^{\frac 1 2}\, .$$
In view of  \eqref{eq: usew},  we have 
$$\begin{aligned} \int_{E_m(\lambda, \nu) \leq A^2}  |\cW((n, m, \nu, \lam),x^{-1}) |^2 (E_m(\nu, \lam))^{-s}  \, d\wh x &= \!\sum_{m\in \N} \! \int_{E_m(\nu, \lam) \leq A^2}  (E_m(\nu, \lam))^{-s} d\nu d\lam\, .  \end{aligned} $$
Then applying Formula\refeq{magic formula introduction} with $F(r)= r^{-s} {\bf 1}_{]0, A^2]}  (r)$,  we deduce that
$$ \|u_{\ell, A}\|_{L^\infty(G)} \leq C A^{\frac  Q 2 -s}\, .$$
Thus,   choosing $A=A_\beta= c \beta^{\frac p Q}$ for some small enough  positive real number $c$ 
 ensures that
 $$\bigl|(|u|>\beta)\bigr| \leq \bigl|(|u_{\ell,A_\beta}|>\beta/2)\bigr|+  \bigl|(|u_{h,A_\beta}|>\beta/2)\bigr| =  \bigl|(|u_{h,A_\beta}|>\beta/2)\bigr|\, ,$$
which thanks to Bienaym\'e-Tchebitchev inequality yields
\[
\|u\|_{L^p(G)}^p  \lesssim  \int_0^\infty \beta^{p-3} \, \norm {u_{h,
A_\beta}}{L^2(G)}^2\,d\beta\, .\]
Hence, by virtue of   Fourier-Plancherel  formula\refeq{FPH},    
$$
\norm u {L^p(G)}^p  \lesssim
C\int_0^\infty \beta^ {p-3} \int_{E_m(\nu, \lam)\geq A_\beta^2}|\cF(u)(\wh x)|^2\,d\wh x\,d\beta\,,
$$
which completes the proof  thanks to the  Fubini theorem. 
\end{proof}
Notice that refined versions of Sobolev embeddings (see~\cite{gerard,gerardmeyeroru}) can be obtained by slightly adapting the above proof. We shall not pursue this further here.   \medbreak

 The following theorem is to be compared with the Poincar\'e inequality.
\begin{theorem}
\label {poincareth}
{\sl
Let~$s$ be a nonnegative real number and~$K$ a subset of $G$ with finite measure. 
There exists a positive constant~$C(s, K)$ 
such that for  all functions $u$
in the subspace~$H^s_K(G )$ of functions in $H^s(G)$ with  support in~$K,$
we have $$
  \norm u {\dot H^s(G)} \leq
\|u\|_{H^s(G)} \leq C(s, K) \norm u {\dot H^s(G)}\, .
$$
}
\end{theorem}
\begin{proof}
 The first inequality is obvious, and  in view of Fourier-Plancherel formula\refeq{FPH} the second one amounts  to prove that
$$
\|\cF(u) \|_{L^2( \widehat{G} )} \leq  (C|K|)^\frac{2s}Q \norm u {\dot H^s(G)}\, .
$$
 To this end,  let us   again decompose~$u$ into low and high frequencies as in~(\ref{eq:dec}). We thus set~$\e>0$ and write
 $$
 \begin{aligned}
  \|\cF(u) \|^2_{L^2( \widehat{G} )} &=  \|\cF(u_{h, \e}) \|^2_{L^2( \widehat{G} )} +  \|\cF(u_{\ell, \e}) \|^2_{L^2( \widehat{G} )}  \\
  &=  \int_{E_m(\nu, \lam) \geq \e^2} E_m(\nu, \lam) ^{-2s}
 E_m(\nu, \lam) ^{2s}|\cF(u) (n, m,\nu, \lam)|^2\,d \wh x+  \|\cF(u_{\ell, \e}) \|^2_{L^2( \widehat{G} )}. 
  \end{aligned} $$   
  The first integral may be bounded by $\e^{-4s}\|u\|_{\dot H^s(G)}^2.$
 To handle the second one, we first take advantage of\refeq{eq: useful} which since  $u$  is compactly supported gives rise to $$\sum_{n\in \N} |\cF(u) (n, m, \nu, \lam) |^2 \leq \|u\|^2_{L^1(G)}\leq |K| \,\|u\|^2_{L^2( G)}= (2\pi)^{-3}  |K| \|\cF(u) \|^2_{L^2( \widehat{G} )}\, .$$ 
Thanks to Identity\refeq{magic formula introduction} with $F(x)=  {\bf 1}_{]0, \e^2]}  (x)$,   this implies  that 
 $$  \|\cF(u_{\ell, \e}) \|^2_{L^2( \widehat{G} )} \leq (2\pi)^{-3}   |K|  \, \|\cF(u) \|^2_{L^2( \widehat{G} )} \e^{Q} \, .$$
We deduce  that 
 $$
 \|\cF(u) \|^2_{L^2( \widehat{G} )}  \leq \frac 1 {\e^{4s}} 
\|u\|_{\dot H^s(G)}^2 + C |K| \e^{Q}  \|\cF(u) \|^2_{L^2( \widehat{G} )}  \,,
$$
which   achieves the proof of the result   choosing~$\e^Q = 1/(2C|K|)$. 
\end{proof}

\medbreak

Decomposing   functions  into  low and high frequencies is a key tool to  establish  functional inequalities,  but  also to 
investigate nonlinear Partial Differential Equations.  Let us      showcase  again the efficiency of this method by establishing the   
   Sobolev embedding  corresponding to the  critical exponent 
$s=Q/2.$ To this end, we introduce  the space~$BMO(G)$ of  locally integrable functions~$u$ on 
$G$  with Bounded Mean Oscillations: 
$$
\|u\|_{BMO(G )}\eqdefa \sup_{B}\frac 1 {|B|}\int_{ B}|(u-u_{B})(x)|\,dx<\infty
\with u_{B}\eqdefa  \frac 1 {|B|}\int_{B} u (x)\,dx,
$$
where the  supremum is taken over 
 all  balls ${B}$ of $G$, and where $|B|$ denotes the Lebesgue measure of the ball $B$.
\begin{theorem}
\label{inclussobbmoHb}{\sl
The space $L_{\rm loc}^1(G) \cap \dot H^{\frac Q 2}(G)$
is  included in~$BMO(G)$. Moreover,
there exists a constant~$C>0$ such that  
$$\|u\|_{BMO (G)} \leq C\|u\|_{\dot H^{\frac Q 2}(G)}\, ,$$
for all functions $u$ in $L_{\rm loc}^1(G) \cap \dot H^{\frac Q 2}(G)$.}
\end{theorem}
\begin{proof}
As previously, we use the decomposition\refeq{eq:dec}.
Then applying the Cauchy-Schwarz inequality, we infer  that   for any  
ball~$B$ in $G$, we have 
$$
\int_{ B} |(u-u_{ B})(x)| \frac {dx} {|B|} \leq \|u_{\ell,A}-(u_{\ell,A})_{ B}
\|_{L^2\big(B, \frac {dx}{| B|}\big)} +\frac 2 {|B|^{\frac 1
2}}\|u_{h,A}\|_{L^2(G )}.
$$
 In order to estimate  the low frequency part, we shall use the metric structure of the Engel  group:
recall that   for any~$(x,x')$ of~$G\times G$, there exists  a horizontal curve of~$\Omega_{x,x'}$ joining~$x$
to~$x'$. 
Using the Carnot-Carath\'eodory distance~$d_G$ defined in \eqref{eq:dsr},  we infer   that, for any ball~$B_{R}$ of~$G$ of radius~$R$, there holds 
\begin{equation}\label{eq:lfBMO}
\|u_{\ell,A}-(u_{\ell,A})_{ B_R}
\|_{L^2\big(B_R, \frac {dx}{| B_R|}\big)}    \lesssim  R A \|u\|_{\dot H^{\frac Q 2}(G )}\,.\end{equation}
Indeed, by definition of $(u_{\ell,A})_{ B_R},$ we have 
$$ (u_{\ell,A}-(u_{\ell,A})_{B_R})(x)=   \frac 1 {|{B_R}|}\int_{ {B_R}} (u_{\ell,A}(x) - u_{\ell,A}(x') ) dx' \, .$$
Since  for any curve~$\gamma =
\ds \bigcup_{p=1}^{p=J}\gamma_p$ where~$\gamma_p: [0, T_p]
\longrightarrow G $ belongs to $\Omega_{x_p, x_{p+1}}$, namely $$
\partial_t \gamma_p(t)= \pm X_i(\gamma_p(t)), \quad
\gamma_p(0)= x_p, \quad \gamma_p(T_p)= x_{p+1},$$ $$\gamma_1(0)=
x, \quad \gamma_J(T_J)= x'\quad\mbox{for}\quad p=1,\dots,J-1 \andf i \in \{1, 2\} \, ,$$
there holds 
$$ u_{\ell,A}(x_{p+1}) -  u_{\ell,A}(x_{p})= \int_0^{T_p}
\partial_t(u_{\ell,A}(\gamma_p(t)))\,dt, $$
we readily gather that  \begin{equation}\label{eq:est1} | u_{\ell,A}(x) - u_{\ell,A}(x')| \lesssim   {d}_G(x,x')\sup_{i\in \{1,2\} } \|X_i
 u_{\ell,A}\|_{L^{\infty}(G)}\,.
\end{equation}
We deduce    that
$$\|u_{\ell,A}-(u_{\ell,A})_{ {B_R}}
\|_{L^2\big({B_R}, \frac {dx}{| {B_R}|}\big)} \lesssim R \sup_{i\in \{1,2\}} \|X_i
 u_{\ell,A}\|_{L^{\infty}(G)}\, .$$
By the inversion formula\refeq{inverseFourierH}, we have 
$$\|X_i
 u_{\ell,A}\|_{L^{\infty}(G)} \lesssim \int_{E_m(\lambda, \nu) \leq A^2}  |\cW((n, m, \nu, \lam),x^{-1}) | | \cF(X_i u ) (\wh x)|  d \wh x.$$
Then combining the Cauchy-Schwarz inequality together with\refeq{magic formula introduction} and\refeq{eq: usew},  we get
$$\|X_i
 u_{\ell,A}\|_{L^{\infty}(G)} \lesssim A \|u\|_{\dot H^{\frac Q 2}(G)},$$
 which ensures that  $$\|u_{\ell,A}-(u_{\ell,A})_{ {B_R}}
\|_{L^2\big({B_R}, \frac {dx}{| {B_R}|}\big)}   \lesssim     RA \|u\|_{\dot H^{\frac Q 2}(G)} \, \cdotp$$
To bound $u_{h,A},$ we combine  Identity\refeq{FPH} with Formula\refeq{magic formula introduction}, which gives rise to
 $$\begin{aligned} \|u_{h,A}\|^2_{L^2(G)} \leq C
  A^{- Q} \,  \|u\|^2_{\dot H^{\frac Q 2}(G)} \, .\end{aligned}$$
We know by \eqref{eq:homball}  that 
   $ | {B_R}| = C R^Q $. Then, the latter estimate implies that  
   \begin{equation}\label{eq:lfBMOh}\frac 2 {|{B_R}|^{\frac 1
2}}\|u_{h,A}\|_{L^2(G )} \lesssim  (AR)^{- \frac Q 2} \,  \|u\|_{\dot H^{\frac Q 2}(G)}\, .\end{equation} 
Gathering the estimates\refeq{eq:lfBMO} and \eqref{eq:lfBMOh} and 
choosing~$A=R^{-1}$  completes the proof of the result.
\end{proof}

\medbreak 
One can also prove, by a similar high-low decomposition technique,  embeddings between H\"older spaces and Sobolev spaces. 
 To ease the notations, we denote in what follows~${\cX} = (X_1, X_2)$ the family of  horizontal left-invariant vector fields on~$G$, and   we set   for any multi-index~$\alpha$ in~$\{1,2\}^k$:
$${\mathcal \cX}^{\alpha} = \prod_{j=1}^kX_{\alpha_j} 
\, .$$ 
Then we denote by~$C^{k,\rho}(G)$, (for~$(k,\rho)$  in~$\N\times]0,1]$) the
H\"older space on the Engel group,  consisting in  functions~$u$
on~$G$ such that $$ \|u\|_{C^{k,\rho}(G)} \eqdefa
 \sup_{|\al|\leq k}
\Bigl(
\norm {\cX^\al u} {L^\infty}
+\sup_{x\not = y}
\frac{|\cX^\al u(x)-\cX^\al u(y)|}{\wt {d}(x,y)^\rho}
\Bigr)<\infty.
$$ 
We have the following   result.
\begin{proposition}
\label{inclusionsobolevholder} {\sl If~$s>\frac
Q 2$ and~$s-\frac Q 2$ is not an integer,  then the
space~$H^s(G)$ is included  in the 
H\"older  space of index
$$ (k,\rho) =\biggl(\Bigl[ s-\frac Q 2 \Bigr], s-\frac
Q 2- \Bigl[ s-\frac Q 2 \Bigl]\biggr)
$$
and we have for all $u\in  H^s(G),$
$$
\|u\|_{C^{k,\rho}(G)}  \leq C_{s}\|u\|_{H^s(G)}.$$
}
\end{proposition}
\begin{proof}
We prove the result only in the case when the integer part 
of~$s-Q/2$ is~$0$, namely the case when $s=\frac Q 2 +\rho$, with $0< \rho <1$. Using again Decomposition\refeq{eq:dec},   
we infer, according to \eqref{eq:est1},  that the low frequency part of~$u$ satisfies $$| u_{\ell,A}(x) - u_{\ell,A}(x')| \lesssim  {d}_G(x,y)\sup_{i\in \{1,2\} } \|X_i
 u_{\ell,A}\|_{L^{\infty}(G)}\,.
$$
In view of Formula\refeq{magic formula introduction}, one gets 
$$ \|X_i
 u_{\ell,A}\|_{L^{\infty}(G)} \leq \|u\|_{\dot H^{s}(G)} \Big(\sum_{m\in \N }  \int_{E_m(\lambda, \nu) \leq A^2} E_m(\nu, \lam)^{-s+1}  d\nu  d\lam\Big)^{\frac 1 2}  \lesssim A^{1- (s-\frac Q 2)} \|u\|_{\dot H^{s}(G)}  ,$$
which implies that
$$| u_{\ell,A}(x) - u_{\ell,A}(x')| \lesssim   {d}_G(x,y)A^{1- (s-\frac Q 2)} \|u\|_{\dot H^{s}(G)}.$$
Along the same lines, we obtain 
$$\|u_{h, A}\|_{L^\infty(G)}\leq \|u\|_{H^s(G)} \Big(\!\sum_{m\in \N } \! \int_{E_m(\nu, \lam) \geq A^2}     E_m(\nu, \lam)^{-s}    d\nu  d\lam\Big)^{\frac 1 2}  \lesssim A^{\frac Q 2-s}\|u\|_{H^s(G)}  \, .$$
Consequently
\begin{eqnarray*}
|u(x)-u(y)| 
& \lesssim &  \left({d}_G(x,y)A^{1- \rho} +A^{-\rho}\right)\|u\|_{\dot H^{s}(G)}\, .  \end{eqnarray*}
Choosing~$A={d}_G(x,y)^{-1}$, we conclude the proof of the result.
\end{proof}

 \subsection{Bernstein inequalities}\label{Bernstein}

 Similarly to the Euclidean case, Formula\refeq{actionsubf} allows to give a definition of a function whose Fourier transform is compactly supported, in the following way.  
\begin{definition} \label{freqlocgen}
 {\sl We say that a   function~$u$ in $L^2(G)$ is frequency localized in a ball~${\mathcal B}_\Lambda$ centered at~$0$  of radius~$\Lambda>0$ if
 $$
 \mbox{Supp} \,  \cF (u)
 \subset \Big\{\wh x=(n,m, \nu, \lam) \in \widehat{G}
 \, / \, E_m(\nu, \lam) \leq \Lambda^2 
 \Big\} \, .$$
Similarly we say that a   function~$u$ on~$G$ is frequency localized in a ring~${\mathcal C}_\Lambda$ centered at 0  of  small radius~$\Lambda/2$ and large radius~$\Lambda$ if $$
 \mbox{Supp} \,  \cF (u)
 \subset \Big\{\wh x=(n,m, \nu, \lam) \in \widehat{G}
 \, / \,  \frac{\Lambda^2 }4
\leq E_m(\nu, \lam) \leq \Lambda^2 
 \Big\} \, .$$
 }
\end{definition}
 \begin{remark}
{\sl Equivalently, $u$ in $L^2(G)$ is frequency localized in~${\mathcal B}_\Lambda$  if there exists a  function~$\psi$ in $\cD(\R)$   supported in~${\mathcal B}_1$, valued in the interval~$[0, 1]$   and equal to~$1$ near~$0$ such that 
 for any~$\wh x=(n,m, \lambda, \nu)$ in~$\widehat{G}$,
\begin{equation}\label{freqlocballFourier}
 \cF (u)(n,m, \nu, \lam) =  \cF (u)(n,m, \nu, \lam)\,\psi (\Lam^{-2} E_m(\nu, \lam))  \, .
\end{equation}
Similarly~$u$  is frequency localized in~${\mathcal C}_\Lambda$   if there exists a     function~$\phi$ in~$\cD(\R)\setminus\{0\}$  valued in the interval~$[0, 1]$ and supported in~${\mathcal C}_1$   such that
 for any $\wh x=(n,m, \nu, \lam)$ in $\widehat{G}$,
\begin{equation}\label{freqlocringFourier}
 \cF (u)(n,m, \nu, \lam)=  \cF (u)(n,m, \nu, \lam)\,\phi (\Lam^{-2} E_m(\nu, \lam))  \, .
\end{equation}
}
\end{remark}

This definition allows classically to recover equivalent definitions of Sobolev and H\"older spaces via the well-known Littlewood-Paley decomposition, and to define generalizations of those spaces known as Besov spaces; these turn out to be very important tools, namely to refine    Sobolev inequalities, and to study nonlinear PDEs. For an introduction to this topic, we refer the reader for instance to\ccite{bahouri lp}. We shall not pursue further this line of investigation here, but only prove   the following proposition, known as the Bernstein inequalities. The proof of this result is inspired by the corresponding result on the Heisenberg group in the monograph of Bahouri-Chemin-Danchin\ccite{bcdbookh} -- we refer also to~\cite{bahouri gallagher} for the easier case~(\ref{eq:lech1}).
 \begin{proposition}
\label  {bernstein}
{\sl  With the above notation,  
\begin{itemize}
\item if~$u$  is frequency
localized in~${\mathcal B}_\Lambda$, 
  then for all $1 \leq p \leq q \leq \infty$,      $k\in\N$ and $\beta\in\N^{2}$ with~$|\beta|=k,$  there exists a constant~$C_k$ depending only on~$k$ such that
\begin{equation}
\label{eq:lech1}  \norm{\cX^\beta u}{L^q(G)}\leq C_k \Lam^{k+Q (\frac{1}{p}-\frac{1}{q}) } \norm
u{L^p(G)} \, .
\end{equation}
\item if~$u$  is frequency
localized in~$\cC_\Lambda$,  then for all $ p \geq 1$,   $k\in\N$ and $\beta\in\N^{2}$ with $|\beta|=k,$  there exists a constant~$C_k$ depending only on~$k$ such that
 \begin{equation}
\label{eq:lech3} 
  \norm u {L^p(G)} \leq C_k \Lambda^{-k} \sup_{|\beta|=k}
 \norm{\cX^\beta u}{L^p(G)} \, .
\end{equation}
 \end{itemize}}
\end {proposition}

  \begin{remark}
{\sl Spectral truncations are  convenient means of approximating functions.  Indeed  Proposition~{\rm\ref{bernstein}} shows  that for any~$u \in L^2(G)$ for instance, $ \psi(- \Lambda^{-2}\D_G) \, u $ belongs to~$ H^s(G)$ for any~$s \geq 0$, and, as a consequence of the Plancherel formula,~$ \psi(- \Lambda^{-2}\D_G) \, u $ converges to~$u$ in~$L^2(G)$.}
\end{remark}

\begin{proof}[Proof of Proposition~\rm\ref{bernstein}]
By density and to make sense of the next computations, we assume   that~$u $ belongs to~$ \cS(G)$. First,  we notice that~(\ref{freqlocballFourier}) and~(\ref{freqlocringFourier}) can be restated respectively as
$$
u  = \psi(- \Lambda^{-2}\D_G) \, u \quad \mbox{and} \quad
u  = \phi(- \Lambda^{-2}\D_G) \, u \, .
$$
In view of Hulanicki's result\ccite{hul},  there exist   functions~$h_\psi $ and $h_\phi $ in $\cS(G)$ such that, for all~$u $ in~$ \cS(G)$, there holds
\begin{equation}
\label{eq:hula*}  \psi(- \Lambda^{-2}\D_G) \, u = u \star \Lambda^{Q} (h_\psi \circ \d_{\Lambda}) \andf \phi(- \Lambda^{-2}\D_G) \, u = u \star \Lambda^{Q} (h_\phi \circ \d_{\Lambda})\,.
\end{equation}
Let us prove that the functions~$h_\phi$ and~$h_\psi$ are even,
 that is to say, for all $x \in G$, $$h_\psi(x)=h_\psi(x^{-1}) \andf h_\phi(x)=h_\phi(x^{-1})\, .$$ Since the analysis of $h_\phi$ is similar to that of $h_\psi$, we limit here ourselves to the case of $h_\psi$. By definition 
\begin{equation}
\label{eq:usef**}
 \cF (h_\psi)(n,m, \nu, \lam) = \psi (E_m(\nu, \lam))  \delta_{m,n}\, ,
\end{equation}which in view of the inversion Fourier formula\refeq{inverseFourierH} implies that $$h_\psi(x) = (2\pi)^{-3}  \sum_{m\in \N}  \int_{\R\times \R^*}
\cW\big((m,m, \nu, \lambda), x^{-1}\big) \psi (E_m(\nu, \lam))  \, d\nu d\lambda \,. $$  Also, $$h_\psi(x^{-1}) = (2\pi)^{-3}  \sum_{m\in \N}  \int_{\R\times \R^*}
\cW\big((m,m, \nu, \lambda), x\big) \psi (E_m(\nu, \lam))  \, d\nu d\lambda \,. $$But, in view of\eqref{relationeven}-\eqref{relationcWinverse},  we have
$$\cW\big((m,m, \nu, \lambda), x^{-1}\big)=\cW\big((m,m, \nu, -\lambda), x\big) \,, $$
which gives the result since $E_m(\nu, -\lam)=E_m(\nu, \lam)$.

Let us return to the proof of the proposition.  We observe that by scale invariance, it is enough to prove the proposition for~$\Lambda=1$. In order to establish\refeq{eq:lech1},  we first combine  Definition\refer{freqlocgen} with Identity\refeq{eq:hula*} which implies that  
$$ u= u \star h_\psi  \, ,$$
with $h_\psi \in \cS(G)$. 
Invoking\refeq{convderv}, we infer that 
$$\cX^\beta u=u \star  \cX^\beta h_\psi \, ,$$
which leads to the result, thanks to Young's inequalities\refeq{young}.

\medskip Let us turn to the case when~$u$  is frequency
localized in a unit ring: we use~(\ref{eq:hula*})  
again, and notice that
 $$ \D_G h_\phi=  \wt \D_G h_\phi \, .$$ 
 Moreover, since $h_\phi$ is frequency localized away from the origin,  for any integer $k$,  one has
\begin{equation}\label{eq:crucial}u=  u \star h_\phi = u \star (-\D_G)^k  h^k_\phi= u \star (- \wt \D_G)^k  h^k_\phi \, , \end{equation}
where $h^k_\phi$ is the even Schwartz class function defined by 
$$ \cF (h^k_\phi)(n,m, \nu, \lam)\eqdefa \big(E_m(\nu, \lam)\big)^{-k}\phi \big(E_m(\nu, \lam)\big)  \delta_{n, m}\, .
$$
We claim that for any $u, v  \in  \cS(G)$ and all~$i \in \{1,2\}$,   one has
\begin{equation}\label{eq:rightleftconv}  X_i u\star   v= u \star \wt X_i  v\, .\end{equation}
Indeed by definition,  one has 
$$
 (X_1 u) \star   v(x)  =  \int_{G} (\partial_{y_{1}} u) (y) v( y^{-1} \cdot x)\, dy = - \int_{G}  u (y) \partial_{y_{1}}\big( v( y^{-1} \cdot x)\big)\, dy 
$$
by integration by parts, and since $$
 y^{-1} \cdot x =  \big(- y_1+x_1,- y_2+x_2, -y_3+x_3-y_1(-y_2+x_2), -y_4 +x_4-y_1(-y_3+x_3)+\frac 1 2 y_1^2 (-y_2+x_2)\big) \, , 
$$
and by Remark~\ref{deftildeX}
$$
\tilde X_1 = \partial_{x_{1}}+x_{2}\partial_{x_{3}}+x_{3}\partial_{x_{4}}
$$
then we obtain  
$$     (X_1 u) \star   v(x)   =   \int_{G}  u (y) (\wt X_1v) ( y^{-1} \cdot x)\, dy$$
whence~(\ref{eq:rightleftconv}) for~$i=1$. 
Along the same lines, since 
$$
(X_2 u) \star   v(x)  =  \int_{G} \left(\left(\partial_{y_{2}}+y_{1}\partial_{y_{3}}+\frac{y_{1}^{2}}{2}\partial_{y_{4}} \right)u\right) (y) v( y^{-1} \cdot x)\, dy \, , 
$$
performing an integration by parts, we get 
$$     (X_2 u) \star   v(x)   =   \int_{G}  u (y) (\wt X_2 v) ( y^{-1} \cdot x)\, dy \, , $$
from which~(\ref{eq:rightleftconv}) also follows  for~$i=2$.  
Then invoking\refeq{eq:crucial},  
we deduce that  
$$u=  u \star (- \wt \D_G)^k  h^k_\phi= \sum^2_{i=1} X_i u \star \wt X_i (- \wt \D_G)^{k-1}  h^k_\phi \, . $$
 By induction, we obtain  
$$u=  \sum_{|\beta|=k} \cX^\beta u \star \Phi_{\beta, k}\, ,$$
for some functions $\Phi_{\beta, k}$ in $\cS(G)$. 
This completes the proof of the proposition thanks to Young's inequalities\refeq{young}.
\end{proof}

\subsection{Application to the heat equation}\label{heat}

 The heat kernel of the sub-Riemannian Laplacian has been object of several investigations in the last decades, both from the analytic and geometric viewpoints. We refer the reader to~\cite{BBN12,BJ,bGR,SM} and references therein for a complete discussion. 
In this paragraph we show the efficiency of Formula~\eqref{magic formula introduction} by analyzing the heat kernel on the Engel group. 

 It is well-known that this kernel  is a Schwartz class function; see for instance\ccite{DVH}  and the references therein.  Here we   show in an elementary way, thanks to~\eqref{magic formula introduction}, that it belongs to~$\ds  H^s(G)$ for any~$s \geq 0$.
 As already mentioned in the introduction,  the Fourier transform $U$ given by Theorem\refer{t:key} allows to compute explicitly the solutions of evolution equations associated with~$-\D_G$. For instance, if we consider the   heat equation on $G$ $$
(H_G)\qquad \left\{
\begin{array}{rcl}
\ds \partial_t u -  \D_G u & =  &0\\
{ u}{}_{|t=0} &= & u_0\, ,
\end{array}
\right.
$$   applying the Engel Fourier transform and  taking advantage of the identities\refeq{actionsubf}, then integrating in time the resulting ODE, we deduce that, for  all $\wh x=(n,m, \nu, \lam)$ in $\widehat{G}$, 
 $$\cF(u(t)) (n, m, \nu, \lam)= e^{-t{E}_m(\nu, \lam) }\cF(u_0) (n, m, \nu, \lam)\, .$$  
Invoking the Fourier inversion formula \eqref{inverseFourierH} along with the convolution identity\refeq{newFourierconvoleq1}, we infer  that 
\beq 
\label{defhkernel1} u(t,\cdot)= u_0 \star  h_t\with \cF(h_t) (n, m, \nu, \lam)=e^{-t{E}_m(\nu, \lam)}\delta_{n,m} \,. \eeq
Then according to the scaling property\refeq{eq: scaledil}, the heat kernel $h_t$ is given,  for all~$t>0$,  by
\beq 
\label{defhkernel}
h_t = \frac 1 {t^{\frac Q 2}} (h \circ \delta_{\frac 1 t } )\with h(x)= (2\pi)^{-3}  \int_{\hat G} 
  e^{- {E}_m(\nu, \lam)}  \cW\big((n, m, \nu, \lam), x^{-1}\big) \delta_{m,n} d \hat x \,.
 \eeq
 
Thanks to Formula\refeq{magic formula introduction}, we deduce that the heat kernel   on $G$ belongs to~$\ds \cap_s H^s(G)$. 
Indeed combining\refeq{SobinH} together with\refeq{defhkernel},  we infer that  for all $s \in \R$, 
$$\|h\|^2_{H^s(G)}=  (2\pi)^{-3}   \sum_{m\in \N} \int_{\R\times\R^*} F_s({E}_m(\nu, \lam)) d\nu d\lam \,,$$
where  $F_s(r)\eqdefa(1+r)^{s} e^{- 2 r}$   which  ensures the result.

 \section{Metric structure on the   frequency set $\wh G$}\label{topofrequency}
The aim of this paragraph is to endow the frequency set $\wh G=\N^{2}\times \R\times \R^*$ with a distance.  To do so, we   have to keep in mind that, as in the Euclidean setting, we expect the Fourier transform to transform the regularity of functions on~$G$ into decay of the Fourier transform on~$\wh G$. So   first let us start by observing that in view of the relations\refeq{eq: 34}, \eqref{eq: nu} and\refeq{actionsubf}, one has 
  \begin{eqnarray*}
\cF(- \D_G  u) (n,m, \nu, \lambda) & =  & E_m(\nu, \lambda)\cF(u) (n,m, \nu, \lambda) \\  
\cF(-\wt \D_G u)(n,m, \nu, \lambda) & =  & E_n(\nu, \lambda)\cF(u) (n,m, \nu, \lambda) \\  \cF(X_4 u) (n,m, \nu, \lambda) &= & - i \lambda \cF(u) (n,m, \nu, \lambda)  \\   \cF\big(\big(X_{4}X_{2}-\frac12 X_{3}^{2}\big)u\big) (n,m, \nu, \lambda) &= &   \nu \cF(u) (n,m, \nu, \lambda)\, .  
  \end{eqnarray*}
Our aim now is to endow~$\wh G$ with  a  distance $\wh d$ in accordance with the above relations and which is moreover homogeneous of degree one with respect to the dilation  $\wh \d_a$ defined by\refeq{eq: dilationhat}.
     This motivates our definition of the distance $\wh d$ between two   elements $\wh x=(n, m, \nu, \lambda)$ and~$\wh x'=(n', m', {\nu}', {\lambda}')$ of the set $\wh G$ as follows: 
\begin{equation} \label{def:dist} \begin{aligned}
\wh d(\wh x,\wh x') & \eqdefa \big|E_m(\nu, \lambda) -E_{m'}({\nu}', {\lambda'})\big|^{\frac 1 2} +\big |(E_m-E_n)(\nu, \lambda)-(E_{m'}-E_{n'})({\nu}', {\lambda'})\big|^{\frac 1 2}  \\ & \quad +|\nu-{\nu}'|^{\frac 1 4} +|\lam-{\lambda}'|^{\frac 1 3}.
\end{aligned} \end{equation}
To check that $\wh d$ is a distance on~$\wh G$, the main point consists in proving that $$\wh d(\wh x,\wh x')= 0 \Rightarrow \wh x=\wh x' \, .$$  In view of\refeq{def:dist}, this amounts  
to showing that if for some integers $k, k'$, we have    $E_k(\nu, \lambda)=E_{k'}(\nu, \lambda)$, then $k=k'$. This  follows from the first item of Proposition\refer{p:def-E-psi} which asserts that the energy levels do not intersect. 
Now the fact that $\wh d$ is homogeneous of degree one with respect to the dilation~$\wh \d_a$, namely  for all $a>0$  $$\wh d(\wh \d_a \wh x, \wh \d_a\wh x') =a \wh d(\wh x,\wh x')\, ,$$   follows  from  the scaling property $E_k(\nu,\lambda) = |\lambda|^{2/3}\mathsf{E}_k\Big(\frac{\nu}{|\lambda|^{4/3}}\Big)$ (see\refeq{eq: relnumu}).

\medskip Since $\lambda$ belongs to $\R^*$, the set $(\wh G, \wh d)$ is not complete. Its  completion is described by  the following proposition.  
\begin{proposition}
\label {completionHtilde}
{\sl 
The completion of the set~$\wh G$  for the distance~$\wh d$ is the set  
$$
\wh G
 \cup \wh G_0 \with  \wh G_0 \eqdefa  \big( \R_+\times \R \times \{0_{\R^2}\} \big) \cup \big\{((2m+1)  \sqrt{2 \nu}, 2(m-n)\sqrt{2 \nu}, \nu, 0), \, \, (n,m,   \nu)\in \N^2 \times \R^*_+\big\} .$$
}
\end{proposition}
\begin{proof}
We denote by $S$ the completion of the set~$\wh G$  for the distance~$\wh d$, that is to say the set of all limits of Cauchy sequences $(n_p,m_p, \nu_p, \lam_p)_{p\in \N}$ in~$(\widehat{G},\wh d)$, and our goal is to prove that 
\begin{align*}
&S = \wh G
 \cup \wh G_0, \quad \text{ with } \quad  \wh G_0 = \wh G_{0,0} \sqcup  \wh G_{0,1} ,  \\
&\wh G_{0,0} =  \R_+\times \R \times \{0_{\R^2}\} , \quad   \wh G_{0,1} = \big\{((2m+1)  \sqrt{2 \nu}, 2(m-n)\sqrt{2 \nu}, \nu, 0), \, \, (n,m,   \nu)\in \N^2 \times \R^*_+\big\}. 
\end{align*}
 
\medskip

We first prove that $S \subset \wh G \cup \wh G_0$. Let $(n_p,m_p, \nu_p, \lam_p)_{p\in \N}$ be a Cauchy sequence in~$(\widehat{G},\wh d)$. Then~$\suite \nu p \N$ and~$\suite \lam p \N$  are Cauchy sequences of real numbers, and thus they converge respectively to some $\nu$ and $\lambda$ in $\R.$ Moreover,  $E_{m_p}(\nu_p,\lambda_p), E_{n_p}(\nu_p,\lambda_p)$ are Cauchy sequences in $\R_+$ and thus converge in $\R_+$: there exist $\dot x \in \R_+$ and $\dot y \in \R_+$ such that
\begin{equation}
\label{e:conv-em-en}
E_{m_p}(\nu_p,\lambda_p) \stackrel{p\to\infty} \longrightarrow \dot x , \quad E_{n_p}(\nu_p,\lambda_p) \stackrel{p\to\infty}\longrightarrow \dot y .
\end{equation}
Recalling the scaling relation~\eqref{eq: relnumu}, this reads $ |\lambda_p|^{2/3}\mathsf{E}_{m_p}(\mu_p) \stackrel{p\to\infty} \longrightarrow \dot x$ and $ |\lambda_p|^{2/3}\mathsf{E}_{n_p}(\mu_p) \stackrel{p\to\infty} \longrightarrow \dot y$, with~$\mu_p =  \frac{\nu_p}{|\lambda_p|^{4/3}}\cdotp$
\medskip
If~$\lambda \neq 0$, we have $\mu_p \stackrel{p\to\infty}\longrightarrow \frac{\nu}{|\lambda|^{4/3}} \in \R$ and
 $|\lambda_p|^{2/3}\mathsf{E}_{m_p}(\mu_p) \stackrel{p\to\infty}\longrightarrow \dot x$, $|\lambda_p|^{2/3}\mathsf{E}_{n_p}(\mu_p) \stackrel{p\to\infty}\longrightarrow \dot y$. As a consequence $\mathsf{E}_{m_p}(\mu_p) \to \dot x|\lambda|^{-2/3}$,  $\mathsf{E}_{n_p}(\mu_p) \to \dot y |\lambda|^{-2/3}$ and, according to Item~\ref{eigenvalue-ordering} in Proposition~\ref{p:def-E-fimu}, we infer that the sequences $(m_p)_{p\in \N}$ and $(n_p)_{p\in \N}$ are constant after a certain index. Therefore, there exist~$m$ and $n$  in $\N$ such that
 $$\wh d((n_p,m_p, \nu_p, \lam_p), (n,m, \nu, \lam))\stackrel{p\to\infty}\longrightarrow 0\, .$$ 
Consequently  (in that case) the limit of the sequence~$(n_p,m_p, \nu_p, \lam_p)_{p\in \N}$ in $(\widehat{G}, \wh d)$  belongs to $\widehat{G}$. 
 
\medskip 
We now consider the case $\lambda =0$, that is to say $\lam_p \to 0$, and recall that $\nu_p \to \nu \in \R$.
If $\nu<0$, recalling  that by\refeq{eq:neqbeh}, for any~$k$ and any $\mu <0$, one has $\mathsf{E}_k(\mu) \geq |\mu|^2$, we deduce that 
$$
|\lambda_p|^{2/3}\mathsf{E}_{m_p}\big(\frac{\nu_p}{|\lambda_p|^{4/3}} \big)\geq \frac{\nu_p^2}{|\lambda_p|^2} \stackrel{p\to\infty}\longrightarrow  +\infty \, , \quad 
|\lambda_p|^{2/3}\mathsf{E}_{n_p}\big(\frac{\nu_p}{|\lambda_p|^{4/3}} \big)\geq \frac{\nu_p^2}{|\lambda_p|^2} \stackrel{p\to\infty}\longrightarrow  +\infty \, ,
$$
which contradicts~\eqref{e:conv-em-en}. In the case $\lambda=0$, we thus necessarily have $\nu \geq 0$ and we distinguish the two cases $\nu=0$ and $\nu >0$.
\medskip
(i) Firstly if $\nu=0$, then according to~\eqref{e:conv-em-en}, 
$$
\left( E_{m_p}(\nu_p,\lambda_p)  ,
E_{m_p}(\nu_p,\lambda_p)-E_{n_p}(\nu_p,\lambda_p) ,
\nu_p , 
\lambda_p \right) 
\stackrel{p\to\infty}\longrightarrow  (\dot x, \dot x -\dot y , 0, 0) \in \hat G_{0,0} \, .$$

\medskip
(ii) Secondly if $\nu >0$, then $\nu_p>0$, for $p$ large, and  according to the scaling relation~\eqref{eq: relnumu} and~\eqref{e:conv-em-en}, $$E_{m_p}(\nu_p,\lambda_p) = \nu_p^{1/2} \mu_p^{-1/2}\mathsf{E}_{m_p}(\mu_p) \stackrel{p\to\infty}\longrightarrow  \dot x ,$$ with $\mu_p = \frac{\nu_p}{|\lambda_p|^{4/3}} \to + \infty$. As a consequence, $\mu_p^{-1/2}\mathsf{E}_{m_p}(\mu_p) \to \nu^{-1/2}\dot x$.
Setting $h_p = \mu_p^{-3/2} \to 0$ and performing the change of scales~\eqref{e:Pmu-Ph}, it follows that $E_{m_p}(h_p) = \mu_p^{-2}\mathsf{E}_{m_p}(\mu_p)$ is an eigenvalue of the semiclassical Schr\"odinger operator in~\eqref{e:def-semiclassic-2-well}.
Since \begin{equation}
\label{e:interm-toto}
h_p^{-1} E_{m_p}(h_p)= \mu_p^{-1/2}\mathsf{E}_{m_p}(\mu_p) \stackrel{p\to\infty}\longrightarrow  \nu^{-1/2}\dot x \,,
\end{equation}
 there exists $\beta>0$ such that $h_p^{-1} E_{m_p}(h_p) \leq \beta$ for all $p \in \N$. 
Lemma~\ref{l:helffer-sjostrand} implies first the existence of $N_\beta >0$ such that $m_p \leq N_\beta$ for all $p \in \N$, and second that there exists $m\in \N$ with $m \leq N_\beta$ such that \begin{equation} \label{behash} h_p^{-1} E_{m_p}(h_p) \to \sqrt{2}(2m+1), \, \,  \mbox{as} \, \,  h \to 0^+.\end{equation} 
Combining\refeq{behash} and~\eqref{e:interm-toto} yields $\dot x= \sqrt{2\nu}(2m+1)$. 

The same method applies to the sequence~$E_{n_p}(\nu_p,\lambda_p)$ yielding existence of $n\in\N$ such that~$\dot y= \sqrt{2\nu}(2n+1)$, and we finally obtain 
$$
\left(E_{m_p}(\nu_p,\lambda_p),
E_{m_p}(\nu_p,\lambda_p)-E_{n_p}(\nu_p,\lambda_p) ,
\nu_p \, , 
\lambda_p \right) 
\stackrel{p\to\infty}\longrightarrow (\sqrt{2\nu}(2m+1), 2(m-n)\sqrt{2 \nu}, \nu , 0) \in \hat G_{0,1}  \,.$$
This concludes the proof of $S \subset \wh G \cup \wh G_{0,0}\cup \wh G_{0,1}$.

\bigskip
We now prove the converse statement, that is $S \supset \wh G \cup \wh G_{0,0}\cup \wh G_{0,1}$. 
If $(n,m, \nu,\lambda) \in \wh G$, that is to say with $\lambda \neq 0$, then the constant sequence  $(n,m, \nu,\lambda)$ converges to  $(n,m, \nu,\lambda)$ in $\wh G$.

If $(\dot x, \dot y, 0, 0) \in \wh G_{0,0}$, we claim that there exists a Cauchy sequence $(n_p,m_p, \nu_p,\lambda_p) \in \wh G$ such that  
\begin{equation}
\label{e:seq-to-to-to}
\lambda_p \to 0 , \quad \nu_p \to 0 , \quad E_{m_p}(\nu_p,\lambda_p) \to \dot x, \quad E_{n_p}(\nu_p,\lambda_p) \to \dot y \, .
\end{equation}
Indeed, if $\dot x = 0$ and $\dot y = 0$, then, choose $\nu_p=0$, $m_p=n_p = 1$ and any sequence $\lambda_p \to 0$. We then have  
$$
E_{m_p}(\nu_p,\lambda_p) = |\lambda_p|^{2/3}\mathsf{E}_1(0) \to 0 , \quad E_{n_p}(\nu_p,\lambda_p) = |\lambda_p|^{2/3}\mathsf{E}_1(0) \to 0  \,.
$$
Otherwise, either $\dot x \neq 0$ or $\dot y \neq 0$. Assume for instance $\dot x \neq 0$, and recall  that by virtue of  Lemma~\ref{l:asympt-k-grand} one has~$\mathsf{E}_{k}(0) \sim \left(\frac{2\pi}{\Vol_1} k \right)^{4/3}$ as $k\to+\infty$. Then  applying Lemma~\ref{l:suites} below  to the sequence~$u_k =\mathsf{E}_{k}(0)$, we  infer that there exist sequences $(m_p)_{p \in \N},(n_p)_{p \in \N} \in \N^\N$ such that $\ds \frac{\mathsf{E}_{n_p}(0)}{\mathsf{E}_{m_p}(0)} \to \frac{\dot y}{\dot x}$. Setting then $\nu_p=0$ and $\lambda_p = \left(\frac{\dot x}{\mathsf{E}_{m_p}(0)} \right)^{3/2}$, we have
$$
\begin{aligned}
&E_{m_p}(\nu_p,\lambda_p) = |\lambda_p|^{2/3}\mathsf{E}_{m_p}(0) = \dot x \, , \\
 &   E_{n_p}(\nu_p,\lambda_p) = |\lambda_p|^{2/3}\mathsf{E}_{n_p}(0)  = |\lambda_p|^{2/3}\mathsf{E}_{m_p}(0) \frac{\mathsf{E}_{n_p}(0) }{\mathsf{E}_{m_p}(0)} \stackrel{p\to\infty}\longrightarrow \dot x \frac{\dot y}{\dot x} = \dot y \, ,
\end{aligned}
$$
 which proves the existence of a Cauchy sequence satisfying~\eqref{e:seq-to-to-to}, and thus $\wh G_{0,0} \subset S$.

Finally, let $(\sqrt{2\nu}(2m+1), 2(m-n)\sqrt{2 \nu}, \nu , 0) \in \hat G_{0,1}$, that is to say $\nu >0$ and $m,n\in \N$. First, choosing $\beta> 2\sqrt{2}(2\max\{m,n\}+1)$, Lemma~\ref{l:helffer-sjostrand} implies that 
$$
h^{-1} E_{2m}(h) \to \sqrt{2}(2m+1) , \quad h^{-1} E_{2n}(h) \to \sqrt{2}(2n+1) \,, \quad \text{ as } h \to 0 \, . $$
Second, we fix any sequence $\lambda_p \to 0$ and notice that $\mu_p \eqdefa \frac{\nu}{|\lambda_p|^{4/3}} \to +\infty$. 
As already noticed in~\eqref{e:interm-toto}, we have with $h_p\eqdefa \mu_p^{-3/2} \to 0$ 
$$E_{m}(\nu,\lambda_p) = \sqrt{\nu} \mu_p^{-1/2}\mathsf{E}_{m}(\mu_p) = \sqrt{\nu} h_p^{-1} E_{m}(h_p) \stackrel{p\to\infty}\longrightarrow   \sqrt{2\nu}(2m+1)  \,,$$
and similarly $E_{n}(\nu,\lambda_p) \to \sqrt{2\nu}(2n+1)$. This proves the existence of a Cauchy sequence satisfying 
$$
\lambda_p \to 0 , \quad \nu_p \to \nu , \quad E_{m_p}(\nu_p,\lambda_p) \to \sqrt{2\nu}(2m+1), \quad E_{n_p}(\nu_p,\lambda_p) \to\sqrt{2\nu}(2n+1) \,, 
$$
and thus $\wh G_{0,1} \subset S$. This concludes the proof of the proposition.
\end{proof}
\begin{lemma}
\label{l:suites}
{\sl Assume that $(u_n)_{n \in \N}$  in $(0,\infty)^\N$ is such that $u_n \to + \infty$ and $\frac{u_{n+1}}{u_n} \to 1$, as $n \to \infty$. Then, the set $\big\{\frac{u_{m}}{u_n}, \, \, (n,m)\in \N^2\big\}$ is dense in $\R_+$. }
\end{lemma}
\begin{proof} 
Our purpose is  to prove that   for all $\ell \geq 0$ and all $\epsilon > 0$, there exists $(n,m)\in \N^2$ such that~$\ds \big|\frac{u_{m}}{u_n}-\ell\big| \leq \epsilon$. We shall argue according to the value of $\ell$.

\smallskip
(i) The result is true when $\ell =0$ (respectively $\ell=1$) since by hypothesis  $\ds \frac{u_{0}}{u_n}$ tends to $0$ (respectively $\ds \frac{u_{n+1}}{u_n}$ tends to $1$), as $n$ goes to infinity. 

\smallskip
(ii) Assume now that $\ell >1$. Since the sequence $(u_n)_{n \in \N}$  is such that   $\frac{u_{n+1}}{u_n} \to 1$, there exists~$n_0 \in \N$ such that for all $n \geq n_0$, one has~$\ds \big|\frac{u_{n+1}}{u_n}-1\big| \leq \frac{\epsilon}{\ell}\cdotp$ 
Then using that $u_n \to + \infty$, one can define  
$$ m_0= \min_{m \geq n_0} \{u_{m} > \ell u_{n_0}\}$$ 
which clearly satisfies   $\ds \frac{u_{m_0-1}}{u_{n_0}} \leq \ell$. We deduce    that $$\ell  < \frac{u_{m_0}}{u_{n_0}}   < \big(1+ \frac{\epsilon}{\ell}\big) \frac{u_{m_0-1}}{u_{n_0}} < \big(1+ \frac{\epsilon}{\ell}\big) \ell=\ell+\epsilon\virgp$$
that is to say $\ds \big|\frac{u_{m_0}}{u_{n_0}}-\ell\big| \leq \epsilon$, which completes the proof of the claim in that case. 

\smallskip 
(iii) The case when $\ell<1$ can be dealt by inverting  $m$ and $n$ in the proof of the case $\ell >1$. 
\end{proof}

\medbreak

Invoking\refeq{actionsubf} together with\refeq{def:dist}, we readily gather  that, as in the Euclidean case,  
 the regularity of a function implies the decay of its Fourier transform   in~$(\widehat{G},\wh d)$.
 In the next statement we have used the  notation
 $$
  \|u\|_{N,\cS(G)}\eqdefa\sup_{x \in G} (1+d_G(x,0))^N\big |(\mbox{Id} - \Delta_G ^N)u(x)\big|
 $$
for the semi-norms on~$\cS(G)$.
\begin{proposition}
\label {decayFouriercSH}
{\sl
Denoting by $\wh 0$ the point in ~$\wh G_0$  corresponding to $(\dot x=0, \dot y=0, \nu=0, \lam=0)$, for any~$k$ in~$\N$, an integer~$N_k$ and a constant~$C_k$ exist  such that 
$$
(1+\wh d(\wh x,\wh 0))^k  |\cF(u)(\wh x)|  \leq  C_k  \|u\|_{N_k,\cS(G)}\, .
$$
}
\end{proposition}
\begin{proof}
Taking advantage of\refeq{actionsubf}, we get that
$$
E^k_m(\nu, \lambda)  \cF(u)(\wh x)  = \cF\bigl( (-\D_G)^k u \bigr) (\wh x) \, .
$$
Hence, invoking\refeq{embtG}, we infer that 
$$\begin{aligned}
E^k_m(\nu, \lambda)  |\cF(u)(\wh x)| & \leq  \bigl\|(-\D_G)^k u\bigr\|_{L^1(G)}\\ & \leq  C_k \|u\|_{N_k,\cS(G)}.
\end{aligned}$$
Similarly, one has
$$
E^k_n(\nu, \lambda)  \cF(u)(\wh x)  = \cF\bigl((-\wt \D_G)^k u \bigr)(\wh x) \, ,
$$
which implies that 
$$
E^k_n(\nu, \lambda)   |\cF(u)(\wh x)|  \leq  C_k \|u\|_{N_k,\cS(G)}\, .
$$
Finally, using that $$\nu \cF(u)(\wh x) = \cF\big(\big(X_{4}X_{2}-\frac12 X_{3}^{2}\big)u\big) (\wh x) \andf i \lambda \cF(u)(\wh x)= -\cF (X_{4}u) (\wh x)\, ,$$ we end up with the result. 
\end{proof}
\medbreak

\appendix 
\section{Irreducible representations}\label{irrep} 
In this section we briefly summarize the Kirillov theory which permits to compute explicitly the irreducible unitary representations for nilpotent groups and in particular to recover those of the Engel group described in Section~\ref{defsta}. For a comprehensive description we refer the reader to \cite{corwingreenleaf,kirillov}.   
 See also~\cite[Section~2]{Helffer-Nourrigat} for another derivation.

   \subsection{Induced representations} Let $G$ be a nilpotent Lie group and $H$ be a subgroup.
 Given a representation $\mathcal{X}:H\to U(V)$ of $H$ onto the space $U(V)$ of unitary operators on a vector space~$V$ one can define an {\it induced representation} $\mathcal{R}:G\to U(W)$ on a Hilbert space $W$ which we now define. Consider functions $f:G\to V$ such that $\mathcal{X}(h) f=f\circ L_{h}$, where $L_{h}$ denotes the left translation, or 
 \begin{equation}\label{eq:picov} 
  f(hg)=\mathcal{X}(h)f(g) \, , \qquad h\in H \, , g\in G \, .
  \end{equation}
   Notice that for such a function, since $\mathcal{X}(h)$ is unitary, we have that $\|f(hg)\|$ is independent of $h$ and hence the norm of $\|f(Hg)\|$ is well-defined, where $Hg$ denotes the left coset of $g$ in $H\backslash G$. We also require that
  \begin{equation}\label{eq:picov2}
  \int_{H\backslash G} \|f(Hg)\|^{2}d\mu<\infty \, ,
  \end{equation}
  where $d\mu$ is an invariant measure on $H\backslash G$. This means that the function $f$ is in $L^{2}(H\backslash G,d\mu)$.
Then we set
$$W \eqdefa \{f:G\to V\mid f \text{\ satisfies\ \eqref{eq:picov}-\eqref{eq:picov2}} \}.$$
Finally one defines $\mathcal{R}:G\to U(W)$ as follows
   $$\mathcal{R}(g)f \eqdefa f\circ R_{g} \, ,\qquad \text{i.e., }(\mathcal{R}(g)f)(g')=f(g'g) \, ,$$
  where the $R_{g}$ is the right translation. One can check that $\mathcal{R}$ is unitary and strongly continuous.   
   
  \begin{remark} \label{r:a1} {\sl  In order to compute explicitly the induced representation one can use the following observation. Consider the natural projection $\pi: G\to H\backslash G$ of the group onto its quotient. Given any section\footnote{recall that a section is a map $s:H\backslash G\to G$ such that $\pi\circ s=\mathrm{id}_{H\backslash G}$.} $s:H\backslash G\to G$ we can consider its image $K\eqdefa s(H\backslash G)$ and write elements of $G$ as products $H\cdot K$. If $g'g=hk$, where $h\in H$ and $k\in K$ (both depending on $g'g$), we can write
 \begin{equation}\label{eq:me}
 (\mathcal{R}(g)f)(g')=f(g'g)=f(hk)=\mathcal{X}(h)f(k) \, .
 \end{equation}
 In what follows we apply this construction when $\mathcal{X}$ is a character of the group. Thus, in the induced representation, $\mathcal{X}$ represents the exponential part while the component $f(k)$ is a ``shift''.
The crucial step in the computations will be to solve the  equation
 \begin{equation}\label{eq:ggprimokh}
 g'g=h(g'g)k(g'g) \, .
 \end{equation}
Since $f$ satisfies \eqref{eq:picov},  
it is enough to solve \eqref{eq:ggprimokh} for $g'\in K$. (In a compact form, one has to solve~$K\cdot G=H\cdot K$.)}
   \end{remark}
  
  \subsection{Coadjoint orbits and Poisson structure} Given a Lie group $G$ and its Lie algebra $\mathfrak{g}$ one can consider the 
so-called {\it coadjoint action} for $g\in G$
  $$\mathrm{Ad}^{*}_{g}:\mathfrak{g}^{*}\to \mathfrak{g}^{*},\qquad \langle\mathrm{Ad}^{*}_{g}\eta,v\rangle\eqdefa \langle \eta,(\mathrm{Ad}_{g^{-1}})_{*}v\rangle \, ,$$
 where $\mathrm{Ad}_{g}$ is the usual adjoint map. Notice that $\mathrm{Ad}^{*}$ can be seen as an action of $G$ on $\mathfrak{g}^{*}$. Given~$\eta\in \mathfrak{g}^{*}$ the {\it coadjoint orbit} of $\eta$ is by definition the set
   $$\mathcal{O}_{\eta} \eqdefa \{\mathrm{Ad}^{*}_{g}\eta\mid g\in G\} \, .$$
The dual of the Lie algebra $\mathfrak{g}^{*}$ has the natural structure of Poisson manifold with the bracket
$$\{a,b\}(\eta) \eqdefa \langle \eta,[da,db]\rangle \, ,$$
where $a,b:\mathfrak{g}^{*}\to \R$ are smooth functions and $da,db$ are their differentials thought as elements of~$(\mathfrak{g}^{*})^{*}\simeq\mathfrak{g}$ (hence the Lie bracket $[da,db]$ is a well-defined element of $\mathfrak{g}$). Given a smooth function~$a:\mathfrak{g}^{*}\to \R$ we can define its {\it Poisson vector field} by setting for every smooth $b:\mathfrak{g}^{*}\to \R$
$$\vec a (b) \eqdefa \{a,b\} \, .$$
The computation of the coadjoint orbits can be done in a coordinate independent way using the Poisson structure. 
The set of all Poisson vector  at a point defines a distribution
$$D_{\eta} \eqdefa \{\vec a (\eta)\mid a\in C^{\infty}(\mathfrak{g}^{*})\},$$
which does not have in general  constant rank (notice indeed that we always have $D_{0}=\{0\}$). 
We can define also the {\it Poisson orbit} of $\eta\in  \mathfrak{g}^{*}$, in the sense of dynamical systems, as follows
$$\mathcal{O}^{P}_{\eta} \eqdefa \{e^{t_{1}\vec a_{1}}\circ \ldots \circ e^{t_{\ell}\vec a_{\ell}}(\eta)\mid \ell \in \N, t_{i}\in \R, a_{i}\in C^{\infty}(\mathfrak{g}^{*})\} \, .$$
Notice that both $\mathcal{O}^{P}_{\eta}$ and $\mathcal{O}_{\eta}$ are subsets of $\mathfrak{g}^{*}$ containing $\eta$.
\begin{proposition} [\cite{kirillov}] {\sl For every $\eta\in \mathfrak{g}^{*}$ we have the equality $\mathcal{O}^{P}_{\eta}=\mathcal{O}_{\eta}$. Each orbit is an even dimensional symplectic manifold.}
\end{proposition}
\subsection{Computation of coadjoint orbits}\label{s:a3}
To compute explicitly coadjoint orbits on a nilpotent Lie group $G$ one can use the following method (cf.\ for instance \cite[Ch.~18]{ABB19}). Consider a basis of the Lie algebra $X_{1},\ldots,X_{n}$ such that
$$[X_{i},X_{j}]=\sum_{k=1}^{n}c_{ij}^{k}X_{k} \, ,$$
Thanks to the fact that the vector fields are left-invariant, the functions $c_{ij}^{k}$ are constant. Define the corresponding coordinates on the fibers of $T^{*}G$ given by
$h_{i}(p,x)=p\cdot X_{i}(x)$. Notice that $h_{i}$ are functions which are linear on fibers.
These functions, due to left-invariance, can be thought as smooth functions on $\mathfrak{g}^{*}$ and satisfy the relations
$$\{h_{i},h_{j}\}=\sum_{k=1}^{n}c_{ij}^{k}h_{k} \, .$$
We recall that a {\it Casimir} is a smooth function $f\in C^{\infty}(\mathfrak{g}^{*})$ such that 
$$\{a,f\}=0,\qquad \forall\, a\in  C^{\infty}(\mathfrak{g}^{*}) \, .$$
If we consider an arbitrary function $f\in C^{\infty}(\mathfrak{g}^{*})$ as a function of the coordinates just introduced~$f=f(h_{1},\ldots,h_{n})$, then $f$ is a Casimir if and only if 
$\{f,h_{j}\}=0$ for all $j=1,\ldots,n$, 
which means
$$\sum_{i=1}^{n}\frac{\partial f}{\partial h_{i}}c_{ij}^{k}=0,\qquad j,k=1,\ldots,n.$$
With similar computations, the Poisson vector field associated to a function $f$ is given by
\begin{equation}\label{eq:pvf}
\vec f=\sum_{i,j,k=1}^{n}\frac{\partial f}{ \partial h_{i}}c_{ij}^{k}h_{k}\frac{\partial}{\partial h_{j}} \, \cdotp
\end{equation}
We stress that the Poisson vector field associated to a Casimir is the zero vector field. Moreover, for coordinate functions $h_{1},\ldots,h_{n}$ we have
\begin{equation}\label{eq:pvf2}
\vec h_{i}=\sum_{j,k=1}^{n}c_{ij}^{k}h_{k}\frac{\partial}{\partial h_{j}} \, \cdotp
\end{equation}
Clearly, to compute Poisson orbits $\mathcal{O}^{P}_{\eta}$, it is sufficient to consider the flow of the vector fields from the family $\vec{h}_{1},\ldots,\vec{h}_{n}$.

 \subsection{Kirillov theory} The Kirillov theory gives a way to describe all  irreducible unitary representations of $G$ in terms of coadjoint orbits of the group. The Kirillov theorem can be described as the following three-steps algorithm:
 \begin{enumerate}
 \item Fix an element $\eta\in \mathfrak{g}^{*}$ and any maximal (with respect to inclusion) Lie subalgebra $\mathfrak{h}$ of $ \mathfrak{g}$ in such a way that $\eta([\mathfrak{h},\mathfrak{h}])=0$.
\item Consider the one-dimensional representation $\mathcal{X}_{\eta,\mathfrak{h}}:H\to S^{1}=U(\C)$ defined by
$$\mathcal{X}_{\eta,\mathfrak{h}}(e^{X})=e^{i\langle \eta, X\rangle},\qquad X\in \mathfrak{h} \, .$$
where as usual $\langle \eta, X\rangle$ denotes the duality product $\mathfrak{g}^{*}$ and $\mathfrak{g}$.
\item Compute the induced representation $\mathcal{R}_{\eta,\mathfrak{h}}:G\to U(W)$.
 \end{enumerate}
Notice that, due to the previous discussion, the space $W$ of functions $f:G\to V$ satisfying \eqref{eq:picov} and are in $L^{2}(H\backslash G)$ can be identified with $L^{2}(\R^{d})$ with $d=\dim \mathfrak{g}-\dim \mathfrak{h}$.

The Kirillov theorem states that the map which assigns to $\eta\in \mathfrak{g}^{*}/G$ to $\mathcal{R}_{\eta,\mathfrak{h}}$ in $\widehat G$ is a bijection. This is formalized in the following statement.
\begin{theorem} [\cite{kirillov}] {\sl We have the following properties:
\begin{itemize}
\item[(a)] every irreductible unitary representation of a nilpotent Lie group $G$ is of the form $\mathcal{R}_{\eta,\mathfrak{h}}$ for some $\eta\in \mathfrak{g}^{*}/G$ and $\mathfrak{h}$ maximal subalgebra of $\mathfrak{g}$ such that $\eta([\mathfrak{h},\mathfrak{h}])=0$,
\item[(b)] two representations $\mathcal{R}_{\eta,\mathfrak{h}}$ and $\mathcal{R}_{\eta',\mathfrak{h}'}$ are equivalent if and only if $\eta$ and $\eta'$ belong to the same coadjoint orbit.
\end{itemize}}
\end{theorem}
Here two irreductible unitary representations $R_{1}:G\to U(W_{1})$ and $R_{2}:G\to U(W_{2})$ are equivalent if there exists an isometry $T:W_{1}\to W_{2}$ between the corresponding Hilbert spaces  such that
$T\circ R_{1}(g)\circ T^{-1}=R_{2}(g)$ for every $g\in G$.

{\bf Notation.} In what follows we write $\mathcal{R}_{\eta}\eqdefa \mathcal{R}_{\eta,\mathfrak{h}}$ by removing the Lie algebra from the parameters to simplify the notation.

\subsection{The irreducible representations on the Engel group}
Recall that the Engel group is a nilpotent Lie group of dimension 4 with a basis of the Lie algebra satisfying
$$[X_{1},X_{2}]=X_{3},\qquad [X_{1},X_{3}]=X_{4} \, .$$
Following the discussion in Section~\ref{s:a3}, to find a basis of the Poisson vector fields it is enough to compute  $\vec h_{i}$ for every $i=1,2,\ldots,5$. Using formula \eqref{eq:pvf} we have that 
\begin{equation}\label{eq:pvfields}
\vec h_{1}=h_{3}\partial_{h_{2}}+h_{4}\partial_{h_{3}},\qquad \vec h_{2}=-h_{3}\partial_{h_{1}},\qquad \vec{h}_{3}=-h_{4}\partial_{h_{1}},\qquad \vec h_{4}=0 \, .
\end{equation}
Notice that $h_{4}$ is a Casimir since the corresponding vector field $X_{4}$ is in the center of the Lie algebra. The Lie algebra admits a second independent Casimir.
\begin{lemma} {\sl The function $f=\frac12 h_{3}^{2}-h_{2}h_{4}$ is a Casimir. In particular all coadjoint orbits are contained in the level sets $L_{\nu,\lambda}$ defined by
\begin{equation}\label{eq:orbite1}
\begin{cases}
h_{4}=\lambda, \\
\frac12 h_{3}^{2}-\lambda h_{2}=\nu \, .
\end{cases}
\end{equation} }
\end{lemma}
\begin{proof}
This is a consequence of an explicit calculation. Indeed we have $\{f,h_{j}\}=0$ for $j= 2,3,4,$ since $\{h_{i},h_{j}\}(p,x)=p\cdot [X_{i},X_{j}](x)$ which vanishes identically if $i$ and $j$ are both different from $1$. Moreover 
$$\{f,h_{1}\}=\{h_{3},h_{1}\}h_{3}-\{h_{2},h_{1}\}h_{4}=-h_{4}h_{3}+h_{3}h_{4}=0 \, . $$
This proves the lemma.  
\end{proof}
Coadjoint orbits are given by the flow of the Poisson vector fields restricted to the level sets of the Casimirs. One gets the following description.
\begin{proposition} {\sl \label{p:poisson}
In coordinates $(h_{1},h_{2},h_{3},h_{4})$ on $\mathfrak{g}^{*}$, the coadjoint orbits are described as follows:
\begin{itemize}
\item[(i)] if $\lambda=\nu=0$, then every point $(h_{1},h_{2},0,0)$ is an orbit,
\item[(ii)] if $\lambda=0$ and $\nu\neq 0$, then orbits are planes $\{h_{3}=c\}$ for $c\in \R$, \item[(iii)] if $\lambda \neq 0$, then the orbit coincides with the set defined by the equations \eqref{eq:orbite1}.
\end{itemize}}
\end{proposition}
\begin{proof} Case (i) is easy. By assumption $\lambda=\nu=0$, then $h_{3}=h_{4}=0$ by \eqref{eq:orbite1}. Hence coadjoint orbits are contained in the set $L_{0,0}=\{(h_{1},h_{2},0,0) \mid h_{1},h_{2}\in \R\}$ but since all Poisson vector fields vanish on this 2-dimensional set thanks to \eqref{eq:pvfields}, all points in $L_{0,0}$ are orbits.
 
 Case (ii) is similar. By assumption $\lambda=0$, $\nu \neq 0$, then $h_{4}=0$ and $h_{3}\neq 0$ by \eqref{eq:orbite1}. 
 Hence coadjoint orbits are contained in the set $L_{\nu,0}=\{(h_{1},h_{2},h_{3},0) \mid h_{1},h_{2}\in \R, h_{3}\neq 0\}$. When restricted to $L_{\nu,0}$ the only non zero Poisson vector fields are
\begin{equation}\label{eq:pvfields2}
\vec h_{1}=h_{3}\partial_{h_{2}},\qquad \vec h_{2}=-h_{3}\partial_{h_{1}} \, ,
\end{equation}
 so that if $h_{3} \neq 0$ orbits are planes $\{h_{3}=c\}$ for $c\in \R$.

Case (iii). Here $\lambda\neq 0$ hence each orbit is contained in the level set $L_{\nu,\lambda}$ defined by equations~$h_{4}=\lambda$ and $\frac12 h_{3}^{2}-\lambda h_{2}=\nu$ as in \eqref{eq:orbite1}. On the other hand the non zero vector fields \eqref{eq:pvfields} restricted to the level set have the form
\begin{equation}\label{eq:pvfields3}
\vec h_{1}=\nu\partial_{h_{2}}+\lambda \partial_{h_{3}},\qquad \vec h_{2}=-\nu\partial_{h_{1}},\qquad \vec{h}_{3}=-\lambda \partial_{h_{1}} \, ,
\end{equation}
Since $\lambda\neq 0$, it is not difficult to check that the orbit in this case   coincides with the level set itself. 
\end{proof}

Let us now compute all irreducible representations corresponding to the case (iii), i.e., $\lambda \neq 0$. In this case the orbit is the set $L_{\nu,\lambda}$ described by \eqref{eq:orbite1} and
on this set we fix the element~$\eta=(0,-\nu/\lambda,0,\lambda)$. Then we choose the subalgebra
$$\mathfrak{h}=\mathrm{span}\{X_{2},X_{3},X_{4}\},\qquad [\mathfrak{h},\mathfrak{h}]=0 \, .$$
which clearly satisfies $\eta([\mathfrak{h},\mathfrak{h}])=0$ and is maximal with respect to inclusion since $\eta$ is not zero. The corresponding 1-dim representation acts on $H=\exp(\mathfrak{h})$ as follows
$$\mathcal{X}_{\nu,\lambda}(e^{x_{2}X_{2}+x_{3}X_{3}+x_{4}X_{4}})=e^{i(-\nu x_{2}/\lambda +\lambda x_{4})} \, .$$
Let us write points on $G$ as follows
\begin{equation}\label{eq:coordg}
g=e^{x_{2}X_{2}+x_{3}X_{3}+x_{4}X_{4}}e^{x_{1}X_{1}}.
\end{equation}
Following the discussion in   Remark~\ref{r:a1}, we consider  the complement $K=\exp(\R X_{1})$ and we have to solve the equation $K\cdot G=H\cdot K$.  Thanks to Lemma~\ref{l:lemmacb} below (applied in the form~$e^{A}e^{B}=e^{C(A,B)}e^{A}$) we have the identity
\begin{align*}
e^{\theta X_{1}}&e^{x_{2}X_{2}+x_{3}X_{3}+x_{4}X_{4}}e^{x_{1}X_{1}}
=e^{x_{2}X_{2}+(x_{3}+\theta x_{2})X_{3}+(x_{4}+\theta x_{3}+\frac{\theta^{2}}{2}x_{2})X_{4}}e^{(\theta+x_{1})X_{1}} \, .
\end{align*}
We deduce that
$$\mathcal{R}_{\nu,\lambda}f(e^{\theta X_{1}})=\mathcal{X}_{\nu,\lambda}(e^{x_{2}X_{2}+(x_{3}+ \theta x_{2})X_{3}+(x_{4}+\theta x_{3}+\frac{\theta^{2}}{2}x_{2})X_{4}})f(e^{(\theta+x_{1})X_{1}}) \, .$$
Introducing the  notation $\widetilde f(\theta)\eqdefa f(e^{\theta X_{1}})$ we can summarize the above result as follows
\begin{proposition} {\sl All unitary irreducible representations on the Engel group corresponding to coadjoint orbits of case (iii) are parametrized by $\lambda \neq 0$ and $\nu\in \R$, acting on $L^{2}(\R)$ as follows:
\begin{equation}\label{eq:finale}
\mathcal{R}_{\nu,\lambda}\widetilde f(\theta)=\exp\left[i\left(-\frac{\nu}{\lambda} x_{2}+\lambda(x_{4}+\theta x_{3}+\frac{\theta^{2}}{2}x_{2})\right)\right] \widetilde f(\theta+x_{1}) \, .
\end{equation}
where $(x_{1},x_{2},x_{3},x_{4})$ are coordinates on $G$ defined by \eqref{eq:coordg}.}
\end{proposition}
We state here without proof the following algebraic lemma.
\begin{lemma}\label{l:lemmacb} {\sl Assume that the Lie algebra generated by $A,B$ is nilpotent. Then
we have that~$e^{A}e^{B}e^{-A}=e^{C(A,B)}$ with
$$C(A,B)=e^{\mathrm{ad}(A)}B=\sum_{k=0}^{s-1}\frac{\mathrm{ad}^{k}(A)}{k!}B,$$
where $s$ is the nilpotency step of the structure. In particular in the case of the Engel group  we have
$$C(A,B)=B+[A,B]+\frac{1}{2}[A,[A,B]] \, .$$
}
\end{lemma}
\begin{remark}{\sl 
Formula \eqref{eq:finale} gives the representations of the element of the group $G$ parametrized by coordinates $(x_{1},x_{2},x_{3},x_{4})$, where $(0,0,0,0)$ is the origin of the group (which corresponds indeed to the identical representation). 

Hence, differentiating \eqref{eq:finale} with respect to the variables $x_{i}$ at $x=0$, we get also the representations of the element of the Lie algebra, as follows
\begin{align*}
X_{1}\widetilde f&=\frac{d}{d\theta}\widetilde f \, ,\qquad
X_{2}\widetilde f=i\left(-\frac{\nu}{\lambda}+\lambda \frac{\theta^{2}}{2}\right) \widetilde f \, ,\qquad
X_{3}\widetilde f=i\lambda \theta \widetilde f \, , \qquad
X_{4}\widetilde f=i \lambda \widetilde f \, ,
\end{align*}
which indeed satisfy $[X_{1},X_{2}]=X_{3}$ and $[X_{1},X_{3}]=X_{4}$ as differential operators. 
Notice that the Laplacian in this form is written as
$$X_{1}^{2}+X_{2}^{2}=\frac{d^{2}}{d\theta^{2}}-\left( \frac{\lambda}{2}\theta^{2}-\frac{\nu}{\lambda}\right)^{2} \, .$$}
\end{remark}

\begin{remark}{\sl 
Notice that in the explicit computations of Section~\ref{Fourier} only the representations corresponding to the case (iii) of Proposition~\ref{p:poisson} are involved, since in the Fourier trasform the representations are integrated with respect to the the Plancherel measure, which in this coordinates is written as $dP=d\lambda d\nu$. Computing the representations corresponding to the case (i) and (ii) reduces to the representations of the Euclidean plane and the Heisenberg group, respectively. See\ccite{Dixmier,kirillov} for more details on the Plancherel measure and \cite{kirillov2} for an explicit formula on nilpotent Lie groups.}
\end{remark}

\section{Spectral theory  
}\label{ap2}
\subsection{Spectral analysis of the quartic oscillator $P_{\nu,\lambda}$}\label{anspctm} 
We first collect general properties of the operator $P_{\nu,\lambda}$ defined in~\eqref{eq: oscop} for $(\nu, \lambda)\in\R\times \R^*$, and endowed with the domain 
$$
D(P_{\nu,\lambda}) = \left\{ u \in L^2(\R) ,- \frac{d^{2}}{d\theta^{2}} u+ \left( \frac{\lambda}{2}\theta^{2}-\frac{\nu}{\lambda}\right)^{2}u \in L^2(\R) \right\}  \, .
$$
\begin{proposition}
\label{p:def-E-psi}
{\sl For any $(\nu, \lambda)\in\R\times \R^*$, the following statements hold true.
The operator~$(P_{\nu,\lambda},D(P_{\nu,\lambda}))$ is selfadjoint on $L^2(\R)$, with compact resolvents. Its spectrum consists in countably many real eigenvalues, accumulating only at $+\infty$. Moreover, 
\begin{enumerate}
\item \label{eigenvalue-ordering} all eigenvalues are simple and positive, and we may thus write $\Sp(P_{\nu,\lambda}) = \{ E_m(\nu, \lambda), m \in \N\}$ with 
\begin{align*}
& 0< E_0(\nu, \lambda) < E_1(\nu, \lambda) < \cdots < E_m(\nu, \lambda) < E_{m+1}(\nu, \lambda) \to + \infty , \\
& \dim \ker (P_{\nu,\lambda}-E_m(\nu, \lambda)) = 1 \,  ,
\end{align*}
\item \label{regularity} all eigenfunctions are real-analytic and belong to $\mathcal{S}(\R)$,
\item \label{parity} for all $m\in \N$, functions in $\ker (P_{\nu,\lambda}-E_m(\nu, \lambda))$ have the parity of $m$,
\item \label{def-psim} for all $m\in \N$, there is a unique function $\psi_m^{\nu,\lambda}$ in $ \ker (P_{\nu,\lambda}-E_m(\nu, \lambda))$ such that 
$$
 \psi_m^{\nu,\lambda} \text{ is real-valued}, \quad \| \psi_m^{\nu,\lambda}\|_{L^2(\R)} = 1 , \quad \psi_m^{\nu,\lambda}(0)>0 \text{ if } m\text{ is even} ,\quad \frac{d}{d\theta} \psi_m^{\nu,\lambda}(0)>0 \text{ if } m\text{ is odd} \,  , 
$$
\item the family $\big( \psi_m^{\nu,\lambda} \big)_{m\in \N}$ forms a Hilbert basis of $L^2(\R)$.
\end{enumerate}}
\end{proposition}
This proposition serves as a definition for the eigenvalue $E_m(\nu, \lambda)$ and the associated eigenfunction $\psi_m^{\nu,\lambda}$ for $m\in \N$. Note that for $\psi_m^{\nu,\lambda}$, we made a particular choice.

\begin{proof}
If $\psi \in D(P_{\nu,\lambda})$, the inner product of $P_{\nu,\lambda}\psi$ with $\psi$ implies in particular that $\psi \in H^1(\R)$ and $\left( \frac{\lambda}{2}\theta^{2}-\frac{\nu}{\lambda}\right)\psi \in L^2(\R)$, whence the compactness of the embedding $D(P_{\nu,\lambda}) \hookrightarrow L^2(\R)$ and that of the resolvent of $P_{\nu,\lambda}$.
The structure of the spectrum is a direct consequence of the first stated facts.
 Then we notice that the coefficients of $P_{\nu,\lambda}$ are real and one may thus choose real-valued eigenfunctions.
The fact that the eigenvalues are simple follows from the classical Sturm-Liouville argument, see e.g.~\cite{Reed-Simon-4}. The latter also yields that any real-valued eigenfunction $\psi$ associated to~$E_m$ has exactly $m$ zeroes.

The property $(-\infty ,0] \cap \Sp(P_{\nu,\lambda}) =\emptyset$ follows from the fact that $P_{\nu,\lambda}\psi = E \psi$ for $\psi \in D(P_{\nu,\lambda})) \setminus \{0\}$ implies 
$$
0 \leq \|\psi'\|_{L^2(\R)}^2 +  \left\|  \left( \frac{\lambda}{2}\theta^{2}-\frac{\nu}{\lambda}\right) \psi \right\|_{L^2(\R)}^2 = E \|\psi\|_{L^2(\R)}  \, .
$$
Hence, $E\geq 0$. If $E=0$, then the left hand-side yields $\psi'=0$ in $\mathcal{D}'(\R)$, thus $\psi =0$ (since $\psi \in L^2(\R)$), which is a contradiction.

In Item~\ref{regularity}, real-analyticity of the eigenfunctions follows from the analytic Cauchy-Lipschitz theorem. That eigenfunctions belong to $\mathcal{S}(\R)$ follows from Agmon estimates, see~\cite{HS:84,Helffer:booksemiclassic,DS:book}.
Item~\ref{parity} is a consequence of the fact that $ \left( \frac{\lambda}{2}\theta^{2}-\frac{\nu}{\lambda}\right)^2$ is even.  Hence, if $\psi_m$ is an eigenfunction associated to $E_m$, then $x \mapsto \psi_m(-x)$ is also an eigenfunction. Simplicity of the spectrum implies that $x \mapsto \psi_m(-x)$ is proportional to $\psi_m$. Since we choose $\psi_m$ real-valued and $L^2$-normalized, we necessarily have $\psi_m(-x) = \pm \psi_m(x)$. That $\psi_m$ has the parity of $m$ follows from the fact that $\psi_m$ has $m$ zeroes.

Concerning Item~\ref{def-psim}, since $\dim \ker (P_{\nu,\lambda}-E_m(\nu, \lambda)) = 1$, there are only two normalized eigenvalues, say $\psi$ and $-\psi$. In case $m$ is even (resp. odd), these eigenvalues are even (resp. odd) from Item~\ref{parity} and hence one has $\psi(0)\neq 0$ (resp. $\psi'(0)\neq 0$), and we choose among $\pm \psi$ the one having positive value  (resp. positive derivative) at zero.

Finally, the last item is a consequence of the spectral theorem for compact selfadjoint operators.
\end{proof}

We now explain how the study of the two parameter family of operators $P_{\lambda,\nu}$ reduces to that of $\mathsf{P}_\mu$.
We start with the following scaling argument, referring to the scaling operator $T_\alpha$ defined in~\eqref{e:Talpha}.
\begin{lemma}[Scaling]
\label{l:scaling}
{\sl For all $\alpha>0$ and $(\nu,\lambda)\in\R \times \R^*$, the operators $\alpha^2 P_{\nu,\lambda}$ and $P_{\alpha^4\nu,\alpha^3\lambda}$  are unitarily equivalent: we have 
\begin{align*}
  P_{\alpha^4\nu,\alpha^3\lambda} =\alpha^2T_\alpha P_{\nu,\lambda} T_{\alpha^{-1}}  \, .
\end{align*}
 In particular, we have for all  $\alpha>0$ and $(\nu,\lambda)\in\R \times \R^*$, and all $m\in \N$,
\begin{align}
\label{eq: hom}
 E_m(\alpha^4  \nu,\alpha^3  \lam)&= \alpha^2 E_m(\nu,\lam) \,  , \\ 
\label{e:relbis}
 \psi_m^{\alpha^4 \nu,\alpha^3 \lambda}(\theta) & = T_{\alpha} \psi_m^{\nu,\lambda}(\theta)   \, .
\end{align}}
\end{lemma}
This scaling property will later allow us to get rid of one of the two parameters. 
Note that the last property can also be written, if needed: for all $a>0$, we have 
\begin{equation}
\label{e:rel}
a^{1/4}\psi_k^{a^{-2}\nu, a^{-3/2}\lambda}(a^{1/2}\theta)= \psi_k^{\nu, \lambda}(\theta) \, .
\end{equation}

\begin{proof}[Proof of Lemma~{\rm\ref{l:scaling}}]
The first statement simply follows from the following computation:
\begin{align*}
T_\alpha P_{\nu, \lambda} T_{\alpha^{-1}} &= -\frac{1}{\alpha^2}\frac{d^2}{d\theta^2}  + \left(\lambda\alpha^2 \frac{\theta^2}{2} - \frac{\nu}{\lambda}\right)^2 
 = \frac{1}{\alpha^2}\left[  -\frac{d^2}{d\theta^2}  + \left(\lambda\alpha^3 \frac{\theta^2}{2} - \alpha\frac{\nu}{\lambda}\right)^2 \right]\\
& = \frac{1}{\alpha^2}\left[  -\frac{d^2}{d\theta^2}  + \left((\lambda\alpha^3) \frac{\theta^2}{2} - \frac{(\nu\alpha^4)}{(\lambda\alpha^3)}\right)^2 \right]
 = \frac{1}{\alpha^2} P_{\alpha^4\nu, \alpha^3\lambda} \,  .
\end{align*}
Concerning the second statement, we deduce from the first one that 
$$
P_{\alpha^4\nu, \alpha^3\lambda} T_{\alpha} \psi_m^{\nu,\lambda} = \alpha^2T_\alpha P_{\nu,\lambda} T_{\alpha^{-1}}T_{\alpha} \psi_m^{\nu,\lambda}  =
 \alpha^2 T_\alpha P_{\nu,\lambda}   \psi_m^{\nu,\lambda}  = \alpha^2 E_m(\nu,\lam)  T_\alpha    \psi_m^{\nu,\lambda}   \, .
$$
Hence, $T_{\alpha} \psi_m^{\nu,\lambda}$ is an eigenfunction associated to the eigenvalue $\alpha^2 E_m(\nu,\lam)$. From Item~\ref{eigenvalue-ordering} of Proposition~\ref{p:def-E-psi}, we deduce that $\alpha^2 E_m(\nu,\lam)$ is the $m$-th eigenvalue of $P_{\alpha^4\nu, \alpha^3\lambda}$, whence~\eqref{eq: hom}.
From the uniqueness of the eigenvalue in Item~\ref{eigenvalue-ordering} of Proposition~\ref{p:def-E-psi} and the fact that $\psi \mapsto T_\alpha\psi$ preserves the sign of $\psi(0)$ and $\psi'(0)$, we deduce~\eqref{e:relbis}.
\end{proof}

We also notice that $P_{\nu,\lambda}=P_{\nu,-\lambda}=P_{\nu,|\lambda|}$.
Now, we choose a particular value of $\alpha$ with so that to reduce to a one-parameter problem, namely $\alpha = |\lambda|^{-1/3}>0$.

\begin{definition}[Reference operator]
{\sl For $\mu \in \R$, and $m\in \N$, we set 
\begin{align}
\label{e:def-sfPEphi}
\mathsf{P}_\mu \eqdefa  P_{\mu,1} =  - \frac{d^2}{d\theta^2}  + \left( \frac{\theta^2}{2} - \mu \right)^2, 
\qquad \mathsf{E}_m(\mu) \eqdefa  E_m(\mu, 1) \, ,  
\qquad  \varphi_m^\mu \eqdefa   \psi_m^{\mu,1} \, .
\end{align}}
\end{definition}
Note that Proposition~\ref{p:def-E-psi} applies to $\mathsf{P}_\mu,\mathsf{E}_m(\mu),\varphi_m^\mu$ and we use it implicitly. In particular, $\mathsf{E}_m(\mu)$ is the $m$-th eigenvalue of $\mathsf{P}_\mu$ and $\varphi_m^\mu$ is the (with the appropriate choice) associated eigenfunction.

According to Lemma~\ref{l:scaling} taken for $\alpha = |\lambda|^{-1/3}>0$, we have the following statement.
\begin{corollary}[Scaling and reference operator]
\label{c:scaling}
{\sl For all $(\nu,\lambda)\in\R \times \R^*$, and all $m\in \N$, we have 
\begin{align}
P_{\nu,\lambda} & =  |\lambda|^{2/3}  T_{|\lambda|^{1/3}}  \mathsf{P}_\mu T_{|\lambda|^{-1/3}} , \quad  \mu =  \frac{\nu}{|\lambda|^{4/3}} \in \R \,  ,   \\
\label{eq: relnumu} E_m(\nu,\lambda) &= |\lambda|^{2/3}\mathsf{E}_m(\mu) , \quad  \mu =  \frac{\nu}{|\lambda|^{4/3}} \in \R ,   \\
\psi_m^{\nu,\lambda} & = T_{|\lambda|^{1/3}} \varphi_m^\mu , \quad  \mu =  \frac{\nu}{|\lambda|^{4/3}}\in \R  \,  .
\end{align}}
\end{corollary}
As a consequence, we are left with the study of the family of operators $\mathsf{P}_\mu$, depending on a single parameter $\mu \in \R$.

\subsection{Spectral theory for semiclassical Schr\"odinger operators} 

In this section, we collect several results of spectral theory, that are used in the main part of the paper to study the operator~$\mathsf{P}_\mu$ (or equivalently $P_{\nu, \lambda}$).

We refer e.g. to \cite[Section~6.4]{Zworski:book} for the following very classical Weyl law.
\begin{theorem}[Weyl's law in dimension $1$]
\label{t:semiclassical-weyl}
{\sl Assume that $V \in C^\infty(\R;\R)$ is real valued and satisfies~$|\p^\alpha V(\x)| \leq \langle \x \rangle^{k}$ for all $\alpha$ and all $\x \in \R$, and $V(\x) \geq c \langle \x \rangle^{k}$ for $|\x| \geq R>0$. Then, for all~$h>0$, the operator 
\begin{align}
\label{e:Ph-def}
P(h) = -h^2 \frac{d^2}{d\x^2} + V(\x), \quad D(P(h)) = \Big\{u \in L^2(\R),  -h^2 \frac{d^2}{d\x^2}u + Vu \in L^2(\R)\Big\} \, , 
\end{align}
acting on $L^2(\R)$ is selfadjoint, has compact resolvent, has discrete real spectrum, and an orthonormal basis of eigenfunctions.
Moreover, for any $a<b$, 
\begin{align}
\label{e:weyl}
\sharp \left(\Sp(P(h)) \cap [a,b] \right) = (2\pi h)^{-1} \Big( \Vol \{(\x,\xi) \in\R^2,  a \leq \xi^2 + V(\x) \leq b \} + o(1) \Big)\red  \, ,
\end{align}
as $h \to 0^+$.}
\end{theorem}
Note that the phase space volume (taken according to the symplectic volume form $d\x\, d\xi$) is given by 
$$
\Vol \{(\x,\xi) \in\R^2,  a \leq \xi^2 + V(\x) \leq b \} = \int_{\{ a \leq \xi^2 + V(\x) \leq b \}} d\x\, d\xi  \, .
$$
In the $1$-dimensional context, it can often be computed more simply, see e.g. Remark~\ref{r:homogeneity-L} below.

We shall also make use of the following lemma, which is a simple consequence of the minimax and maximin formulae (see~\cite[Chapter~11 and discussion top of p148]{Helffer:book-spectral-theory}).
\begin{lemma}
\label{t:comparison}
{\sl Let $H$ be a Hilbert space. 
Assume that $(A,D(A))$ and $(B,D(B))$ are two selfadjoint operators, with compact resolvents, that are bounded from below and such that $D(B)\subset D(A)$.
 Denote for $j \in \N$ by $E_j(A)$ (resp. $E_j(B)$) the $j-$th eigenvalue of the operator $A$ (resp. $B$), defined by the minimax formula, so that in particular 
 $E_0(A) \leq E_1(A) \leq \cdots \leq E_j(A) \leq E_{j+1}(A) \leq \cdots \to + \infty$.
 
Assume further that $(A u ,u)_H \leq (B u , u)_H$ for all $u$ in a dense set of $D(B)$. Then we have 
$$
 E_j(A) \leq  E_j(B), \quad \text{ for all } j \in \N  \, . 
$$}
\end{lemma}

We now consider the operator 
$$P(h) = -h^2 \frac{d^2}{d\x^2} + \left(\frac{\x^2}{2}-\eps(h)\right)^2, \quad D(P(h)) = \Big\{u \in L^2(\R),  -h^2 \frac{d^2}{d\x^2}u +  \left(\frac{\x^2}{2}-\eps(h)\right)^2u \in L^2(\R)\Big\}. $$
Note that 
\begin{align}
\label{e:dvt-Ph}
P(h) = -h^2 \frac{d^2}{d\x^2} + \frac{\x^4}{4}-\eps(h)\x^2 + \eps(h)^2 \,  .
\end{align}
Since $-\eps(h)\x^2 + \eps(h)^2$ is a relatively compact perturbation of $-h^2 \frac{d^2}{d\x^2}+  \frac{\x^4}{4}$, we notice that 
$$D(P(h)) = \Big\{u \in L^2(\R),  -h^2 \frac{d^2}{d\x^2}u + \frac{\x^4}{4} u \in L^2(\R)\Big\}   $$ 
does not depend on $\eps(h)$.

\begin{proposition}
\label{p:semiclass-perturb}
{\sl For any $L>0$, there are two continuous nondecreasing functions $\Gamma_\pm : \R_+\to \R_+$ such that $\Gamma_\pm(\eps_0)>0$ for $\eps_0 >0$ and $\Gamma_\pm(0)=0$ satisfying the following statement. For all~$\eps_0 >0$ and all $|\eps(h)| \leq \eps_0$, we have 
$$
 \Vol_L- \Gamma_-(\eps_0)+ o(1) 
 \leq (2\pi h) \sharp \left( \Sp(P(h)) \cap [0,L]\right) 
\leq  \Vol_L + \Gamma_+(\eps_0)+ o(1) 
$$
as $h \to 0^+$, where
$$
\Vol_L \eqdefa  \Vol \Big\{(\x,\xi) \in\R \times \R ,  \xi^2 + \frac{\x^4}{4} \leq L \Big\} = \int_{\{ \xi^2 + \frac{\x^4}{4} \leq L \}} d\x \, d\xi  .
$$ In particular, if $\eps(h) \to 0$ as $h \to 0^+$, we have 
$$
 (2\pi h) \sharp \left( \Sp(P(h)) \cap [0,L]\right) 
=  \Vol_L + o(1) \quad \text{ as }h \to 0^+  \, .
$$}
\end{proposition}
\begin{remark}
\label{r:homogeneity-L}{\sl 
Notice that we can take advantage of the homogeneity of the symbol $\xi^2 + \frac{\x^4}{4}$ to prove that $\Vol_L = L^{3/4} \Vol_1$.
Indeed, we have explicitly
\begin{align*}
\Vol_L 
& = \int_{x_-(L)}^{x_+(L)} \sqrt{L -  \frac{x^4}{4} } dx , \quad \text{ where }\quad \frac{x_\pm(L)^4}{4} = L , \quad \pm x_\pm(L)>0 \\
& = \int_{-(4L)^{1/4}}^{(4L)^{1/4}} \sqrt{L -  \frac{x^4}{4} } dx , \quad \text{ and thus, setting }y \eqdefa  x/L^{1/4} ,\\
& =  \int_{-4^{1/4}}^{4^{1/4}} \sqrt{L -  L \frac{y^4}{4} } L^{1/4}dy 
=  L^{3/4}\int_{-4^{1/4}}^{4^{1/4}} \sqrt{1 - \frac{y^4}{4} } dy  
=L^{3/4} \Vol_1  \, .
\end{align*}}
\end{remark}
The proof of the proposition relies on a comparison argument using Theorem~\ref{t:semiclassical-weyl} and Lemma~\ref{t:comparison}.
\begin{proof}[Proof of Proposition~{\rm\ref{p:semiclass-perturb}}]
For $|\eps(h)|\leq \eps_0$, we define 
\begin{align*}
P^-(h) & = -h^2 \frac{d^2}{d\x^2} + \frac{\x^4}{4} -\eps_0 \x^2, \\
P^+(h) & = -h^2 \frac{d^2}{d\x^2} + \frac{\x^4}{4} + \eps_0 \x^2 +\eps_0^2 = -h^2 \frac{d^2}{d\x^2} + \left( \frac{\x^2}{2} + \eps_0\right)^2 \,  ,
\end{align*}
with respective domains defined as in~\eqref{e:Ph-def}. According to the same remarks as above, we have~$D(P^\pm(h))=D(P(h))$.
According to~\eqref{e:dvt-Ph}, we further notice that 
$$
\left( P^-(h) u ,  u \right)_{L^2(\R)} \leq \left( P(h) u ,  u \right)_{L^2(\R)} \leq \left( P^+(h) u ,  u \right)_{L^2(\R)} \quad \text{ for all } u \in \mathcal{S}(\R)  \, ,
$$ 
where $\mathcal{S}(\R)$ is dense in $D(P(h))$.

For $j \in \N$, we now denote by $E_j^\pm$ (resp.~$E_j$) the $j-$th eigenvalue of the operator~$P^\pm(h)$~(resp.~$P(h)$), defined by the minimax formula. Lemma~\ref{t:comparison} yields for all $j \in \N$ and $h>0$
$$
E_j^- \leq E_j \leq E_j^+ .
$$ 
As a consequence, for any $L,h >0$,
$$
\sharp \left\{j \in \N,  E_j^+ \leq L \right\} \leq \sharp \left( \Sp(P(h)) \cap [0,L]\right) =\sharp \left\{j \in \N,  E_j \leq L \right\} \leq \sharp \left\{j \in \N,  E_j^- \leq L \right\}  \, .
$$
Theorem~\ref{t:semiclassical-weyl} then implies that for any $\eps_0, L >0$ we have in the limit $h \to 0^+$;
$$
  \int_{\{ \xi^2 +  \left( \frac{\x^2}{2} + \eps_0\right)^2 \leq L \}} d\x\, d\xi+ o(1) 
 \leq (2\pi h) \sharp \left( \Sp(P(h)) \cap [0,L]\right) 
\leq  \int_{\{ \xi^2 + \frac{\x^4}{4} -\eps_0 \x^2\leq L \}} d\x\, d\xi+ o(1)  \, . 
$$
The sought result follows by taking $\Gamma(\eps_0) =\max \{\Gamma_+(\eps_0),\Gamma_-(\eps_0)\}$ with 
\begin{align*}
\Gamma_+(\eps_0)&\eqdefa  \int_{\{ \xi^2 + \frac{\x^4}{4} -\eps_0 \x^2\leq L \}} d\x\, d\xi - \int_{\{ \xi^2 +  \frac{\x^4}{4}  \leq L \}} d\x\, d\xi  \, , \\
\Gamma_-(\eps_0)&\eqdefa  \int_{\{ \xi^2 +  \frac{\x^4}{4}  \leq L \}} d\x\, d\xi -\int_{\{ \xi^2 +  \left( \frac{\x^2}{2} + \eps_0\right)^2 \leq L \}} d\x\, d\xi  \, .
\end{align*}
and noticing that $\Gamma$ has the desired properties. \end{proof}

\end{document}